\documentclass[12pt]{article}
\usepackage{amsthm}
\usepackage{mathtools}
\usepackage{amssymb,amsmath,graphicx,
	amsfonts,euscript}
\usepackage{color}
\usepackage{cite}
\allowdisplaybreaks

\usepackage{hyperref}
\numberwithin{equation}{section}
\DeclareMathOperator{\Op}{Op}

\newcommand{\be}{\begin{eqnarray}}
\newcommand{\ee}{\end{eqnarray}}
\newcommand{\bes}{\begin{eqnarray*}}
\newcommand{\ees}{\end{eqnarray*}}

\usepackage{amsthm}

\theoremstyle{definition} 

\newtheorem{corollary}{Corollary}[section]
\newtheorem{definition}[corollary]{Definition}

\newtheorem{lemma}[corollary]{Lemma}
\newtheorem{proposition}[corollary]{Proposition}
\newtheorem{remark}[corollary]{Remark}
\newtheorem{theorem}[corollary]{Theorem}

\newcommand{\beq}{\begin{equation}}
\newcommand{\eeq}{\end{equation}}
\newcommand{\ben}{\begin{eqnarray}}
\newcommand{\een}{\end{eqnarray}}
\newcommand{\beno}{\begin{eqnarray*}}
\newcommand{\eeno}{\end{eqnarray*}}
\usepackage{mathrsfs}
\usepackage{enumerate}
\usepackage{geometry}
\usepackage{lastpage}
\usepackage{fancyhdr}
\title{\large\bf{ Lower bound of the energy norm for type II blow-up solutions to the energy-critical wave equation in non-euclidean geometries} }

\author{
   Siwar Ben Said\\
Laboratoire Analyse, G{\'e}om{\'e}trie et Applications\\
  Université Sorbonne Paris Nord\\
   99 Avenue Jean Baptiste Cl{\'e}ment \\93430 Villetaneuse, France\\
  siwar.bensaid@math.univ-paris13.fr
   }

\begin{document}
\maketitle
\begin{abstract}
This article is concerned with the study of the focusing energy-critical wave equation on the three-dimensional Riemannian manifold  $\mathcal{M}=\mathbb{R}^3$ equipped with a smooth time-independent Riemannian metric $g$, which coincides with the Euclidean metric outside a compact set. Our focus is on Type II blow-up solutions, namely finite-time solutions that remain bounded in the energy space, with the goal of revealing their fundamental properties. We establish in particular lower bounds for the energy norm as the solution approaches its maximal time of existence in terms of the bounds of  the ground state of the Euclidean problem. 
\end{abstract}
\noindent\textbf{MSC-2020:}  35xx, 35Lxx, 35Bxx.
\tableofcontents
\section{Introduction and main result} In this article, we study the nonlinear wave equation on the three-dimensional Riemannian manifold $\mathcal{M}=\mathbb{R}^3$, given by 
\begin{equation} \label{eq1A2}
\left\{
\begin{array}{l}
\partial^2_tu-\Delta_gu=u^5,\quad\quad\quad \text{in}\quad\mathbb{R}\times \mathcal{M}, \ \\[10pt]
u_{|t=0}= u_0\in \dot{H}^1(\mathcal{M}) , \quad \partial_tu_{|t=0}= u_1 \in L^2(\mathcal{M}),
\end{array}
\right.
\end{equation}
 where $u$ is real-valued and $\Delta_g$ denotes the Laplace-Beltrami operator associated with a smooth, time-independent Riemannian metric $g(x)=(g_{i,j}(x))_{1\leq i,j\leq 3}$ such that $g_{i,j}(x)=\delta_{i,j}$ for $|x|>R$, for some radius $R>0$. 
 
 Throughout the paper, solutions to \eqref{eq1A2} will refer to solutions in the sense of the following definition, based on Duhamel’s formulation of the equation.
\begin{definition} [Finite energy solution]
\label{defA2}
A finite energy solution of the system \eqref{eq1A2} on a time interval $I$ with $0\in I$ and initial conditions $(u_0,u_1)\in\dot{H}^1(\mathcal{M})\times L^2(\mathcal{M})$, is a function $u\in L^5_{\mathrm{loc}}(I,L^{10}(\mathcal{M}))$ such that for all $t\in I$
\begin{align*}
u(t)=S_L(t)(u_0,u_1)+\int_0^t\frac{\mathrm{sin}\big((t-s)\sqrt{-\Delta_g}\big)}{\sqrt{-\Delta_g}}u^5(s) ds,
\end{align*}
where \begin{align} \label{S_LlinearA2}
S_L(t)(u_0,u_1):=\mathrm{cos}(t\sqrt{-\Delta_g})u_0 +\frac{\mathrm{sin}(t\sqrt{-\Delta_g})}{\sqrt{-\Delta_g}}u_1.
\end{align}
\end{definition}
In Definition \ref{defA2}, by $u\in L^5_{\mathrm{loc}}(I,L^{10}(\mathcal{M}))$, we mean that $u\in L^5(J,L^{10}(\mathcal{M}))$ for any compact interval $J\subset I$.

We define the conserved energy of the solution to \eqref{eq1A2} by
\begin{align*}
    \mathcal{E}(u(t),\partial_tu(t))= \mathcal{E}(u_0,u_1)=\frac{1}{2}\|u_0\|_{\dot{H}^1(\mathcal{M})}^2+\frac{1}{2}\|u_1\|_{L^2(\mathcal{M})}^2-\frac{1}{6}\|u_0\|^6_{L^6(\mathcal{M})}.
\end{align*}
It is known that for any $(u_0,u_1)\in \dot{H}^1(\mathcal{M})\times L^2(\mathcal{M})$, there exists a unique local solution of \eqref{eq1A2} in the sense of Definition \ref{defA2}. By classical arguments, one constructs the corresponding maximal solution. For the sake of completeness, we give the details in section \ref{section2.1A2}.

Let $T_+(u)\in (0, +\infty]$ be the maximal positive time of existence for the solution $u$ to the equation \eqref{eq1A2}. Then the following blow-up criterion holds.
\begin{align} \label{blowupcriterionA2}
\text{If}\quad T_+(u)<+\infty \quad \text{then} \quad u\notin L^5([0,T_+(u)),\,L^{10}(\mathcal{M})).
\end{align}
That is, if $$\|u\|_{L^5([0,T_+(u)),\,L^{10}(\mathcal{M}))} <+\infty,\quad\quad \text{then}\quad T_+(u) = +\infty,$$
and, in this case, the solution scatters forward in time in $\dot{H}^1(\mathcal{M})\times L^2(\mathcal{M})$, i.e. there exists a linear solution $v$ such that
\begin{align*}
\underset{t\rightarrow +\infty}{\mathrm{lim}}\Big(\|u(t)-v(t)\|_{\dot{H}^1(\mathcal{M})}+\|\partial_tu(t)-\partial_tv(t)\|_{L^2(\mathcal{M})}\Big)=0.
\end{align*}
Here, unlike the case of lower order nonlinearity (of the form $|u|^{p-1}u$ with $p<5$), it is not possible to improve the blow-up criterion to
\begin{align} 
T_+(u)<+\infty \Longrightarrow \underset{t\rightarrow T_+(u)}{\mathrm{lim}\; \mathrm{sup}}\;\|(u,\partial_tu)\|_{\dot{H}^1(\mathcal{M})\times L^2(\mathcal{M})}=+\infty.
\end{align}
Indeed, in the three-dimensional space with a Euclidean metric $g=I_3$, Krieger, Schlag and Tataru \cite{8A2} established the existence of radial solutions of type II blow-up, i.e., radial solutions such that
\begin{align}\label{typetwoblowupA2}
 T_+(u)<+\infty \quad \text{and} \quad \underset{t\in [0, T_+(u))}{\mathrm{sup}}\;\|(u,\partial_tu)\|_{\dot{H}^1(\mathcal{M})\times L^2(\mathcal{M})}<+\infty.
\end{align}
It is known from the works of
Aubin \cite{3A2} and Talenti \cite{4A2} that the stationary solution to \eqref{eq1A2}, endowed with the Euclidean metric, namely, the solution of the nonlinear elliptic equation
\begin{align}
-\Delta W=W^5,\label{ellipticA2}
\end{align}
given by
\begin{align}\label{expofwA2}
W:=\frac{1}{(1+\frac{|x|^2}{3})^\frac{1}{2}}
\end{align}
is the unique minimizer, up to translation,
scaling and multiplication by a scalar constant, for the Sobolev inequality
\begin{align}\label{sobolevA2}
\|u\|_{L^6(\mathbb{R}^3)}\leq C_{I_3}\|\nabla u\|_{L^2(\mathbb{R}^3)} \quad \text{for} \;\; u \in \dot{H}^1(\mathbb{R}^3).
\end{align}
That is, if $$\|u\|_{L^6(\mathbb{R}^3)}= C_{I_3}\|\nabla u\|_{L^2(\mathbb{R}^3)}\quad \text{and}\quad u\neq 0,$$ then  $$u(x)=\mu W(\lambda (x+X)),$$
where $\mu \in \mathbb{R}\setminus \{0\}$, $\lambda >0$, $X\in \mathbb{R}^3$ and $C_{I_3}$ is the best Sobolev constant.
Note that equation \eqref{ellipticA2} implies $$E(W, 0)= \frac{1}{3}\|\nabla W\|^2_{L^2(\mathbb{R}^3)},$$
where 
\begin{align*}
  E(W, \partial_tW)  := \frac{1}{2}\|W\|_{\dot{H}^1(\mathbb{R}^3)}^2+\frac{1}{2}\|\partial_tW\|_{L^2(\mathbb{R}^3)}^2-\frac{1}{6}\|W\|^6_{L^6(\mathbb{R}^3)}.
\end{align*}
The work of Krieger, Schlag, and Tataru \cite{8A2} shows that, in the Euclidean setting, problem \eqref{eq1A2} admits Type II blow-up solutions of the form
 \begin{align*}
     u(t,x)= \frac{1}{\lambda(t)^\frac{1}{2}} W(\frac{x}{\lambda(t)})+\eta(t,x), \quad \lambda(t)=(-t)^\alpha,\;\alpha >\frac{3}{2}, \quad t\in [-t_0,0),\; t_0> 0,
 \end{align*}
where 
\begin{align*}
   \mathcal{E}_{\mathrm{loc}}(\eta):=\int_{\{|x|<t\}} \big((\partial_t\eta)^2+|\nabla\eta|^2+|\eta|^6\big) \; dx \longrightarrow 0, \quad \text{as}\;\; t\rightarrow 0.
\end{align*}
In the Euclidean setting, for $3\leq d\leq 5$, Kenig and Merle \cite{9A2}, followed by Duyckaerts, Kenig, and Merle \cite{18A2}, established universal properties of Type II blow-up solutions. Motivated by these results, we study their extension to the non-Euclidean setting.

Before stating our main result, we introduce the following definition.
\begin{definition} [Regular/Singular point]\label{defregulA2}
A point $x_0 \in \mathcal{M}$ is a regular point if
for all  $ \varepsilon>0$ there exists $R>0$ such that for any $t\in [0,T_+(u))$, one has
\begin{align} \label{3.111Ths}
\int_{\mathscr{B}_R(x_0)}\Big(|\nabla_gu|^2_g+|u|^2+|\partial_tu|^2 \Big)d\mathrm{vol}<\varepsilon,
\end{align}
where $\mathscr{B}_R(x_0)$ denotes the geodesic ball of radius $R$ centered at $x_0\in \mathcal{M}$, as defined in Section \ref{notA2}. If not, we call $x_0$ 
a singular point. We denote by $S$ the set of singular
points.
\end{definition}
\begin{remark}
    Note that the norm $\|u\|_{\dot{H}^1(\mathscr{B}_R(x_0))}$, defined by 
    \begin{align*}
    \|u\|_{\dot H^1(\mathscr{B}_R(x_0))} = \inf_{\bar{u}} \|\bar{u}\|_{\dot H^1(\mathcal{M})}, 
    \end{align*}
    where $\bar{u}$ ranges over all extensions of $u$ belonging to $\dot{H}^1(\mathcal{M})$, is equivalent to 
    \begin{align*}
   \Big(\int_{\mathscr{B}_R(x_0)} \big(|\nabla_g u|^2_g + |u|^2\big)\,d\mathrm{vol}\Big)^{\frac{1}{2}}. 
     \end{align*}
    Indeed, this follows easily from the equivalence of the Euclidean and Riemannian norms and Hardy's inequality: 
    \begin{align*}
    \int_{\mathscr{B}_R(x_0)} |u|^2 \,d\mathrm{vol}& \lesssim \int_{B_R^{Euc}(x_0)} |u|^2 \; dx\lesssim R^2 \int_{\mathbb{R}^3} \frac{|u|^2}{|x|^2} \; dx\lesssim R^2 \int_{\mathbb{R}^3} |\nabla \bar{u}|^2 \;dx \\&\lesssim R^2 \int_{\mathbb{R}^3} |\nabla_g \bar{u}|^2_g \,d\mathrm{vol},
    \end{align*}
    where $B_R^{Euc}(x_0)$ denotes the Euclidean ball of radius $R$ centered at $x_0 \in \mathcal{M}=\mathbb{R}^3$.\\
 The inclusion of the $|u|^2$ term in \eqref{3.111Ths} is essential: if $u\equiv 1$ on $\mathscr{B}_R(x_0)$, then $\nabla_g u=0$, while $\|u\|_{\dot H^1(\mathscr{B}_R(x_0))}\neq 0$, as otherwise $u$ would have to vanish.
\end{remark}

Our main result is the following.
\begin{theorem} \label{maintheoremA2}
We suppose $u$ is a solution to \eqref{eq1A2} with type II blow-up, i.e., such that \eqref{typetwoblowupA2} is satisfied. Then the following results hold.
\begin{enumerate}[(1).]
 \item The set $S$ is finite. We denote  
    \begin{align} \label{singularfiniteA2}
    S=\{m_1,...,m_K\},
    \end{align}
  where  $K\in \mathbb{N}^*$ and $m_1,...,m_K$ are distinct points in $\mathcal{M}$.
\item There exists
$(v_0,v_1) \in \dot{H}^1(\mathcal{M})\times L^2(\mathcal{M})$ such that 
\begin{align} \label{3.1**A2}
(u,\partial_tu)\xrightharpoonup[t \to T_+(u)]{}(v_0,v_1) \quad \text{weakly in} \quad  \dot{H}^1(\mathcal{M})\times L^2(\mathcal{M}).
\end{align}
\item
If $\varphi \in C^\infty_c(\mathcal{M})$ is equal to $1$ around each singular point $m_k\in S$, then
\begin{align}\label{3.2**A2}
\underset{t\rightarrow T_+(u)}{\mathrm{lim}} \Big(\|(1-\varphi)(u(t)-v_0)\|_{\dot{H}^1(\mathcal{M})} +\|(1-\varphi)(\partial_tu(t)-v_1)\|_{L^2(\mathcal{M})}\Big)=0.
\end{align}
\item For all $k\in \{1,...,K\}$,
\begin{align} \label{1.25***A2}
\underset{t\rightarrow T_+(u)}{\mathrm{lim \;sup}} \int_{\mathscr{B}_{|t-T_+(u)|}(m_k)}\Big(|\nabla_g u|^2_g+ |\partial_tu|^2\Big) \;d\mathrm{vol}\geq \int_{\mathbb{R}^3} |\nabla W|^2\;dx,
\end{align}
\begin{align}\label{1.26***A2}
\underset{t\rightarrow T_+(u)}{\mathrm{lim \;inf}} \int_{\mathscr{B}_{|t-T_+(u)|}(m_k)}\Big(|\nabla_g u|^2_g+ |\partial_tu|^2\Big) \;d\mathrm{vol}\geq \frac{2}{3}\int_{\mathbb{R}^3} |\nabla W|^2\;dx.
\end{align}
\end{enumerate}
\end{theorem}
An analogous statement holds for negative times. We focus on space dimension $3$ for simplicity, but the result remains valid in dimensions 4 and 5, and even on compact manifolds without boundary, by working locally in coordinate charts. The choice of $\mathcal{M}=\mathbb{R}^3$ here allows us to work with a single global chart. 

These results also raise several open questions. A first, and particularly challenging, problem is to rigorously prove the existence of Type II blow-up solutions in the setting considered here. Although their existence is assumed in our analysis, an explicit construction or a complete existence proof would significantly advance the understanding of energy concentration phenomena on Riemannian manifolds. Another important question is the extension of these results to manifolds with boundary, with different boundary conditions, including Dirichlet, Neumann, and mixed conditions, and the investigation of their influence on the dynamics of blow-up solutions. Recent work of Lafontaine and Laurent \cite{1666A2}, which deals with the case of a flat metric outside a particular obstacle, may provide useful tools and insights for addressing this question.
\begin{remark}
\begin{enumerate}
\item Note that there exists at least one singular point. Otherwise, by \eqref{3.2**A2}, one has
 $$(u,\partial_tu)\underset{t\rightarrow T_+(u)}{\longrightarrow}(v_0,v_1) \quad\text{ in} \quad  \dot{H}^1(\mathcal{M})\times L^2(\mathcal{M}),$$  which contradicts the fact that $T_+(u)$ is the maximal time of existence.
 \item Among the statements of Theorem \ref{maintheoremA2}, assertion (4) is the most difficult to prove, while (1), (2), and (3) follow from the finite speed of propagation and the Cauchy theory. 
 
 We also note that (4) generalizes the result of Kenig-Merle \cite{9A2}, stated in Corollary 7.5. Since blow-up occurs at a point, the solution focuses on an increasingly small region of space where the geometry becomes negligible. Consequently, the leading-order behavior is described by the flat Euclidean problem. This is revealed in the estimates \eqref{1.25***A2} and \eqref{1.26***A2}, which are expressed in terms of the Euclidean ground state norm $\|\nabla W\|_{L^2(\mathbb{R}^3)}^2$ and show the concentration of at least $\frac{2}{3} \|\nabla W\|_{L^2(\mathbb{R}^3)}^2$
 of the energy near the blow-up point (or the singular point). This subsequently enabled us, in \cite{ben} (see also Section 3.3 in \cite{ben1}), to establish an analogue of the scattering part of the Kenig-Merle result. A key ingredient is the analysis of the optimal constant in the Sobolev inequality in the same setting considered in this article. More precisely, we prove in \cite{ben} that either the optimal constant $C_g$ coincides with the Euclidean constant $C_{I_3}$, or $C_g>C_{I_3}$, in which case the Sobolev inequality is attained by a nontrivial solution of the associated critical elliptic equation.\\
 The proof of assertion (4) relies on the study of linear and nonlinear concentrating waves, using techniques related to those developed by Laurent \cite{4A2} and Ibrahim \cite{10A2} in the defocusing case, where global existence holds for all solutions and the analysis is considerably simpler. The main challenge of the present work lies in the focusing nature of the equation.
 \end{enumerate}
\end{remark}

\textbf{Outline of the paper:}
The paper is organized as follows. In Section \ref{section2A2}, we establish several preliminary results needed for the proof of the main theorem. We begin in Section \ref{notA2} by introducing the notation used throughout the paper. In Section \ref{section2.1A2}, we discuss the well-posedness of the Cauchy problem associated with the equation \eqref{eq1A2}. In Section \ref{section2.2A2}, we prove the finite speed of propagation property for solutions to the nonlinear wave equation \eqref{eq1A2}. Section \ref{section3.1A2} is devoted to the proof of the first three points of the main theorem, corresponding to (\ref{singularfiniteA2}), (\ref{3.1**A2}) and (\ref{3.2**A2}). 
In Section \ref{section3.2A2}, we examine the behavior of both linear and nonlinear concentrating waves. The resulting analysis serves to establish, in Section \ref{section55A2}, preliminary materials needed for the proof of assertion (4) of our main theorem. 
\section{Preliminaries} \label{section2A2}
This section is devoted to establishing several standard results necessary for the proof of Theorem \ref{maintheoremA2}.
\subsection{Notations}\label{notA2}
For the sake of clarity, we introduce and recall some standard notations used throughout the paper.

We denote by $d\mathrm{vol}:=(\mathrm{det}(g))^{\frac{1}{2}}dx$ the canonical positive Riemannian density on $\mathcal{M}$. For $1\leq p<\infty$, the $L^2$-inner product and $L^p$-norm, taken with respect to this density, are defined by
\begin{align*}
(u,v)_{L^2(\mathcal{M})}=\int_{\mathcal{M}} u\bar{v}\;d\mathrm{vol}, \quad \quad \|u\|_{L^p(\mathcal{M})}^p=\int_{\mathcal{M}}|u|^p\;d\mathrm{vol}.
\end{align*} 
We also define the space $L^2V(\mathcal{M})$ of $L^2$-vector fields on $\mathcal{M}$, equipped with the norm
\begin{align*}
\|u\|_{L^2V(\mathcal{M})}^2=\int_{\mathcal{M}}g(u,\bar{u}) \;d\mathrm{vol}, \quad\quad u\in L^2V(\mathcal{M}).
\end{align*}
For a function $f$ and a vector field $u$, the Riemannian gradient, the divergence and the unbounded Laplace-Beltrami operator $\Delta_g$ are given, respectively, by
\begin{align*}
(\nabla_gf)^i=\sum_{1\leq j\leq 3}g^{i,j}\partial_{x_j}f,\quad \quad \mathrm{div}_g(u)=(\mathrm{det}(g))^{-\frac{1}{2}}\sum_{1\leq i\leq 3}\partial_{x_i}\big((\text{det}\; g)^{\frac{1}{2}}u^i),
\end{align*}
\begin{align*}
\Delta_g f=(\text{det}\; g)^{-\frac{1}{2}}\sum_{1\leq i,j\leq 3}\partial_{x_i}\big((\text{det}\; g)^{\frac{1}{2}}g^{i,j}(x)\partial_{x_j}f\big),
\end{align*}
where $(g^{i,j})=(g_{i,j})^{-1}$. Throughout the paper, we will frequently use the notation 
\begin{align*}
 \|u\|_{\dot{H}^1(\mathcal{M})}^2:=  \|\nabla_g u\|_{L^2V(\mathcal{M})}^2 = \int_{\mathcal{M}} |\nabla_g u|^2_g \; d\mathrm{vol}, \quad \text{with} \quad |\nabla_g u|^2_g = g(\nabla_g u, \nabla_g u).
\end{align*}

Due to the continuity of the metric $g$ in $\left\{x\in\mathcal{M},\;|x|\leq R\right\}$ and its Euclidean behavior outside this compact set, there exists a constant $0<c<1$ such that 
\begin{align}\label{1.222222eqA2}
c|\xi|^2\leq g^{i,j}(x)\xi_i \xi_j\leq c^{-1}|\xi|^2, \quad \quad \forall x,\xi \in \mathbb{R}^3.
\end{align}
As a result, the first order Sobolev norm and $L^p$-norms $(1\leq p<\infty)$ on $(\mathcal{M},g)$ are equivalent to their Euclidean counterparts:
\begin{align} \label{1.3A2}
\|u\|_{\dot{H}^1(\mathbb{R}^3)}\simeq \|u\|_{\dot{H}^1(\mathcal{M})}\quad \text{and}\quad \|u\|_{L^p(\mathbb{R}^3)}\simeq \|u\|_{L^p(\mathcal{M})}.
\end{align}

We denote by $\mathscr{B}_R(x_0)$ and ${\bf{S}}_R(x_0)$ the geodesic ball and the geodesic sphere centered at $x_0 \in \mathcal{M}$ of radius $R>0$, given by
\begin{align*}
&\mathscr{B}_R(x_0)=\{x\in \mathcal{M} : d_g(x,x_0)\leq R \},
\\&{\bf{S}}_R(x_0)= \{x\in \mathcal{M} : d_g(x,x_0)=R \},
\end{align*} 
where $d_g$ is the Riemannian distance induced by $g$ and defined for two points
$x_0, x_1 \in \mathcal{M}$ by
\begin{align*}
    d_g(x_0,x_1)=\inf_{\gamma} \int_0^1 \sqrt{g_{\gamma(t)}\big(\dot{\gamma}(t), \dot{\gamma}(t)\big)}\, dt,
\end{align*}
where the infimum is taken over all piecewise smooth curves
$\gamma : [0,1] \to \mathcal{M}$ such that
\begin{align*}
\gamma(0)=x_0 \quad \text{and} \quad \gamma(1)=x_1.
\end{align*}
We also define the exponential map
\begin{align*}
\mathrm{exp}_{x_0}:T_{x_0}\mathcal{M} \rightarrow \mathcal{M}, v\mapsto \gamma_v(1),
\end{align*}
 where $\gamma_v: [0,1]\rightarrow \mathcal{M}$ is the unique geodesic satisfying
\begin{align*}
\gamma_v(0)=x_0, \quad \dot{\gamma}_v(0)=v.
\end{align*} 
\subsection{Cauchy problem} \label{section2.1A2}
Here, we examine the well-posedness of the Cauchy problem \eqref{eq1A2}. We recall first the following Strichartz estimates.
\begin{theorem} [Strichartz estimates \cite{7A2}]\label{thm1A2}
Let $I$ be a finite time interval containing $0$. There exists a constant $C_I$ such that
for all $(p,q)$, satisfying the admissibility condition
\begin{align}\label{addmiA2}
\frac{1}{p}+\frac{3}{q}=\frac{1}{2} \quad \text{and} \quad  2< p,q\leq \infty,
\end{align}
and for all $(u_0,u_1)\in \dot{H}^1(\mathcal{M})\times L^2(\mathcal{M})$, $F\in L^1(I,L^2(\mathcal{M}))$, the associated solution $u$ of
\begin{align} 
\begin{cases}
\partial^2_tu-\Delta_gu=F,\quad\quad\quad\quad &\text{in}\quad I\times \mathcal{M}, \ \\[10pt]
\big(u, \partial_tu)_{|t=0}= (u_0, u_1), &\text{in} \quad \mathcal{M},
\end{cases}
\end{align}
satisfies the following estimate 
\begin{align}\label{strichA2}
\frac{1}{C_I}\|u\|_{L^p(I,L^q(\mathcal{M}))}+\underset{t\in I}{\mathrm{sup}}\Big(\|u\|_{\dot{H}^1(\mathcal{M})}&+\|\partial_tu\|_{L^2(\mathcal{M})}\Big)\nonumber
\\&\leq \|u_0\|_{\dot{H}^1(\mathcal{M})}+\|u_1\|_{L^2(\mathcal{M})}+\|F\|_{L^1(I,L^2(\mathcal{M}))}.
\end{align}
\end{theorem}
Solutions of the system \eqref{eq1A2}, as defined in Definition \ref{defA2}, are given by the following  theorem.
\begin{theorem} \label{th2A2}
 Let $I$ be a finite time interval containing $0$. There exists $\delta_0=\delta_0(I)> 0$ with the following property. Let $(u_0,u_1)\in \dot{H}^1(\mathcal{M})\times L^2(\mathcal{M})$. Assume that
\begin{align*}
\|S_L(t)(u_0,u_1)\|_{L^5(I,L^{10}(\mathcal{M}))}=\delta \leq \delta_0.
\end{align*}
Then \eqref{eq1A2} admits a unique solution $u$ on $I$ satisfying
\begin{align}
u\in L^5(I,L^{10}(\mathcal{M})) \quad \text{and}\quad\overset{\rightarrow}{u}=(u,\partial_tu)\in C(I,\dot{H}^1(\mathcal{M})\times L^2(\mathcal{M})).
\end{align}
\end{theorem}
Theorem \ref{th2A2} has two important consequences as pointed out in the following remark.
\begin{remark}  \label{remarlemA2}
\begin{enumerate}[(1).]
\item \textbf{Local well-posedness:} Note that $(5,10)$ is an admissible couple that satisfies \eqref{addmiA2}. Moreover, given $(u_0,u_1)\in \dot{H}^1(\mathcal{M})\times L^2(\mathcal{M})$, one has $S_L(t)(u_0,u_1) \in L^5(\mathbb{R},L^{10}(\mathcal{M}))$.
 Therefore, for sufficiently small $T>0$, there exists $\delta_0=\delta_0(T)>0$ such that 
\begin{align*}
\|S_L(t)(u_0,u_1)\|_{ L^5([-T,T],L^{10}(\mathcal{M}))}\leq \delta_0,
\end{align*}
and Theorem \ref{th2A2} implies that there exists a solution $u$ to \eqref{eq1A2} on $[-T,T]$.
\item \textbf{Small data "global" well-posedness:} Let $I$ be a compact time interval containing $0$. Assume that $(u_0,u_1)\in \dot{H}^1(\mathcal{M})\times L^2(\mathcal{M})$ satisfies $$\|u_0\|_{\dot{H}^1(\mathcal{M})}+\|u_1\|_{L^2(\mathcal{M})}\leq \frac{\delta_0}{C_I},$$ where $C_I$ is the constant in
the Strichartz inequality \eqref{strichA2}. Then, by Theorem \ref{thm1A2}, it follows that
\begin{align*}
\|S_L(t)(u_0,u_1)\|_{L^5(I,L^{10}(\mathcal{M}))}\leq \delta_0.
\end{align*}
Hence, applying Theorem \ref{th2A2}, we deduce the existence of a solution $u$ to \eqref{eq1A2} on $I$ such that $u\in L^5(I,L^{10}(\mathcal{M}))$ and $\overset{\rightarrow}{u}\in C(I,\dot{H}^1(\mathcal{M})\times L^2(\mathcal{M}))$.
\end{enumerate}
\end{remark}
\begin{proof}[Proof of Theorem \ref{th2A2}]
Set 
\begin{align*}
X=\{v\; \text{on} \; I\times\mathcal{M} ,\;\|v\|_{L^5(I,L^{10}(\mathcal{M}))}\leq 2\delta_0\},
\end{align*}
where $\delta_0$ is a small positive constant to be fixed at the end of the proof and consider the map $\phi$ defined on $X$ by
\begin{align}\label{phiiA2}
\phi(v)=S_L(t)(u_0,u_1)+\phi_{Inh}(v),
\end{align}
where 
\begin{align*}
   S_L(t)(u_0,u_1):= \mathrm{cos}(t\sqrt{-\Delta_g})u_0 +\frac{\mathrm{sin}(t\sqrt{-\Delta_g})}{\sqrt{-\Delta_g}}u_1,
\end{align*}
and 
\begin{align*}
    \phi_{Inh}(v):=\int_0^t\frac{\mathrm{sin}((t-s)\sqrt{-\Delta_g})}{\sqrt{-\Delta_g}}v^5(s) ds.
\end{align*}
Applying Theorem \ref{thm1A2} with $(p,q)=(5,10)$, we get
\begin{align*}
\|\phi_{Inh}(v)\|_{L^5(I,L^{10}(\mathcal{M}))}&\leq C_I \|v^5\|_{L^1(I,L^2(\mathcal{M}))}=C_I\|v\|^5_{L^5(I,L^{10}(\mathcal{M}))}
\leq C_I 2^5 \delta_0^5 \leq \delta_0,
\end{align*}
where in the last inequality we assume $\delta_0\leq (2^5 C_I)^{-\frac{1}{4}}$.
Using this estimate along with the assumption of Theorem \ref{th2A2}, yields $$\|\phi(v)\|_{L^5(I,L^{10}(\mathcal{M}))}\leq 2\delta_0,$$
and so $\phi:X\rightarrow X$.
Next, we prove that $\phi$ is a contraction map on $X$.
Let $v_1,v_2 \in X$. Due to Theorem \ref{thm1A2}, one has
\begin{align} \label{2.44A2}
\|\phi(v_1)-\phi(v_2)\|_{L^5(I,L^{10}(\mathcal{M}))}&=\|\phi_{Inh}(v_1)-\phi_{Inh} (v_2)\|_{L^5(I,L^{10}(\mathcal{M}))}\nonumber\\&\leq C_I\|v_1^5-v_2^5\|_{L^1(I,L^2(\mathcal{M}))}.
\end{align}
By Young's inequality, one writes
\begin{align*}
|v_1^5-v_2^5|\leq \frac{5}{2} |v_1-v_2||v_1^4+v_2^4|.
\end{align*}
Then, using Hölder's inequality, we get
\begin{align} \label{2.55A2}
\|v_1^5-v_2^5\|_{L^1(I,L^2(\mathcal{M}))}& \leq \frac{5}{2}\|v_1-v_2\|_{L^5(I,L^{10}(\mathcal{M}))}\Big(\|v_1\|^4_{L^5(I,L^{10}(\mathcal{M}))}+\|v_2\|^4_{L^5(I,L^{10}(\mathcal{M}))}\Big)\nonumber
\\& \leq \frac{5}{2}2 (2\delta_0)^4\|v_1-v_2\|_{L^5(I,L^{10}(\mathcal{M}))}.
\end{align}
Inserting \eqref{2.55A2} in \eqref{2.44A2}, we deduce
\begin{align*}
   \|\phi(v_1)-\phi(v_2)\|_{L^5(I,L^{10}(\mathcal{M}))} \leq  5C_I (2\delta_0)^4\|v_1-v_2\|_{L^5(I,L^{10}(\mathcal{M}))}\leq \frac{1}{2}\|v_1-v_2\|_{L^5(I,L^{10}(\mathcal{M}))},
\end{align*}
where we choose $\delta_0$ so that $\delta_0 < (2^5 5C_I)^{-\frac{1}{4}} $.
By the contraction mapping theorem, we conclude the existence of a unique solution $u$ to \eqref{eq1A2} such that $u\in L^5(I,L^{10}(\mathcal{M}))$. Furthermore, using Theorem \ref{thm1A2}, we have 
\begin{align*}
\overset{\rightarrow}{u}=(u,\partial_tu)\in C(I,\dot{H}^1(\mathcal{M})\times L^2(\mathcal{M})).
\end{align*}
\end{proof}
The subsequent remark will play a crucial role in the proof of the finite speed of propagation for the system \eqref{eq1A2} in Section \ref{section2.2A2}.
\begin{remark}\label{ReA2}
Note that if $u$ is a solution to \eqref{eq1A2} on $[0,T]$, such that 
\begin{align*}
\|u\|_{L^5([0,T],L^{10}(\mathcal{M}))}\leq \frac{\delta_0}{2},
\end{align*}
then, by Theorem \ref{thm1A2}, one has
\begin{align*}
\|S_L(t)(u_0,u_1)\|_{L^5([0,T],L^{10}(\mathcal{M}))}&\leq \|u\|_{L^5([0,T],L^{10}(\mathcal{M}))}+C_{[0,T]}\|u\|^5_{L^5([0,T],L^{10}(\mathcal{M}))}
\\&\leq \frac{\delta_0}{2}+C_{[0,T]}\frac{\delta_0^5}{32}\leq \delta_0([0,T]).
\end{align*}
Therefore, $(u_0,u_1)$ satisfies the assumption of Theorem \ref{th2A2} and as a consequence $u$ is the fixed point of the contraction map $\phi$ defined in \eqref{phiiA2}. Hence,  
\begin{align*}
u=\underset{n\rightarrow +\infty}{\mathrm{lim}} u^{(n)}, \quad \text{in} \quad L^5([0,T],L^{10}(\mathcal{M})),
\end{align*}
where $u^{(0)}=0$ and for all $n\in \mathbb{N}^*$, $u^{(n)}$ is given by
\begin{align*}
u^{(n)}:=S_L(t)(u_0,u_1)+\int_0^t \frac{\mathrm{sin}((t-s)\sqrt{-\Delta_g})}{\sqrt{-\Delta_g}}{(u^{(n-1)})}^5(s) ds.
\end{align*}
\end{remark}
We recall the proposition below, establishing that smoother initial data lead to solutions with higher regularity.
\begin{proposition}
\label{pro11A2}
If the initial data $(u_0,u_1)$ is smoother, that is
$$(u_0,u_1)\in \dot{H}^{s+1}(\mathcal{M})\times \dot{H}^{s}(\mathcal{M}),\quad\quad 0< s\leq 2,$$
then the solution $u$ to \eqref{eq1A2} is also smoother:
$$\overset{\rightarrow}{u}=(u,\partial_tu)\in C([0,T],\dot{H}^{s+1}(\mathcal{M})\times \dot{H}^{s}(\mathcal{M})).$$
\end{proposition}
For a proof of Proposition \ref{pro11A2}, we refer the reader to Kapitanskii \cite{88A2}.
\subsection{Finite speed of propagation} \label{section2.2A2}
In this section, we demonstrate the finite speed of propagation property on the smooth, complete Riemannian manifold $(\mathcal{M}=\mathbb{R}^3,g)$, where the metric $g(x)=(g_{i,j}(x))_{1\leq i,j\leq 3}$ is such that $g_{i,j}(x)=\delta_{i,j}$ for $|x|>R$, for some $R>0$. The following proposition establishes the result for solutions $u$ to the linear wave equation with $u\in C^2$, while Proposition \ref{pro2.122A2} below concerns solutions of the same equation with initial data $(u_0,u_1)\in \dot{H}^1(\mathcal{M})\times L^2(\mathcal{M})$. The case of the nonlinear wave equation \eqref{eq1A2} is treated in Theorem \ref{finiteA2}.
\begin{proposition}
\label{pro1A2} 
Let $\mathscr{B}_R(x_0)$ be the geodesic ball of radius $R$ centered at $x_0\in \mathcal{M}$. Let $T>0$. We denote  
\begin{align*}
\Gamma &:=\{(t,x)\in  \mathbb{R}\times\mathcal{M}, \quad 0\leq t\leq T, \quad \text{and} \quad x\in \mathscr{B}_{R-|t|}(x_0)\}.
\end{align*}
Assume that $u\in C^2\big(\Gamma \big)$ and $u$ solves the linear wave equation 
\begin{align}\label{eq1.6A2}
\partial^2_tu(t,x)-\Delta_gu(t,x)=0,
\end{align}
on $\Gamma$ with initial conditions $u_{|t=0}=0$ and $\partial_tu_{|t=0}=0$ on $\mathscr{B}_{R}(x_0)$. Then $u= 0$ on $\Gamma$.
\end{proposition}
The proof of Proposition \ref{pro1A2} is based on the following lemma. For $p\in \mathcal{M}$, we denote by $B_R(0_p):=\{\xi \in T_p\mathcal{M}:  \|\xi\|<R\}$ the tangential ball of radius $R$ in the tangent space $T_p\mathcal{M}$, where $\|.\|$ is the norm in $T_p\mathcal{M}$ with respect to the Riemannian metric $g$.
\begin{lemma}
\label{lem2A2}
Let $\mathcal{M}$ be a Riemannian manifold,  $p\in \mathcal{M}$ and $\epsilon>0$. Suppose that the exponential map $\mathrm{exp}_p$ is a diffeomorphism from the tangent ball $B_\epsilon(0_p)$ onto the geodesic ball $\mathscr{B}_\epsilon(p)$. Then, for any function $v\in L^1(\mathscr{B}_\epsilon(p))$ and any $s<\epsilon$, the following holds
\begin{align*}
\int_{\mathscr{B}_s(p)}v\; d\mathrm{vol}=\int_0^s\int_{{\bf{S}}_t(p)}v\; d\mathrm{vol}_{{\bf{S}}_t(p)} dt,
\end{align*}
where ${\bf{S}}_t(p)$ denotes the geodesic
sphere of radius $t$ centred at $p$.
\end{lemma}
The proof of Lemma \ref{lem2A2} can be found in \cite{2A2} (see Lemma 1.6).
\begin{proof} [Proof of Proposition \ref{pro1A2}]
Let $x_0\in \mathcal{M}$ and $R>0$. Since $(\mathcal{M}=\mathbb{R}^3,g)$ is complete, the closed ball $\mathscr{B}_R(x_0)$ is compact. Hence, there exists $\epsilon>0$ such that for all $p \in \mathscr{B}_R(x_0)$, $\mathrm{exp}_p$ is a diffeomorphism from $B_\epsilon(0_p)$ onto $\mathscr{B}_\epsilon(p)$. Let $r<\epsilon$ small enough. 
Denote 
\begin{align*}
    C_r(p):=\{(t,x)\in  \mathbb{R}\times\mathcal{M}, \quad 0\leq t\leq r, \quad \text{and} \quad x\in \mathscr{B}_{r-|t|}(p)\}.
\end{align*}
We assume that $u\in C^2\big(C_r(p)\big)$,
$u$ solves \eqref{eq1.6A2} and $u_{|t=0}=\partial_tu_{|t=0}=0$ on $\mathscr{B}_r(p)$.
Applying Lemma \ref{lem2A2} with $s=r-t<\epsilon$, we express the local energy as 
\begin{align*}
E(t)&:=\frac{1}{2}\int_{\mathscr{B}_{r-t}(p)}(|\partial_tu|^2+|\nabla_gu|^2_g)\; d\mathrm{vol}\\&=\frac{1}{2}\int_0^{r-t}\int_{{\bf{S}}_\tau(p)}(|\partial_tu|^2+|\nabla_gu|^2_g)\;d\mathrm{vol} _{{\bf{S}}_\tau(p)} d\tau
=\int_0^{r-t} F(\tau,t) d\tau,
\end{align*}
where 
\begin{align*}
F(\tau,t):=\frac{1}{2}\int_{{\bf{S}}_\tau(p)}(|\partial_tu|^2+|\nabla_gu|^2_g)\;d\mathrm{vol} _{{\bf{S}}_\tau(p)} .
\end{align*}
Since $u\in C^2\big(C_r(p)\big)$, by differentiation under the integral sign one has
\begin{align}
\frac{d}{dt} E(t) &=\frac{d}{dt} \int_0^{r-t} F(\tau,t)d\tau \nonumber
\\&= -F(r-t,t)+\int_0^{r-t}\frac{\partial F}{\partial t}(\tau,t) d\tau  \nonumber
\\&=-F(r-t,t)+\int_0^{r-t}\int_{{\bf{S}}_\tau(p)}\Big(\partial^2_tu\;\partial_tu + g(\nabla_gu,\nabla_g\partial_tu)\Big)\;d\mathrm{vol} _{{\bf{S}}_\tau(p)}d\tau. 
\end{align}
Then, using the equation \eqref{eq1.6A2}, we get
\begin{align*}
\frac{d}{dt} E(t)=& -F(r-t,t)+\int_0^{r-t}\int_{{\bf{S}}_\tau(p)}\Big(\partial^2_tu\;\partial_tu + g(\nabla_gu,\nabla_g\partial_tu)\Big)\;d\mathrm{vol} _{{\bf{S}}_\tau(p)}d\tau.
\\&=
-F(r-t,t)+\int_0^{r-t}\int_{{\bf{S}}_\tau(p)}\Big(\Delta_gu\;\partial_tu +g(\nabla_gu,\nabla_g\partial_tu)\Big)\;d\mathrm{vol}_{{\bf{S}}_\tau(p)}d\tau.
\end{align*}
Recalling that $\mathrm{div}_g(\partial_tu\nabla_gu)=\partial_tu .\Delta_gu+g(\nabla_g\partial_tu,\nabla_gu)$, we obtain
\begin{align*}
\frac{d}{dt} E(t)=-F(r-t,t)+\int_0^{r-t}\int_{{\bf{S}}_\tau(p)}\mathrm{div}_g(\partial_tu\nabla_gu)\;d\mathrm{vol}_{{\bf{S}}_\tau(p)}d\tau.
\end{align*}
Using the fact that $u\in C^2\big(C_r(p)\big)$, one has $\mathrm{div}_g(\partial_tu\nabla_gu)\in L^1(\mathscr{B}_s(p))$. Then, by Lemma \ref{lem2A2}, we get
\begin{align*}
\frac{d}{dt} E(t)&=-F(r-t,t)+\int_{\mathscr{B}_s(p)}\mathrm{div}_g(\partial_tu\nabla_gu)\;d\mathrm{vol}.
\end{align*}
Applying Stokes' formula yields
\begin{align*}
\frac{d}{dt} E(t)&=-F(r-t,t)+\int_{{\bf{S}}_s(p)}\partial_tu\;\partial_nu \;d\mathrm{vol}_{{\bf{S}}_s(p)},
\end{align*}
where $\partial_nu=g(\nabla_gu,n)$  is the outward normal derivative of the function $u$ with respect to the metric $g$.
Then, it follows that
\begin{align*}
\frac{d}{dt} E(t) &\leq -F(r-t,t)+\frac{1}{2}\int_{{\bf{S}}_s(p)}(|\partial_tu|^2+|\partial_nu|^2) \;d\mathrm{vol}_{{\bf{S}}_s(p)} 
\\& \leq \underbrace{-F(r-t,t)+\frac{1}{2}\int_{{\bf{S}}_s(p)}(|\partial_tu|^2+|\nabla_gu|^2_g)\;d\mathrm{vol}_{{{\bf{S}}_s(p)} }}_{=-F(r-t,t)+F(r-t,t) =0}.
\end{align*}
Given that $u_{|t=0}=\partial_tu_{|t=0}=0$ on $\mathscr{B}_r(p)$, one has $E(0)=0$. As a consequence, we deduce that $E(t)=0$ for all $0\leq t\leq r$, which in turn implies $u=0$ on $C_r(p)$ for all $p\in \mathscr{B}_R(x_0)$ and $r<\epsilon$ small enough. Thus, $u=0$ on $$\{(t,x)\in \mathbb{R}\times\mathcal{M},\;0\leq t< \epsilon \;\;\text{and}\;\; x\in \mathscr{B}_{R-|t|}(x_0)\}. $$
One has $u(\epsilon-s,x)=\partial_tu(\epsilon-s,x)=0$ for $x\in \mathscr{B}_{R-s}(x_0)$ and any $s>0$. The $\epsilon$ above works for $\mathscr{B}_{R-\epsilon-s}(x_0)$, and thus we obtain by iteration $u=0$ on $$\{(x,t)\in \mathcal{M}\times\mathbb{R},\;0\leq t< 2\epsilon \;\;\text{and}\;\; x\in \mathscr{B}_{R-|t|}(x_0)\}.$$ We iterate $k$ times until $k\epsilon\geq T$.
Hence, we conclude that $u=0$ on $\Gamma.$
\end{proof}
In the following proposition, we establish the finite speed of propagation for linear solutions in the energy space $\dot{H}^1(\mathcal{M})\times L^2(\mathcal{M})$.
\begin{proposition}\label{pro2.122A2}
Let $\mathscr{B}_R(x_0)$ denote the geodesic ball of radius $R$ centered at $x_0\in \mathcal{M}$, and let $\Gamma$ be defined as in Proposition \ref{pro1A2}.
Assume $u$ is a solution to the linear wave equation 
\begin{align}\label{1.7*A2}
\partial^2_tu(t,x)-\Delta_gu(t,x)=0, \quad (t,x)\in [0,T]\times \mathcal{M},
\end{align} 
with initial data $(u_{|t=0},\partial_tu_{|t=0})=(u_0,u_1)\in \dot{H}^1(\mathcal{M})\times L^2(\mathcal{M})$ such that $u_0=u_1=0$ on $\mathscr{B}_R(x_0)$. Then, $u= 0$ for almost every $(t,x) \in \Gamma$.
\end{proposition}
\begin{proof}[Proof]
Let $\rho \in C^\infty_c(\mathcal{M}) $ be a smooth function with compact support, satisfying $\int_{\mathbb{R}^3}\rho(x)dx=1$, and supported in $\{x\in \mathcal{M}, |x|<1\}$. For $n\in \mathbb{N}^*$, we define $\rho_n(x)=n^3\rho(nx)$ and we introduce the functions $u_{0,n} \in C^\infty(\mathcal{M})$ and $u_{1,n} \in C^\infty(\mathcal{M})$ as the convolution of $\rho_n$ with the functions $u_0$ and $u_1$, respectively, that is
\begin{align}
&u_{0,n}(x)=\rho_n\ast u_0=\int_{\mathbb{R}^3}\rho_n(y)u_0(x-y)dy,\label{con 1.8A2}
\\&
 u_{1,n}(x)=\rho_n\ast u_1=\int_{\mathbb{R}^3}\rho_n(y)u_1(x-y)dy.\label{con 1.9A2}
 \end{align}
 Since, by hypothesis, $u_0=u_1=0$ on $\mathscr{B}_R(x_0)$, it follows that $u_{0,n}(x)=u_{1,n}(x)=0$ for all $x\in \mathscr{B}_{R-\frac{1}{n}}(x_0)$. By Proposition \ref{pro1A2}, the solution $u_n$  to \eqref{1.7*A2} associated with the initial data $(u_{0,n}(x),u_{1,n}(x))$ satisfies $u_n=0$ for all $0\leq t \leq T$ and $x\in \mathscr{B}_{R-\frac{1}{n}-|t|}(x_0)$. \\
 On the other hand, by \eqref{1.3A2}, \eqref{con 1.8A2} and \eqref{con 1.9A2}, we obtain $$\underset{n\rightarrow +\infty}{\mathrm{lim}}\|(u_{0,n}-u_0,u_{1,n}-u_1)\|_{\dot{H}^1(\mathcal{M})\times L^2(\mathcal{M})}=0,$$ which in turn implies that for all $t\in [0,T]$,  $$\underset{n\rightarrow +\infty}{\mathrm{lim}}\|(u_n-u,\partial_tu_n-\partial_tu)\|_{\dot{H}^1(\mathcal{M})\times L^2(\mathcal{M})}=0.$$ Thus, we conclude that $u=0$ for almost every $(t,x) \in \Gamma$. 
\end{proof}
In the following theorem, we establish the finite speed of propagation for solutions to the nonlinear wave equation \eqref{eq1A2}.
\begin{theorem}\label{finiteA2}
Let $\mathscr{B}_R(x_0)$ be the geodesic ball of radius $R$ around $x_0\in \mathcal{M}$ and $\Gamma$ as defined in Proposition \ref{pro1A2}. Let $u$ and $v$ be two solutions of the nonlinear wave equation \eqref{eq1A2} on $[0,T]$ in the sense of Definition \ref{defA2}. We suppose $(u_0,u_1)=(v_0,v_1)$ for all $x\in \mathscr{B}_R(x_0)$. Then $u=v$ for almost all $(t,x)\in \Gamma$.
\end{theorem}
\begin{proof}[Proof]
If $u$ and $v$ are solutions to \eqref{eq1A2} on $[0,T]$, then by Definition \ref{defA2}, one can split the interval $[0,T]$ into subintervals $[\tau_j,\tau_{j+1}]$, $0\leq j\leq J-1$, with $\tau_0=0<\tau_1<...<\tau_J=T$, such that 
\begin{align*}
\forall j\in \{0,..., J-1\}, \quad \mathrm{max} \Big(\|u\|_{L^5([\tau_j,\tau_{j+1}],L^{10}(\mathcal{M}))}, \|v\|_{L^5([\tau_j,\tau_{j+1}],L^{10}(\mathcal{M}))}\Big)\leq \frac{\delta_0}{2}.
\end{align*}
Hence, it suffices to prove the theorem under the additional assumption 
\begin{align*}
    \mathrm{max} \Big(\|u\|_{L^5([0,T],L^{10}(\mathcal{M}))}, \|v\|_{L^5([0,T],L^{10}(\mathcal{M}))}\Big)\leq \frac{\delta_0}{2}.
\end{align*}
From Remark \ref{ReA2}, one has $u=\underset{n\rightarrow +\infty}{\mathrm{lim}} u^{(n)}$ and $v=\underset{n\rightarrow +\infty}{\mathrm{lim}} v^{(n)}$ in $ L^5([0,T],L^{10}(\mathcal{M}))$ where $u^{(n)}$ and $v^{(n)}$ are given by 
\begin{align*}
    u^{(0)}=v^{(0)}=0,\quad & u^{(n+1)}(t)=S_L(t)(u_0,u_1)+\int_0^t \frac{\mathrm{sin}((t-s)\sqrt{-\Delta_g})}{\sqrt{-\Delta_g}}{(u^{(n)})}^5(s) ds, 
    \\& v^{(n+1)}(t)=S_L(t)(v_0,v_1)+\int_0^t \frac{\mathrm{sin}((t-s)\sqrt{-\Delta_g})}{\sqrt{-\Delta_g}}{(v^{(n)})}^5(s) ds.
\end{align*}
 We prove by induction on $n$ that $u^{(n)}(t,x)=v^{(n)}(t,x)$ for almost every $(t,x)\in \Gamma$. This holds for $n=0$ since, by definition, $u^{(0)}=v^{(0)}=0$. We assume that $u^{(n)}(t,x)=v^{(n)}(t,x)$ for almost every $(t,x)\in \Gamma$.
One writes
\begin{align*}
u^{(n+1)}(t)-v^{(n+1)}(t)=S_L(t)(u_0-v_0,&u_1-v_1)+ \\&\int_0^t\frac{\mathrm{sin}((t-s)\sqrt{-\Delta_g})}{\sqrt{-\Delta_g}}\big({(u^{(n)})}^5(s)- {(v^{(n)})}^5(s)\big)ds.
\end{align*}
Due to Proposition \ref{pro2.122A2} and the fact that, by hypothesis, $(u_0,u_1)=(v_0,v_1)$ on $\mathscr{B}_R(x_0)$, we get 
\begin{align*}
S_L(t)(u_0-v_0,u_1-v_1)=0, \quad\text{for almost  all}\quad (t,x)\in \Gamma.
\end{align*}
Hence, using the inductive hypothesis, we deduce ${(u^{(n)})}^5(s)- {(v^{(n)})}^5(s)=0$ for all $x\in \mathscr{B}_{R-s}(x_0)$ and $s\in[0,t]$. Then, from Proposition \ref{pro2.122A2}, we conclude that
\begin{align*}
\frac{\mathrm{sin}((t-s)\sqrt{-\Delta_g})}{\sqrt{-\Delta_g}}\big({(u^{(n)})}^5(s)- {(v^{(n)})}^5(s)\big)=0,
\end{align*}
for almost every $(t,x)$ with $x\in \mathscr{B}_{R-s-(t-s)}(x_0) =\mathscr{B}_{R-t}(x_0)$ and $t\in [0,T]$. Thus,  $u^{(n)}=v^{(n)}$ for almost all $(t,x)\in \Gamma$, which completes the proof by passing to the limit. 
\end{proof}
\subsection{Key properties of regular and singular points} \label{section3.1A2}
Section \ref{section3.1A2} covers the proof of the first three assertions of Theorem \ref{maintheoremA2}, corresponding to (\ref{singularfiniteA2}), (\ref{3.1**A2}), and (\ref{3.2**A2}). 
We will assume throughout this subsection without loss of generality that the blow-up time is $T_+(u) = 1$. The proof is based on the following two lemmas.
\begin{lemma} \label{abA2}
\begin{enumerate}
There exists $\varepsilon>0$ with the following properties.
\item  Let $x_0\in \mathcal{M}$, $t_0\in (0,1)$ and $R>0$. If 
\begin{align}\label{3.5ùùA2}
\int_{\mathscr{B}_{|t_0-1|+R}(x_0)}\Big(|\nabla_gu(t_0)|^2_g+|\partial_tu(t_0)|^2 +|u(t_0)|^2  \Big)\;d\mathrm{vol} <\varepsilon,
\end{align}
and $\varphi \in C^\infty_c (\mathcal{M})$ has a compact support in $\mathscr{B}_R(x_0)$, then $(\varphi u(t), \varphi \partial_t u(t))$ admits
a limit in $\dot{H}^1(\mathcal{M})\times L^2(\mathcal{M})$  as $t$ goes to $1$.
\item Let $t_0 \in (0, 1)$ and $R > 0$. If 
\begin{align}\label{3.6ùùA2}
\int_{\mathcal{M} \setminus \mathscr{B}_{R}(0)}\Big(|\nabla_gu(t_0)|^2_g+|\partial_tu(t_0)|^2+ |u(t_0)|^2 \Big)\;d\mathrm{vol} <\varepsilon,
\end{align}
and $\varphi \in C^\infty (\mathcal{M})$ is equal to $1$ at infinity and is supported in  $\mathcal{M} \setminus \mathscr{B}_{R+|1-t_0|}(0)$, then $(\varphi u(t), \varphi \partial_t u(t))$ has
a limit in $\dot{H}^1(\mathcal{M})\times L^2(\mathcal{M})$  as $t$ tends to $1$.
\end{enumerate}
\end{lemma}
\begin{proof}[Proof]
By \eqref{3.5ùùA2}, we extend $u(t_0)$ and $\partial_tu(t_0)$ by two functions $\tilde{u}_0\in \dot{H}^1(\mathcal{M})$ and $\tilde{u}_1\in L^2(\mathcal{M})$, which are compactly supported in $\mathcal{M}=\mathbb{R}^3$, such that 
\begin{align} \label{1.26*A2}
\tilde{u}_0(x)=u(t_0,x) \quad \text{and} \quad \tilde{u}_1(x)=\partial_tu(t_0,x)\quad \text{if} \quad x\in \mathscr{B}_{|t_0-1|+R}(x_0),
\end{align}
and
\begin{align}\label{1.27*A2}
\int_{\mathcal{M}}\Big(|\nabla_g\tilde{u}_0|^2_g+|\tilde{u}_1|^2+|\tilde{u}_0|^2\Big)\;d\mathrm{vol}\leq C \varepsilon,
\end{align}
for some universal constant $C>0$. From \eqref{1.27*A2},  we get
\begin{align*}
\|\tilde{u}_0\|_{\dot{H}^1(\mathcal{M})}+\|\tilde{u}_1\|_{L^2(\mathcal{M})}\lesssim \varepsilon^\frac{1}{2}.
\end{align*}
Combining this estimate with (2) of Remark \ref{remarlemA2}, we deduce the existence of a solution $\tilde{u}$ to \eqref{eq1A2} with initial condition $(\tilde{u}_0,\tilde{u}_1)$ at $t=t_0$ such that $(\tilde{u},\partial_t\tilde{u})\in C([t_0,t_0+1],\dot{H}^1(\mathcal{M})\times L^2(\mathcal{M}))$. Furthermore, recalling that $u$ and $\tilde{u}$ are both solutions to \eqref{eq1A2} and using \eqref{1.26*A2}, we conclude by Theorem \ref{finiteA2} that 
\begin{align*}
u(t,x)=\tilde{u}(t,x) \quad \text{and} \quad  \partial_tu(t,x)=\partial_t\tilde{u}(t,x) \quad \text{for almost all} \quad (t,x)\in \Gamma,
\end{align*}
where \begin{align*}
\Gamma:=\{(t,x)\in \mathbb{R}\times \mathcal{M}, \quad t_0\leq t< 1 \quad \text{and}\quad  x\in \mathscr{B}_{|t_0-1|+R-|t-t_0|}(x_0)\}.
    \end{align*}
    In particular, $(\varphi u(t), \varphi \partial_t u(t)) = (\varphi \tilde{u}(t), \varphi \partial_t\tilde{u}(t))$ has a limit  in $\dot{H}^1(\mathcal{M})\times L^2(\mathcal{M})$ as $t$ goes to $1$, which concludes the proof of the first item of Lemma \ref{abA2}.\\
The proof of (b) follows from the same approach as in (a). By \eqref{3.6ùùA2}, there exist $\tilde{u}_0\in \dot{H}^1(\mathcal{M})$ and $\tilde{u}_1\in L^2(\mathcal{M})$  such that 
\begin{align} \label{1.26**A2}
\tilde{u}_0(x)=u(t_0,x) \quad \text{and} \quad \tilde{u}_1(x)=\partial_tu(t_0,x)\quad \text{if} \quad x\in \mathcal{M} \setminus \mathscr{B}_{R}(0),
\end{align}
and
\begin{align}\label{1.27**A2}
\int_{\mathcal{M}}\Big(|\nabla_g\tilde{u}_0|^2_g+|\tilde{u}_1|^2 \Big)\;d\mathrm{vol} \leq C\varepsilon,
\end{align}
for some universal constant $C>0$. By \eqref{1.27**A2} and (2) of Remark \ref{remarlemA2}, we deduce the existence of a solution $\tilde{u}$ to \eqref{eq1A2} with initial data $(\tilde{u}_0,\tilde{u}_1)$ at $t=t_0$ such that $(\tilde{u},\partial_t\tilde{u})\in C([t_0,t_0+1],\dot{H}^1(\mathcal{M})\times L^2(\mathcal{M}))$.
Given that $u$ and $\tilde{u}$ are both solutions to \eqref{eq1A2}, and using \eqref{1.26**A2} we get by Theorem \ref{finiteA2} 
\begin{align*}
u(t,x)=\tilde{u}(t,x) \quad \text{and} \quad  \partial_tu(t,x)=\partial_t\tilde{u}(t,x), 
\end{align*}
for almost every $x\in \mathcal{M} \setminus \mathscr{B}_{R+|t_0-t|}(0)$ and $t_0\leq t< 1$.
Hence, $(\varphi u(t), \varphi \partial_t u(t)) = (\varphi \tilde{u}(t), \varphi \partial_t\tilde{u}(t))$ admits a limit in $\dot{H}^1(\mathcal{M})\times L^2(\mathcal{M})$ as $t$ tends to $1$, which completes the proof of the second item of Lemma \ref{abA2}.
\end{proof}
\begin{lemma}\label{lemma 4.2222A2}
For any singular point $s$, and all $t\in [0,1)$,
\begin{align} \label{3.11**A2}
 \int_{\mathscr{B}_{|t-1|}(s)}\Big(|\nabla_gu(t)|^2_g+|\partial_tu(t)|^2+|u(t)|^2\Big)\;d\mathrm{vol}\geq \varepsilon,
\end{align}
where $\varepsilon$ is given by Lemma \ref{abA2}. Furthermore, the set $S$ of singular points is finite.
\end{lemma}
\begin{proof}[Proof]
We proceed by contradiction. Suppose that there exists a singular point $s$, $t_0 \in [0,1)$ and $R>0$ such that
\begin{align}\label{3.12ppA2}
\int_{\mathscr{B}_{|t_0-1|+R}(s)}\Big(|\nabla_gu(t_0)|^2_g+|\partial_tu(t_0)|^2+|u(t_0)|^2\Big)\;d\mathrm{vol} < \varepsilon.
\end{align}
From \eqref{3.12ppA2} and part (a) of Lemma \ref{abA2}, we deduce that $(\varphi u(t), \varphi \partial_t u(t))$ has a limit in $\dot{H}^1(\mathcal{M})\times L^2(\mathcal{M})$  as $t$ tends to $1$, where $\varphi \in C^\infty_c (\mathcal{M})$ is compactly supported in $\mathscr{B}_R(s)$. We denote this limit by $(\tilde{v},\partial_t\tilde{v})$. With \eqref{3.12ppA2}, we obtain
\begin{align*}
\underset{t \rightarrow 1}{\mathrm{lim}}\int_{\mathcal{M}} \Big(|\nabla_g\varphi u(t)|^2_g+|\partial_t\varphi u(t)|^2\Big)\;d\mathrm{vol}&=\int_{\mathscr{B}_{R}(s)}\Big(|\nabla_g\tilde{v}|^2_g+|\partial_t\tilde{v} |^2\Big)\;d\mathrm{vol}
\\&
\leq\int_{\mathscr{B}_{R}(s)}\Big(|\nabla_g\tilde{v}|^2_g+|\partial_t\tilde{v} |^2+|\tilde{v}|^2\Big)\;d\mathrm{vol}<\varepsilon.
\end{align*}
This contradicts the fact that $s$ is a singular point in the sense of Definition \ref{defregulA2}. Hence, we conclude \eqref{3.11**A2}. Then, the finiteness of $S$ follows immediately from \eqref{3.11**A2} and the fact that the blow-up is of type II, and thus \eqref{singularfiniteA2} holds. 
\end{proof}
Now, we are ready to prove \eqref{3.1**A2} and \eqref{3.2**A2} of Theorem \ref{maintheoremA2}. For \eqref{3.1**A2} we need to show that every weak limit of the sequence $(u(t_n),\partial_tu(t_n))$, where $t_n$ tends to $1$ as $n\rightarrow +\infty$, coincides. To do so, let $(v_0,v_1)$ and $(\tilde{v}_0, \tilde{v}_1)$ be two such weak limits. Due to (a) of Lemma \ref{abA2}, one has $(v_0,v_1)=(\tilde{v}_0, \tilde{v}_1)$ around regular points. Since by Lemma \ref{lemma 4.2222A2}, the set of singular points is finite, this implies that $(v_0,v_1)=(\tilde{v}_0, \tilde{v}_1)$, which gives \eqref{3.1**A2}. It remains to prove \eqref{3.2**A2}. Denote by $(v_0,v_1)$ the weak limit of $(u,\partial_tu)$ in $\dot{H}^1(\mathcal{M})\times L^2(\mathcal{M})$ as $t$ goes to $1$. 
By (a) and (b) of Lemma \ref{abA2}, $\big((1-\varphi) u,(1-\varphi) \partial_tu\big)$ converges to $\big((1-\varphi)v_0,(1-\varphi)v_1\big)$ in $\dot{H}^1(\mathcal{M})\times L^2(\mathcal{M})$ as $t$ tends to $1$,
where $\varphi \in C^\infty_c(\mathcal{M})$ is equal to $1$ around each singular point, which concludes the proof of \eqref{3.2**A2}.
\begin{remark}\label{remarkofmaintheA2}
 Let  $u$ be a solution of \eqref{eq1A2} with type II blow-up and $v$ be a solution to \eqref{eq1A2} such that $\big(v(T_+(u)),\partial_tv(T_+(u))\big)=(v_0,v_1)$, where $(v_0,v_1)$ is defined by \eqref{3.1**A2}. Set $a:=u-v.$ By \eqref{3.2**A2} and Theorem \ref{finiteA2}, we conclude that 
\begin{align}\label{supportA2}
 \mathrm{supp}(a)\subset \bigcup\limits_{k=1}^K \Big\{(x,t)\in \mathcal{M}\times \mathbb{R}, \quad 0\leq t<T_+(u) \quad \text{and}\quad x\in \mathscr{B}_{|t-T_+(u)|}(m_k)\Big\}.
 \end{align}
\end{remark}
\section{Description of concentrating waves}\label{section3.2A2}
In this section, we study the behavior of linear and nonlinear concentrating waves. This analysis will be used to establish preliminary results presented in Section \ref{section55A2}, which are essential for the proof of assertion (4) in our main theorem. In the sequel, we denote by $\mathcal{E}$ the energy space $\dot{H}^1(\mathcal{M})\times L^2(\mathcal{M})$.
\subsection{Behavior of linear concentrating waves}
In this part, we describe the asymptotic behavior of linear concentrating waves as defined in Definition \ref{def3.5A2} below. We should mention here that this analysis was previously carried out by Ibrahim \cite{10A2} for the defocusing critical wave equation and Laurent \cite{14A2} for the damped Klein-Gordon equation. Although those works concern nonlinear problems, the sign of the nonlinearity plays no role in the linear setting considered here. We provide detailed proofs for completeness and to obtain more refined estimates within our geometric framework.
\begin{definition}[Concentrating data, linear/nonlinear concentrating wave] \label{def3.5A2}
We define a concentrating data as every element $[(\varphi,\psi),\underline{h}:=(h_n)_n,\underline{x}:=(x_n)_n,\underline{t}:=(t_n)_n]$ in the space $\mathcal{E}\times (\mathbb{R}_+^*\times \mathcal{M}\times \mathbb{R})^{\mathbb{N}}$, such that
\begin{align*}
    \underset{n\rightarrow+\infty}{\mathrm{lim}} h_n=0, \quad \underset{n\rightarrow+\infty}{\mathrm{lim}}x_n=x_\infty \in \mathcal{M}\quad\text{and} \quad \underset{n\rightarrow+\infty}{\mathrm{lim}}t_n=t_\infty\in \mathbb{R} . 
    \end{align*}
A linear concentrating wave $\underline{v}:=(v_n)_n$ associated with the data $[(\varphi,\psi),\underline{h},\underline{x},\underline{t}]$ is a sequence of solutions to
\begin{equation} \label{leneqA2}
\left\{
\begin{array}{l}
\partial^2_tv_n-\Delta_gv_n=0,\quad\quad\quad\quad\hspace{0.8cm} (t,x)\in \mathbb{R}\times \mathcal{M},\\
(v_n,\partial_tv_n)_{|t=t_n}(x)=\frac{1}{\sqrt{h_n}}(\varphi,\frac{1}{h_n}\psi)(\frac{x-x_n}{h_n}).
\end{array}
\right.
\end{equation}
The associated nonlinear concentrating wave is the sequence $(u_n)_n$ of solutions to
\begin{equation} \label{nonnleneqA2}
\left\{
\begin{array}{l}
\partial^2_tu_n-\Delta_gu_n=u_n^5,\quad\quad\quad\quad\hspace{0.8cm}(t,x)\in \mathbb{R}\times \mathcal{M}, \\
(u_n,\partial_tu_n)_{|t=0}(x)=(v_n,\partial_tv_n)_{|t=0}(x).
\end{array}
\right.
\end{equation}
\end{definition}
From now on, we will never distinguish between a concentrating data and its associated linear concentrating wave; we shall usually write $\underline{v}=[(\varphi,\psi),\underline{h},\underline{x},\underline{t}]$. \\

The following lemma shows that, for times close to the concentration time $t_\infty$, the linear concentrating wave is close to the solution of the wave equation with flat metric.
\begin{lemma} \label{lemmmmm3.8A2}
Let $\underline{v}=[(\varphi,\psi),\underline{h},\underline{x},\underline{t}]$ be a linear concentrating wave, and let $(V_{n,\alpha})_n$ be a sequence such that 
\begin{align*}
V_{n,\alpha}(t,x):=\frac{1}{\sqrt{h_n}} V_\alpha\Big(\frac{t-t_n}{h_n},\frac{x-x_n}{h_n}\Big),
\end{align*}
 where $V_\alpha$ is a solution of 
\begin{equation} \label{equationofvdeltaA2}
\left\{
\begin{array}{l}
\partial^2_tV_\alpha-\Delta_{g(x_\infty)} V_\alpha=0,\quad \quad\text{on} \quad\mathbb{R}\times \mathbb{R}^3, \\
(V_\alpha,\partial_tV_\alpha)_{|t=0}(x)=(\varphi_\alpha,\psi_\alpha),
\end{array}
\right.
\end{equation}
with $(\varphi_\alpha,\psi_\alpha)\in (C^\infty_c(\mathcal{M}))^2$ such that
\begin{align*}
    \|(\varphi_\alpha-\varphi,\psi_\alpha-\psi)\|_\mathcal{E}\underset{\alpha \rightarrow +\infty}{\longrightarrow} 0.
\end{align*}
Then, one has
\begin{align*}
\underset{\alpha \rightarrow + \infty}{\mathrm{lim}}\;\underset{n\rightarrow + \infty}{\mathrm{lim \; sup}}\;\underset{t\in [-\Lambda h_n+t_n,\;\Lambda h_n+t_n]}{\mathrm{sup}} \Big\|\Big(v_n(t)-V_{n,\alpha}(t),\partial_t(v_n(t)-V_{n,\alpha}(t))\Big)\Big\|_\mathcal{E}=0,
\end{align*}
for all $\Lambda>0$.
\end{lemma}
The proof of Lemma \ref{lemmmmm3.8A2} relies, in particular, on the following result.
\begin{lemma} \label{lemmm3.7A2}
Let $f(s,y)$ be a regular function defined on $\mathbb{R}\times \mathcal{M}$ such that, for all  $\Lambda>0$, the projection onto the $y$-component of the restriction of $f$ to $ [-\Lambda,\Lambda]\times \mathcal{M}$ is contained in some compact set $K$. Then, one has
\begin{align*}
\underset{n\rightarrow+\infty}{\mathrm{lim\;sup}}\; \|(\Delta_{g(h_n.+x_n)}-\Delta_{g(x_\infty)}) f\|_{L^1([-\Lambda,\Lambda],L^2(\mathcal{M}))}=0,
\end{align*}
where $\Delta_{g(x_\infty)}:=\sum\limits_{1\leq i,j\leq 3}g^{i,j}(x_\infty)\partial_{y_i}\partial_{y_j}$.
\end{lemma}
\begin{proof}[Proof] One writes
\begin{align*}
&\|(\Delta_{g(h_n.+x_n)}-\Delta_{g(x_\infty)}) f\|_{L^1([-\Lambda,\Lambda],L^2(\mathcal{M}))}^2\\&\lesssim \Lambda \int_{-\Lambda}^\Lambda\int_{\mathbb{R}^3}\Big|\sum_{1\leq i,j\leq 3} \partial_{y_i}\Big(\big(\sqrt{\mathrm{det}(g(h_ny+x_n))}g^{i,j}(h_ny+x_n)-g^{i,j}(x_\infty)\big)\partial_{y_j}f(s,y)\Big)\Big|^2 dyds
\\& \lesssim \Lambda^2 h_n \sum_{1\leq i,j\leq 3}\underset{y\in \mathbb{R}^3}{\mathrm{sup}}\Big(\partial_{y_i}\sqrt{\mathrm{det}(g(y))}g^{i,j}(y)\Big)\int_{\mathbb{R}^3}\underset{s\in \mathbb{R}}{\mathrm{sup}}\,|\partial_{y_j}f(s,y)|^2 dy
\\& 
+\Lambda^2\sum_{1\leq i,j\leq 3}\underset{y\in K}{\mathrm{sup}}\Big(\sqrt{\mathrm{det}(g(h_ny+x_n))}g^{i,j}(h_ny+x_n)-g^{i,j}(x_\infty)\Big)\int_{\mathbb{R}^3}\underset{s\in \mathbb{R}}{\mathrm{sup}}\,|\partial_{y_i}\partial_{y_j}f(s,y)|^2 dy.
\end{align*}
Thus, passing to the limit when $n$ goes to infinity, we obtain the desired result.
\end{proof}

\begin{proof}[Proof of Lemma \ref{lemmmmm3.8A2}] Let $s=\frac{t-t_n}{h_n}$ and $y=\frac{x-x_n}{h_n}$. We define the following rescaled functions
\begin{align*}
v_n(t,x):=\frac{1}{\sqrt{h_n}} \tilde{v}_n\Big(\frac{t-t_n}{h_n},\frac{x-x_n}{h_n}\Big)=\frac{1}{\sqrt{h_n}} \tilde{v}_n(s,y).
\end{align*}
Observe that $\tilde{v}_n$ is a solution to the system
\begin{equation*} 
\left\{
\begin{array}{l}
\partial^2_s\tilde{v}_n(s,y)-\Delta_{g(h_n.+x_n)} \tilde{v}_n(s,y)=0 \\
(\tilde{v}_n,\partial_s\tilde{v}_n)_{|s=0}(x)=(\varphi,\psi),
\end{array}
\right.
\end{equation*}
while $V_\alpha$ is a solution to \eqref{equationofvdeltaA2}.
Then, $D_{n,\alpha}:=\tilde{v}_n-V_\alpha$ satisfies
\begin{equation*} 
\left\{
\begin{array}{l}
\partial^2_s D_{n,\alpha}-\Delta_{g(h_n.+x_n)}D_{n,\alpha}=(\Delta_{g(h_n.+x_n)}-\Delta_{g(x_\infty)}) V_\alpha\\
(D_{n,\alpha},\partial_s D_{n,\alpha})_{|s=0}(x)=(\varphi-\varphi_\alpha,\psi-\psi_\alpha).
\end{array}
\right.
\end{equation*}
Using the energy estimate  \eqref{strichA2} of Theorem \ref{thm1A2}, we get
\begin{align}\label{5.3adaA2}
\underset{s\in [-\Lambda,\Lambda]}{\mathrm{sup}}&\Big(\|D_{n,\alpha}(s)\|_{\dot{H}^1(\mathcal{M})}+\|\partial_sD_{n,\alpha}(s)\|_{L^2(\mathcal{M})}\Big)\nonumber\\&
\leq \|(\varphi-\varphi_\alpha,\psi-\psi_\alpha)\|_{\mathcal{E}}+ \|(\Delta_{g(h_n.+x_n)}-\Delta_{g(x_\infty)}) V_\alpha\|_{L^1([-\Lambda,\Lambda],L^2(\mathcal{M}))}.
\end{align}
Since $(\varphi_\alpha,\psi_\alpha)\in (C^\infty_c(\mathcal{M}))^2$, the finite speed of propagation implies that $V_\alpha$ has a compact support  contained in a compact set $K$, depending on $\alpha$. Applying Lemma \ref{lemmm3.7A2} to $V_\alpha$ and taking the limit as $n \rightarrow +\infty$ in \eqref{5.3adaA2}, we obtain
\begin{align*}
\underset{n\rightarrow+\infty}{\mathrm{lim}\;\mathrm{sup}}\underset{s\in [-\Lambda,\Lambda]}{\mathrm{sup}}\Big(\|D_{n,\alpha}(s)\|_{\dot{H}^1(\mathcal{M})}+\|\partial_sD_{n,\alpha}(s)\|_{L^2(\mathcal{M})}\Big)\leq \|(\varphi-\varphi_\alpha,\psi-\psi_\alpha)\|_{\mathcal{E}}.
\end{align*}
As a consequence,
\begin{align*}
\underset{\alpha \rightarrow + \infty}{\mathrm{lim}}\;\underset{n\rightarrow+\infty}{\mathrm{lim}\;\mathrm{sup}}\underset{s\in [-\Lambda,\Lambda]}{\mathrm{sup}}\Big(\|D_{n,\alpha}(s)\|_{\dot{H}^1(\mathcal{M})}+\|\partial_sD_{n,\alpha}(s)\|_{L^2(\mathcal{M})}\Big)=0, \quad \forall \Lambda>0.
\end{align*}
Thus,
\begin{align*}
\underset{\alpha \rightarrow + \infty}{\mathrm{lim}}\;\underset{n\rightarrow+\infty}{\mathrm{lim}\;\mathrm{sup}}\underset{t\in [-\Lambda h_n+t_n,\;\Lambda h_n+t_n]}{\mathrm{sup}}\|v_n(t)-V_{n,\alpha}(t)\|_\mathcal{E}=0, \quad \forall \Lambda>0.
\end{align*}
This completes the proof of Lemma \ref{lemmmmm3.8A2}.
\end{proof}
We now present Lemma \ref{lemma3.9A2}, which establishes the “non-reconcentration” property for linear concentrating waves. Before stating it, we introduce the following definition.
\begin{definition}[Focusing/non-focusing property] \label{defoffocusA2}
A given point $x_\infty \in \mathcal{M}$ is said to be a focus point if there exist $x\in \mathcal{M}$ and $t\in \mathbb{R}$ such that the set 
\begin{align*}
\mathcal{F}_{x_\infty}(x,t):=\{\xi \in S^*_x\mathcal{M}: \quad \mathrm{exp}_x t\xi= x_\infty\}
\end{align*}
of directions of geodesics starting from the point $x$ reaching $x_\infty$ in a time  $t$, has positive surface measure. Here, $S^*_x\mathcal{M}$ denotes the unit cosphere bundle. In this case, we say that $(x,x_\infty,t)\in \mathcal{M}^2\times \mathbb{R}$ is a couple of focus at distance $t$.
We denote by $T_{focus}$ the infimum of the $t > 0$ such that there exists a couple of focus at
distance~$t$. \\
Furthermore, if no focus point occurs in $\mathcal{M}$ on a time interval 
 $I$ containing $0$, then 
\begin{align}\label{4.555A2}
|\mathcal{F}_{x_\infty}(y,t)|=0, \quad\quad \text{for any } t\in I\setminus\{0\},\; y\in \mathcal{M}.
\end{align}
This condition is referred to as the non-focusing property. 
\end{definition}
\begin{remark}
 Observe that in the case of the metric considered here, the set $\{|x|\leq R\}$ may be replaced by a compact manifold for which we have necessarily $T_{focus}>0$ and for $\{|x|>R\}$ the metric is flat  where the geodesics never intersect and $T_{focus}=+\infty$. If the time of concentration $I$ is sufficiently small, then the non-focusing property \eqref{4.555A2} holds, as a consequence of a property of the exponential map.
\end{remark}
\begin{lemma} \label{lemma3.9A2}
Let $I$ be a compact time interval containing $t_\infty$ satisfying the following non-focusing property
\begin{align}\label{focusA2}
|\mathcal{F}_{x_\infty}(y,t_\infty-t)|=0, \quad\quad \text{for any } t\in I\setminus\{t_\infty\},\; y\in \mathcal{M}.
\end{align}
Then, the linear concentrating wave $\underline{v}=[(\varphi,\psi),\underline{h},\underline{x},\underline{t}]$ with $\underset{n\rightarrow+\infty}{\mathrm{lim}}t_n=t_\infty$ satisfies
\begin{align} \label{3.15A2}
\underset{n\rightarrow+\infty}{\mathrm{lim}\;\mathrm{sup}}\underset{t\in I \setminus [t_n -\Lambda h_n,t_n+\Lambda h_n]}{\mathrm{sup}} \|v_n(t,.)\|_{L^6(\mathcal{M})}\underset{\Lambda \rightarrow +\infty}{\longrightarrow}0.
\end{align}
\end{lemma}
The proof of Lemma \ref{lemma3.9A2} is based on the following lemma.
\begin{lemma} \label{lemintermeA2} Let $(w_n)_n$ be a sequence with support contained in a compact set $K$, independent of $n$, that converges weakly to $0$ in $\dot{H}^1(\mathcal{M})$. After extraction of a subsequence, let $\alpha$ be its associated positive Radon measure on $\mathcal{M}=\mathbb{R}^3$, that is
\begin{align} \label{3.1888888A2}
\int_{\mathbb{R}^3}|\nabla_gw_n(x)|^2\Theta(x) \sqrt{\mathrm{det}(g(x))}dx \underset{n\rightarrow +\infty}{\longrightarrow} \int_{\mathbb{R}^3} \Theta(x) \sqrt{\mathrm{det}(g(x))}d\alpha(x),
\end{align}
for all $\Theta\in C^\infty_c(\mathcal{M})$. Assume
\begin{align}\label{alphaaaA2}
 \alpha(\{x\})=0, \quad \text{for all } x\in \mathcal{M}.
 \end{align}
Then,
\begin{align*}
     \underset{n\rightarrow+\infty}{\mathrm{lim\;sup}}\;\|w_n\|_{L^6(\mathcal{M})}=0.
\end{align*}
\end{lemma}
The proof of the Lemma \ref{lemintermeA2} is based on Proposition \ref{pro1.4A2} below. In order to state it, we first introduce the following definition.
\begin{definition}[Sequence of transformations] \label{orthogonalA2}
    Let $(h_n)_n$ be a sequence in $(0,+\infty)$ and $(x_n)_n$ a sequence in $\mathcal{M}$. The pair $(h_n,x_n)_n$ is called a sequence of transformations. We say that two sequences of transformations $ \Lambda_n=(h_n,x_n)_n$ and $ \Lambda_n'=(\lambda_n,y_n)_n$ are orthogonal when they satisfy 
    \begin{align*}
       \underset{n\rightarrow +\infty} {\mathrm{lim}}\Big|\mathrm{log}(\frac{h_n}{\lambda_n})\Big|=+\infty\quad \text{or} \quad  \underset{n\rightarrow +\infty} {\mathrm{lim}} \frac{|x_n-y_n|}{h_n}=+\infty.
    \end{align*}
\end{definition}
\begin{proposition}[Profile decomposition of bounded sequences in $\dot{H}^1$] \label{pro1.4A2}
For any bounded sequence $(u_n)_n$ in $\dot{H}^1(\mathcal{M})$, there exist a subsequence $(u'_n)_n$ of $(u_n)_n$, a sequence of transformations $\big((h_n^{(j)},x_n^{(j)})_n\big)_{j\geq 1}$  that are pairwise orthogonal and a sequence $(\varphi_j)_{j\geq 1}$ in $\dot{H}^1(\mathcal{M})$ satisfying for any $\ell\in \mathbb{N}^*$
    \begin{align}\label{1.10A2}
        u'_n(x)=\sum_{j=1}^\ell \big(\frac{1}{h_n^{(j)}}\big)^\frac{1}{2} \varphi_j(\frac{x-x_n^{(j)}}{h_n^{(j)}})+r_n^{(\ell)}(x),
    \end{align}   
  where $\underset{\ell\rightarrow +\infty}{\mathrm{lim}}\;\underset{n\rightarrow+\infty}{\mathrm{lim} \;\mathrm{sup}}\;\|r_n^{(\ell)}\|_{L^6(\mathcal{M})}=0$.   
     Furthermore, one has
     \begin{align}
     &\sqrt{h_n^{(j)}} u'_n(h_n^{(j)}y+x_n^{(j)})\xrightharpoonup[n \rightarrow +\infty]{} \varphi_j \quad \text{in} \quad \dot{H}^1(\mathcal{M}) \quad \forall 1\leq j\leq\ell,
\\&\sqrt{h_n^{(j)}} r_n^{(\ell)}(h_n^{(j)}y+x_n^{(j)})\xrightharpoonup[n \rightarrow +\infty]{}  0 \quad \text{in} \quad \dot{H}^1(\mathcal{M}), \quad \forall 1\leq j\leq\ell.\label{3.9999A2}
\end{align}
\end{proposition}
Proposition \ref{pro1.4A2} is proved by Gérard in \cite{6A2}.
\begin{proof}[Proof of Lemma \ref{lemintermeA2}]
Applying Proposition \ref{pro1.4A2} to the bounded sequence $(w_n)_n$ in $\dot{H}^1(\mathcal{M})$, there exist a pairwise orthogonal sequence of transformations $\big((h_n^{(j)}, x_n^{(j)})_n\big)_{j \geq 1}$ as defined in Definition \ref{orthogonalA2} and a sequence $(\varphi_j)_{j \geq 1} \subset \dot{H}^1(\mathcal{M}),$ such that for all $ \ell \in \mathbb{N}^*$,
\begin{align}\label{eqqA2}
w_n( x) = \sum_{j=1}^\ell \frac{1}{\sqrt{h_n^{(j)}}} \varphi_j\left(\frac{x - x_n^{(j)}}{h_n^{(j)}}\right)+r_n^{(\ell)}(x),
\end{align}
with
\begin{align} \label{therestA2}
\underset{\ell\rightarrow +\infty}{\mathrm{lim}}\;\underset{n\rightarrow+\infty}{\mathrm{lim\; sup}} \;\|r_n^{(\ell)}\|_{L^6(\mathcal{M})}=0.
\end{align}
Furthermore,
\begin{align}
&\sqrt{h_n^{(j)}} w_n(h_n^{(j)}y+x_n^{(j)})\xrightharpoonup[n \rightarrow +\infty]{}  \varphi_j \quad \text{in} \quad \dot{H}^1(\mathcal{M}) \quad \forall 1\leq j\leq\ell,
\\&\sqrt{h_n^{(j)}} r_n^{(\ell)}(h_n^{(j)}y+x_n^{(j)})\xrightharpoonup[n \rightarrow +\infty]{}  0 \quad \text{in} \quad \dot{H}^1(\mathcal{M}), \quad \forall 1\leq j\leq\ell.\label{4.12222aA2}
\end{align}
Note that, since $(w_n)_n$ converges weakly to $0$ in $\dot{H}(\mathcal{M})$, no non-zero profile exists such that $h_n^{(j)}\rightarrow \ell_\infty^{(j)}\in(0,\infty)$ and $x_n^{(j)} \rightarrow x_\infty^{(j)}\in \mathcal{M}$ as $n\rightarrow +\infty$. Moreover, given that the sequence $(w_n)_n$ has a compact support contained in a compact set $K$, independent of $n$, no non-zero profile exists with $h_n^{(j)}\rightarrow +\infty$ or $|x_n^{(j)}|\rightarrow +\infty$ as $n\rightarrow +\infty$. Thus, for every non-zero profile, one has
\begin{align} \label{nonzeroprofileA2}
h_n^{(j)} \underset{n\rightarrow +\infty}{\longrightarrow} 0 \quad\text{and} \quad x_n^{(j)}\underset{n\rightarrow +\infty}{\longrightarrow} x_\infty^{(j)}, \quad \text{for all} \quad j\geq 1.
\end{align}
We now proceed to show, by contradiction, that no non-zero profile exists.
To do so, suppose that there exists a non-zero profile $\varphi_{j_0}\in \dot{H}^1(\mathcal{M})$. Let $\epsilon>0$ and consider $0\leq\tilde{\Theta}_{j_0,\epsilon}\leq 1$ a function in $C^\infty_c(\mathcal{M})$ such that 
\begin{align} \label{boulethetaA2}
 \mathrm{supp}(\tilde{\Theta}_{j_0,\epsilon})\subset B(x_\infty^{(j_0)},\epsilon)\quad \text{and} \quad \tilde{\Theta}_{j_0,\epsilon}(x_\infty^{(j_0)})=1.
 \end{align}
 From \eqref{eqqA2}, one can write
\begin{align} \label{3.23ddA2}
I&:=\underset{n\rightarrow + \infty}{\mathrm{lim \;sup}}\int_{\mathbb{R}^3}  \Big|\nabla_g w_n(x)\Big|^2_g \tilde{\Theta}_{j_0,\epsilon}(x) \sqrt{\det g(x)} dx \nonumber
\\&= \underset{n\rightarrow + \infty}{\mathrm{lim \;sup}}\Big(I_{1,n}+I_{2,n}+I_{3,n}+I_{4,n}\Big),
\end{align}
with
\begin{align*}
&I_{1,n}:= \int_{\mathbb{R}^3} \Big|\sum_{j=1}^\ell \frac{1}{\sqrt{h_n^{(j)}}} \nabla_g\big(\varphi_j(\frac{x - x_n^{(j)}}{h_n^{(j)}})\big)\Big|^2_g\tilde{\Theta}_{j_0,\epsilon}(x) \sqrt{\mathrm{det}(g(x))}dx,
\\&I_{2,n}:=\int_{\mathbb{R}^3}\Big|\nabla_gr_n^{(\ell)}(x)\Big|^2_g\tilde{\Theta}_{j_0,\epsilon}(x) \sqrt{\mathrm{det}(g(x))}dx,
\\&I_{3,n}:= 2\int_{\mathbb{R}^3}\sum_{j=1}^\ell (\frac{1}{h_n^{(j)}})^{\frac{3}{2}}g\Big(\nabla_g\varphi_j(\frac{x - x_n^{(j)}}{h_n^{(j)}}),\nabla_g r_n^{(\ell)}(x)\Big)\tilde{\Theta}_{j_0,\epsilon}(x) \sqrt{\mathrm{det}(g(x))}dx,
\\&I_{4,n}:=2\int_{\mathbb{R}^3}\sum_{\underset{j<j'}{1\leq j,j'\leq \ell}}(\frac{1}{h_n^{(j)}})^\frac{3}{2}(\frac{1}{h_n^{(j')}})^\frac{3}{2}g\Big(\nabla_g\varphi_j(\frac{x - x_n^{(j)}}{h_n^{(j)}}),\nabla_g\varphi_{j'}(\frac{x - x_n^{(j')}}{h_n^{(j')}})\Big)\\&\hspace{6.7cm}\times\tilde{\Theta}_{j_0,\epsilon}(x) \sqrt{\mathrm{det}(g(x))}dx.
\end{align*}
Due to the orthogonality of the profiles $(\varphi_j)_{j \geq 1}$, one has
\begin{align} \label{I_11A2}
\underset{n\rightarrow + \infty}{\mathrm{lim \; sup}}\;I_{1,n}= \underset{n\rightarrow + \infty}{\mathrm{lim \; sup}}\int_{\mathbb{R}^3} \sum_{j=1}^\ell \Big|\frac{1}{\sqrt{h_n^{(j)}}} \nabla_g\big(\varphi_j(\frac{x - x_n^{(j)}}{h_n^{(j)}})\big)\Big|^2_g\tilde{\Theta}_{j_0,\epsilon}(x) \sqrt{\mathrm{det}(g(x))}dx,
\end{align}
and
\begin{align}\label{I_44A2}
   \underset{n\rightarrow + \infty}{\mathrm{lim}}\; I_{4,n}=0.
\end{align}
Applying the change of variables $y=\frac{x - x_n^{(j)}}{h_n^{(j)}}$ then using \eqref{4.12222aA2} and $\eqref{boulethetaA2}$ along with the fact that $(\varphi_j)_{j \geq 1} \subset \dot{H}^1(\mathcal{M})$, yields
\begin{align}\label{I_33A2}
   \underset{n\rightarrow + \infty}{\mathrm{lim}}\;I_{3,n}=&\underset{n\rightarrow + \infty}{\mathrm{lim}}\; 2 \int_{\mathbb{R}^3}\sum_{j=1}^\ell \sum_{1\leq k,f\leq 3}(h_n^{(j)})^\frac{3}{2}g_{k,f}(h_n^{(j)}y+x_n^{(j)})\Big(\nabla_{g(h_n^{(j)}y+x_n^{(j)})}\varphi_j(y)\Big)^k\nonumber\\&\hspace{1.7cm}\times\Big(\nabla_{g(h_n^{(j)}y+x_n^{(j)})}r_n^{(\ell)}(h_n^{(j)}y+x_n^{(j)})\Big)^f\sqrt{\mathrm{det}(g(h_n^{(j)}y+x_n^{(j)}))} dy \nonumber
    \\&\quad+\underset{n\rightarrow + \infty}{\mathrm{lim}}\; 2 \int_{\mathbb{R}^3}\sum_{j=1}^\ell \sum_{1\leq k,f\leq 3} (h_n^{(j)})^\frac{3}{2}g_{k,f}(h_n^{(j)}y+x_n^{(j)})\Big(\nabla_{g(h_n^{(j)}y+x_n^{(j)})}\varphi_j(y)\Big)^k
    \nonumber\\&\hspace{1.7cm}\times
    \Big(\nabla_{g(h_n^{(j)}y+x_n^{(j)})}r_n^{(\ell)}(h_n^{(j)}y+x_n^{(j)})\Big)^f\big(\tilde{\Theta}_{j_0,\epsilon}(h_n^{(j)}y+x_n^{(j)})-1\big)\ \nonumber\\&\hspace{1.7cm}\times\sqrt{\mathrm{det}(g(h_n^{(j)}y+x_n^{(j)}))} dy
    =0.
\end{align}
Noting that $I_{2,n}\geq 0$ and collecting \eqref{3.23ddA2},\eqref{I_11A2}, \eqref{I_44A2}, and \eqref{I_33A2}, we conclude
\begin{align}
I&\geq \underset{n\rightarrow + \infty}{\mathrm{lim\; sup}} \int_{\mathbb{R}^3} \Big|\frac{1}{\sqrt{h_n^{(j_0)}}} \nabla_g\big(\varphi_{j_0}(\frac{x - x_n^{(j_0)}}{h_n^{(j_0)}})\big)\Big|^2_g\tilde{\Theta}_{j_0,\epsilon}(x) \sqrt{\mathrm{det}(g(x))}dx \nonumber
\\&
=\underset{n\rightarrow + \infty}{\mathrm{lim\; sup}}\int_{\mathbb{R}^3}  |\nabla_g\varphi_{j_0}(y)|^2_g\tilde{\Theta}_{j_0,\epsilon}(h_n^{(j_0)}y+x_n^{(j_0)}) \sqrt{\mathrm{det}(g(h_n^{(j_0)}y+x_n^{(j_0)}))}dy \nonumber
\\& =\int_{\mathbb{R}^3}  |\nabla_g\varphi_{j_0}(y)|^2_g \sqrt{\mathrm{det}(g(x_\infty^{(j_0)}))}dy =C>0, \label{4.200A2}
\end{align}
where in the last line we use \eqref{nonzeroprofileA2}, \eqref{boulethetaA2} and the fact that the non-zero profile $\varphi_{j_0}$ is bounded in $\dot{H}^1(\mathcal{M})$.\\
On the other hand, combining \eqref{3.1888888A2} with the boundedness of $(w_n)_n$ in $\dot{H}^1(\mathcal{M})$, it follows from the dominated convergence theorem, that
\begin{align} \label{3.24A2}
I = \int\limits_{\mathbb{R}^3}\tilde{\Theta}_{j_0,\epsilon}(x) \sqrt{\mathrm{det}(g(x))}d\alpha(x) \underset{\epsilon\rightarrow 0}{\longrightarrow}\sqrt{\mathrm{det}(g(x_\infty^{(j_0)}))} \alpha(\{x_\infty^{(j_0)}\}).
\end{align}
Collecting \eqref{4.200A2} and \eqref{3.24A2}, we get
\begin{align*}
\alpha(\{x_\infty^{(j_0)}\}) \geq \frac{C}{\sqrt{\mathrm{det}(g(x_\infty^{(j_0)}))}}>0,
\end{align*}
which gives a contradiction with \eqref{alphaaaA2}. Hence, for all $j\geq 1$, we have $\varphi_j=0$ and then, from \eqref{eqqA2} along with \eqref{therestA2}, we deduce
\begin{align*}
   \underset{n\rightarrow+\infty}{\mathrm{lim\;sup}}\;\|w_n\|_{L^6(\mathcal{M})}=0,
\end{align*}
which completes the proof of Lemma \ref{lemintermeA2}.
\end{proof}
\begin{proof}[Proof of Lemma \ref{lemma3.9A2}]
Here, we follow the proof of Lemma 2.3 given in \cite{10A2}, with additional details provided. We proceed by contradiction, assuming that \eqref{3.15A2} does not hold. Then, there exist a constant \( c > 0 \), a sequence $(\Lambda_k)_k$ of real numbers tending to $+\infty$, a subsequence $(h_{n_k})_k$ of $(h_n)_n$ converging to $0$ and a subsequence $(\tau_{n_k})_k \subset I$ converging to \( \tau \), such that
\begin{align} \label{contraA2}
|\tau_{n_k} - t_\infty| > \Lambda_k h_{n_k} \quad \text{and} \quad \lim_{k\to +\infty} \|v_{n_k}(\tau_{n_k}, \cdot)\|_{L^6(\mathcal{M})} = c > 0.
\end{align}
 In this proof, we distinguish two cases $\tau \neq t_\infty $ and $\tau = t_\infty $. We begin by treating the first one.
 \begin{itemize}
     \item \textbf{Case 1: $\tau \neq t_\infty $.} For clarity, the proof is carried out in three steps.
     \end{itemize}
 \begin{enumerate}
  \item \textbf{Step 1: Construction of a positive Radon measure.}\\
  Recall that $v_{n_k}$ is a solution to
\begin{equation} \label{leneqs1A2}
\left\{
\begin{array}{l}
(\partial^2_t-\Delta_g)v_{n_k}=0,  \quad \quad\text{on} \quad\mathbb{R}\times \mathcal{M}, \\
(v_{n_k},\partial_tv_{n_k})_{|t=t_{n_k}}(x)=\frac{1}{\sqrt{h_{nk}}}(\varphi,\frac{1}{h_{n_k}}\psi)(\frac{x-x_{n_k}}{h_{nk}}).
\end{array}
\right.
\end{equation}   
By the density of $(C^\infty_c(\mathcal{M}))^2$ in $\mathcal{E}$, there exists $(\varphi_\alpha,\psi_\alpha)\in (C^\infty_c(\mathcal{M}))^2$, such that
\begin{align*}
    \|(\varphi_\alpha-\varphi,\psi_\alpha-\psi)\|_\mathcal{E}\underset{\alpha \rightarrow +\infty}{\longrightarrow} 0.
\end{align*}
We denote by $v_{{n_k},\alpha}$ the solution to
\begin{equation} \label{leneqsA2}
\left\{
\begin{array}{l}
(\partial^2_t-\Delta_g)v_{{n_k},\alpha}=0,  \quad \quad\text{on} \quad\mathbb{R}\times \mathcal{M}, \\
(v_{{n_k},\alpha},\partial_tv_{{n_k},\alpha})_{|t=t_{n_k}}(x)=\frac{1}{\sqrt{h_{nk}}}(\varphi_\alpha,\frac{1}{h_{n_k}}\psi_\alpha)(\frac{x-x_{n_k}}{h_{nk}})=(\varphi_{{n_k},\alpha}(x),\psi_{{n_k},\alpha}(x)).
\end{array}
\right.
\end{equation}
Then, by energy conservation and \eqref{1.222222eqA2}, one has
\begin{align}\label{3.288888888A2}
   \underset{k\to +\infty}{\mathrm{lim}}\underset{t\in \mathbb{R}}{\mathrm{sup}}\;\|\big(v_{{n_k},\alpha}-v_{n_k},\partial_t(v_{{n_k},\alpha}-v_{n_k})\big)\|_{\mathcal{E}} \underset{\alpha \rightarrow +\infty}{\longrightarrow}0. 
\end{align}
 Therefore, due to \eqref{1.3A2}, \eqref{sobolevA2} and \eqref{3.288888888A2}, one gets 
\begin{align} \label{3.299999999A2}
    \underset{k\to +\infty}{\mathrm{lim}} \underset{t\in \mathbb{R}}{\mathrm{sup}}\;\|v_{{n_k},\alpha}(t,.)-v_{n_k}(t,.)\|_{L^6(\mathcal{M})}\underset{\alpha \rightarrow +\infty}{\longrightarrow}0.
\end{align}
From the second equality in \eqref{leneqsA2}, we deduce that the sequence 
$(\varphi_{{n_k},\alpha},\psi_{{n_k},\alpha})_k$ converges weakly to $(0,0)$ in the energy space $\mathcal{E}$.
As a consequence, for all $t\in \mathbb{R}$, $(v_{{n_k},\alpha}(t, .))_k $ converges weakly to $0$ in $ \dot{H}^1(\mathcal{M})$. This follows from the fact that the operator $S_L(t-t_\infty)$, defined in \eqref{S_LlinearA2}, is linear and unitary in the energy space $\mathcal{E}$, and therefore preserves weak convergence. Since $\tau_{n_k}$ converges to $\tau$, the sequence $(v_{{n_k},\alpha}(\tau_{n_k}, .))_k $ converges weakly to $0$ in $ \dot{H}^1(\mathcal{M}) $. Hence, after extracting a subsequence, $\mu_{k,\alpha}:=|\nabla_gv_{{n_k},\alpha}(\tau_{n_k},x)|^2_g$ vaguely converges to a positive Radon measure $\alpha_\alpha$ on $\mathcal{M}=\mathbb{R}^3$ so that
\begin{align} \label{3.18A2}
\int_{\mathbb{R}^3}|\nabla_gv_{{n_k},\alpha}(\tau_{n_k}, x)|^2_g\Theta(x) \sqrt{\mathrm{det}(g(x))}dx \underset{k\rightarrow +\infty}{\longrightarrow} \int_{\mathbb{R}^3} \Theta(x) \sqrt{\mathrm{det}(g(x))}d\alpha_\alpha(x),
\end{align}
for all $\Theta\in C^\infty_c(\mathcal{M})$ and all $\alpha>0$.\\
Now, consider two symbols $a\in S^1(\mathbb{R}^{3}\times \mathbb{R}^{3})$ and $c\in S^0(\mathbb{R}^{3}\times \mathbb{R}^{3})$ (see Appendix  \ref{appenAA2}) defined by
\begin{align*}
    a(x,\xi)=\Big(\sum\limits_{1\leq i,j\leq 3}g^{i,j}(x)\xi_i\xi_j\Big)^\frac{1}{2},
\end{align*}
and 
\begin{align*}
  c(x,\xi)= \frac{1}{2a(x,\xi)}\Big(\sum\limits_{1\leq i,j\leq 3}(\mathrm{det}(g))^{-\frac{1}{2}}&D_{x_i}\Big(g^{i,j}(x)(\mathrm{det}(g))^{\frac{1}{2}}\Big) \xi_j\\&-\frac{1}{i}\nabla_\xi a(x,\xi).\nabla_x a(x,\xi)\Big).
\end{align*}
Set
\begin{align*}
    q(x,\xi)=a(x,\xi)+c(x,\xi).
\end{align*}
Then, let $Q=\Op(q)$ denote the pseudo-differential operator  associated with the symbol $q$ (see Appendix  \ref{appenAA2})
and set
\begin{align}\label{4.20000A2}
    v_{{n_k},\alpha,\pm}(t,.)=(\partial_t\pm iQ)v_{{n_k},\alpha}(t,.), \quad \text{ for all } t\in \mathbb{R}.
\end{align}
From \eqref{4.20000A2} and since for any $t\in \mathbb{R}$, $\big(v_{{n_k},\alpha}(t, .)\big)_k $ converges weakly to $0$ in $ \dot{H}^1(\mathcal{M})$, one can associate to the sequence $\big(v_{{n_k},\alpha,\pm}(t,.)\big)_k$ a positive Radon measure $\mu^t_{\pm,\alpha}$ on $\mathcal{M}\times S^2$ such that for every pseudo-differential operator $B$ of order $0$, one has
\begin{align} \label{3.19A2}
\big(Bv_{{n_k},\alpha,\pm}(t,.),v_{{n_k},\alpha,\pm}(t,.)\big)_{L^2(\mathcal{M})}\underset{k\rightarrow +\infty}{\longrightarrow} \int_{\mathbb{R}^3\times S^2} \sigma_0(B)\sqrt{\mathrm{det}(g(x))}d\mu^t_{\pm,\alpha}(x,\xi),
\end{align}
where $\sigma_0(B)$ is the principal symbol of $B$ (see \cite{188A2} and Theorem 5.14 in \cite{13A2}). 
\item \textbf{Step 2: Propagation of this measure.}\\
Consider the pseudo-differential operator $R_0$ defined by
\begin{align} 
\label{4.2111A2}
    R_0=(\partial_t^2-\Delta_g)-(\partial_t-iQ)(\partial_t+iQ),
\end{align}
with principal symbol $\sigma(R_0)$ of order $0$ as shown in Lemma 3.2 of \cite{12A2}. By the finite speed of propagation, $v_{{n_k},\alpha}(t,.)$ has a compact support contained in a compact set $K$, independent of $k$. Thus, due to \eqref{4.2111A2} and Rellich theorem, one has
\begin{align*}
    (\partial_t \mp iQ)v_{{n_k},\alpha,\pm}(t,.)&=(\partial_t \mp iQ)(\partial_t\pm i Q) v_{{n_k},\alpha}(t,.)
    \\&=\big( -R_0+\partial_t^2-\Delta_g\big) v_{{n_k},\alpha}(t,.)=-R_0 v_{{n_k},\alpha}(t,.)=o(1),
\end{align*}
 as $k\rightarrow +\infty$ in $ L^2(\mathcal{M})$ for all $\alpha>$0.
Then, since $B$ is a pseudo-differential operator of order $0$, one can write
\begin{align} \label{teboulA2}
\frac{d}{dt}(B& v_{{n_k},\alpha,\pm}(t,.),v_{{n_k},\alpha,\pm}(t,.))_{L^2(\mathcal{M})}\nonumber\\&= \pm i\big( [BQ-Q^*B]v_{{n_k},\alpha,\pm}(t,.),v_{{n_k},\alpha,\pm}(t,.)\big)_{L^2(\mathcal{M})}+o(1),
\end{align}
as $ k\rightarrow +\infty$ in $ L^2(\mathcal{M})$.
Given that $a\in S^1(\mathbb{R}^{3}\times \mathbb{R}^{3})$ and $c\in S^0(\mathbb{R}^{3}\times \mathbb{R}^{3})$ (see Appendix  \ref{appenAA2}), one has
\begin{align*}BQ=\Op(b)\Op(q)=\Op(b \# q)\quad \text{and}\quad Q^*B=\Op(q^*)\Op(b)=\Op(q^*\# b)\end{align*} with
\begin{align*}
   & b \# q \sim (a+c)b -i\nabla_\xi b.\nabla_x a \quad \text{mod} \quad S^{-1},
   \\& q^* \# b \sim (a+ \bar{c}-i \partial_\xi\partial_x a)b-i\nabla_\xi a . \nabla_xb \quad \text{mod} \quad S^{-1}.
\end{align*}
Then, the principal symbol $\sigma([BQ-Q^*B])$ is expressed as
\begin{align*}
    \sigma([BQ-Q^*B])&= \Big(i \partial_\xi\partial_x a-\bar{c}+c\Big)b -i\nabla_\xi b.\nabla_x a  +i\nabla_\xi a . \nabla_xb 
    \\&= \Big(i \partial_\xi\partial_x a+2 i\mathrm{Im}(c)\Big)b +i H_{a(x,\xi)}b(x,\xi),
\end{align*}
where $H_{a(x,\xi)}b(x,\xi):=\{a(x,\xi),b(x,\xi)\}$ and $\{.,.\}$ denotes the Poisson bracket. Hence, we deduce by \eqref{3.19A2} and \eqref{teboulA2}, in the sense of distributions, that
\begin{align} \label{4.2333sA2}
 \frac{d}{dt}(B v_{{n_k},\alpha,\pm}(t,.)&,v_{{n_k},\alpha,\pm}(t,.))_{L^2(\mathcal{M})}\nonumber\\&\underset{k\rightarrow+ \infty}{\longrightarrow}\int_{\mathbb{R}^3\times S^2} \sigma_0(B)\sqrt{\mathrm{det}(g(x))} \frac{d}{dt}d\mu^t_{\pm,\alpha}(x,\xi) \nonumber
 \\&= \mp \int_{\mathbb{R}^3\times S^2} \Big( \partial_\xi\partial_x a+2\mathrm{Im}(c)\Big)\sigma_0(B)\sqrt{\mathrm{det}(g(x))}d\mu^t_{\pm,\alpha}(x,\xi) \nonumber
 \\&\quad
 \mp  \int_{\mathbb{R}^3\times S^2} H_{a(x,\xi)}\sigma_0(B)\sqrt{\mathrm{det}(g(x))} d\mu^t_{\pm,\alpha}(x,\xi).
\end{align}
Thus, \eqref{4.2333sA2} implies that $\mu^t_{\pm,\alpha}$ satisfies 
\begin{equation} \label{eqmuA2}
\left\{
\begin{array}{l}
\Big(\partial_t \pm H_{a(x,\xi)}\pm e(x,\xi)\Big)\mu^t_{\pm,\alpha} =0\\
\mu^t_{\pm,\alpha_{|t=t_\infty}}=\mu^{t_\infty}_{\pm,\alpha},
\end{array}
\right.
\end{equation}
with
\begin{align*}
e(x,\xi):=\partial_{\xi}\partial_x a(x,\xi)+2 \mathrm{Im}(c),
\end{align*}
and where according to \eqref{4.20000A2}, $\mu^{t_\infty}_{\pm,\alpha}$ is the positive Radon measure associated to the sequence $\big((\partial_t\pm iQ)v_{{n_k},\alpha}(t_{n_k},.)\big)_k$  so that 
\begin{align*}
\int_{\mathbb{R}^3} (|\psi_{{n_k},\alpha}(x)\pm i Q \varphi_{{n_k},\alpha}(x)|^2) &\Theta(x)\sqrt{\mathrm{det}(g(x))}dx \nonumber
\\&
\underset{k\rightarrow +\infty}{\longrightarrow} \int_{\mathbb{R}^3\times S^2} \Theta(x)\sqrt{\mathrm{det}(g(x))}d\mu^{t_\infty}_{\pm,\alpha}(x,\xi),
\end{align*}
for all  $\Theta\in C^\infty_c(\mathcal{M})$.
\item \textbf{Step 3: Contradiction argument.}\\
Due to \eqref{eqmuA2}, we get
\begin{align} \label{4.2555sA2}
\mu^{\tau}_{\pm,\alpha}(x,\xi) &=\mathrm{exp}\Big(\mp\int_{t_\infty}^\tau e(\Phi_{s-\tau}(x,\xi))ds \Big). \mu^{t_\infty}_{\pm,\alpha}\big(\Phi_{t_\infty-\tau}(x,\xi)\big) \nonumber
\\&=f^\tau(x,\xi). \mu^{t_\infty}_{\pm,\alpha}\big(\Phi_{t_\infty-\tau}(x,\xi)\big),
\end{align}
where  $\Phi_t$ denotes the Hamiltonian flow of $H_{a(x,\xi)}$ and $$f^\tau(x,\xi):=\mathrm{exp}\Big(\mp\int_{t_\infty}^\tau e(\Phi_{s-\tau}(x,\xi))ds \Big).$$
By direct computation, one writes
\begin{align}\label{mutinfiA2}
\mu^{t_\infty}_{\pm,\alpha}=k_{\pm,\alpha}(\xi)\delta_{x-x_\infty}\otimes d\sigma(\xi),
\end{align}
with
\begin{align*}
k_{\pm,\alpha}(\xi):=(2\pi)^{-3}\int_0^{+\infty} |\hat{\psi}_\alpha(r\xi)\pm i a(x_\infty,r\xi) \hat{\varphi}_\alpha(r\xi)|^2 r^2 dr,
\end{align*}
where $\hat{\psi}_\alpha$ and $\hat{\varphi}_\alpha$ denote the Fourier transforms of $\psi_\alpha$ and $\varphi_\alpha$, respectively. Note that one has the identity
\begin{align}\label{3.39999A2}
\Pi \circ \Phi_{t_\infty-\tau}(\{x\}\times S^2)=\{\mathrm{exp}_x((t_\infty-\tau)\xi): \; \xi \in S^2\}
\end{align}
where $\Pi:T^*\mathcal{M}\rightarrow \mathcal{M}$ denotes the natural projection. Combining  \eqref{4.2555sA2}, \eqref{mutinfiA2} and \eqref{3.39999A2}, we obtain 
\begin{align*}
\mu^{\tau}_{\pm,\alpha}(\{x\} \times S^2)&\leq C(t_\infty,\tau) \int_{\Phi_{t_\infty-\tau}(\{x\}\times S^2)} d\mu^{t_\infty}_{\pm,\alpha}(x,\xi)
\\&\leq C(t_\infty,\tau) \int_{\mathcal{F}_{x_\infty}(x,t_\infty-\tau)} k_{\pm,\alpha}(\xi)d\sigma(\xi)
\\& \leq C(t_\infty,\tau) C' |\mathcal{F}_{x_\infty}(x,t_\infty-\tau)|,
\end{align*}
for all $x\in \mathcal{M}$,
with
\begin{align*}
C(t_\infty,\tau)&:=\mathrm{sup}\{f^\tau(x,\xi): \; (x,\xi)\in \mathrm{supp}( \mu^{t_\infty}_{\pm,\alpha})\}=\mathrm{sup}\{f^\tau(x_\infty,\xi): \;\; \xi \in S^2\},
\end{align*}
and
\begin{align*}
C':=\mathrm{sup}\{k_{\pm,\alpha}(\xi):\; \xi\in \mathcal{F}_{x_\infty}(x,t_\infty-\tau)\}.
\end{align*}
Since $\tau \in I\setminus\{t_\infty\}$, the hypothesis \eqref{focusA2} yields
\begin{align} \label{4.277777777sA2}
\mu^{\tau}_{\pm,\alpha}(\{x\} \times S^2) = 0, \quad \text{for all } x \in \mathcal{M}.
\end{align}
Defining
\begin{align}\label{4.28888sA2}
\mu^t_\alpha:=\frac{1}{2}(\mu^t_{+,\alpha}+\mu^t_{-,\alpha}),
\end{align}
one gets, by \eqref{4.277777777sA2} and \eqref{4.28888sA2},
\begin{align} \label{3.24muA2}
\mu^{\tau}_\alpha(\{x\} \times S^2) = 0, \quad \text{for all } x \in \mathcal{M}.
\end{align}
Besides, from \eqref{4.20000A2}, \eqref{3.19A2} and \eqref{4.28888sA2}, we deduce
\begin{align}
\label{3.20A2}
\int_{\mathbb{R}^3} \Big(|\nabla_gv_{{n_k},\alpha}(t,x)|^2_g+|\partial_t v_{{n_k},\alpha}(t,x)|^2\Big) &\Theta(x)\sqrt{\mathrm{det}(g(x))}dx \nonumber
\\&
\underset{k\rightarrow +\infty}{\longrightarrow} \int_{\mathbb{R}^3\times S^2} \Theta(x)\sqrt{\mathrm{det}(g(x))}d\mu^t_\alpha(x,\xi),
\end{align}
for all  $\Theta\in C^\infty_c(\mathcal{M})$ and any $\alpha>0$.
Observe that for $t=\tau_{n_k}$, one has
\begin{align}\label{taunk5.21A2}
 \int_{\mathbb{R}^3} \Big(|\nabla_gv_{{n_k},\alpha}(\tau_{n_k},x)|^2_g&+|\partial_t v_{{n_k},\alpha}(\tau_{n_k},x)|^2\Big) \Theta(x)\sqrt{\mathrm{det}(g(x))}dx\nonumber\\&\geq\int_{\mathbb{R}^3}|\nabla_gv_{{n_k},\alpha}(\tau_{n_k}, x)|^2_g\Theta(x) \sqrt{\mathrm{det}(g(x))}dx  ,
\end{align}
for all positive $\Theta\in C^\infty_c(\mathcal{M})$. Then, recalling that $\tau_{n_k} \underset{k\rightarrow +\infty}{\longrightarrow} \tau$ and using \eqref{3.18A2}, \eqref{3.20A2} and \eqref{taunk5.21A2}, we obtain
\begin{align}\label{alphaA2}
\alpha_\alpha(\{x\})\leq \mu^\tau_\alpha(\{x\} \times S^2), \quad \text{for all } x\in \mathcal{M}.
\end{align}
Hence, \eqref{3.24muA2} gives
\begin{align} 
 \alpha_\alpha(\{x\})=0, \quad \text{for all } x\in \mathcal{M} \quad \text{and any}\quad \alpha>0.
\end{align}
Then, applying Lemma \ref{lemintermeA2} to the sequence $\big(v_{{n_k},\alpha}(\tau_{n_k}, .)\big)_k$, which has a compact support contained in a compact set $K$, independent of $k$, we get
\begin{align*}
 \underset{k\rightarrow+\infty}{\mathrm{lim\;sup}}\;\|v_{{n_k},\alpha}(\tau_{n_k}, .)\|_{L^6(\mathcal{M})}=0, \quad \text{for all} \quad \alpha>0,
\end{align*}
which leads to a contradiction with \eqref{contraA2} and \eqref{3.299999999A2}, concluding the proof for the first case.
 \end{enumerate}
Now, we treat the second case.
\begin{itemize}
     \item \textbf{Case 2: $\tau= t_\infty $.} 
     \end{itemize}
Consider the rescaled function
$$
\tilde{v}_{{n_k},\alpha}(s, y) := \sqrt{\varepsilon_k} v_{{n_k},\alpha}(t_{n,k} + s\varepsilon_k , x_{n_k}+\varepsilon_k y), \quad \text{with } \varepsilon_k:= |t_\infty - \tau_{n_k}|.
$$
Note that the rescaled sequence $(\tilde{v}_{{n_k},\alpha})_k$ consists of solutions to the system
\begin{equation} \label{resceqA2}
\left\{
\begin{array}{l}
(\partial^2_s-\Delta_{g(\varepsilon_k \cdot+x_{n_k})}\tilde{v}_{{n_k},\alpha}=0 \\
(\tilde{v}_{{n_k},\alpha},\partial_s\tilde{v}_{{n_k},\alpha})_{|s=0}(y)=\frac{1}{\sqrt{\tilde{h}_{n_k}}}(\varphi_\alpha,\frac{1}{\tilde{h}_{n_k}}\psi_\alpha)(\frac{y}{\tilde{h}_{n_k}}),
\end{array}
\right.
\end{equation}
where $\tilde{h}_{n_k}:= \frac{h_{n_k}}{ \varepsilon_k}$. Note that, since by the contradiction hypothesis, $(\Lambda_k)_k$ tends to $+\infty$ and by \eqref{contraA2}, one has $\varepsilon_k> \Lambda_k h_{n_k}$, it follows that $\tilde{h}_{n_k}$ tends to $0$ as $k \rightarrow +\infty$.
Denote by $\tilde{Q}_k$ the pseudo-differential operator associated with the symbol $q_k(y,\xi)=q(\varepsilon_ky+x_{n_k},\xi)$ and define
\begin{align*}
   & \tilde{v}_{{n_k},\alpha,\pm}(s,.):=(\partial_s\pm i\tilde{Q}_k)\tilde{v}_{{n_k},\alpha}(s,.), \quad \text{ for all } s\in \mathbb{R},
  \\& \tilde{R}_0:= (\partial_s^2-\Delta_{g(\varepsilon_k.+x_{nk})})-(\partial_s-i\tilde{Q}_k)(\partial_s+i\tilde{Q}_k).
\end{align*}
Following the arguments used in the proof of Lemma 2.3 in \cite{10A2}, one can show that the microlocal defect measure $\mu^s_{\pm,\alpha} $ associated to the sequence $\tilde{v}_{{n_k},\alpha,\pm}$ propagates along the bicharacteristics of the Hamiltonian flow corresponding to the constant-coefficient operator $H_{a(x_\infty,\xi)}$. Here, it's worth noting that, when applying such arguments, we do not make use of the geometric hypothesis \eqref{focusA2}. As in the previous case, this leads to a contradiction, concluding the proof of Lemma \ref{lemma3.9A2}.
\end{proof}
\subsection{Description of nonlinear concentrating waves}
In this subsection, we study the behavior of nonlinear concentrating waves defined in Definition \ref{def3.5A2}, at three key stages: before, near, and after the concentration time. This is described precisely in Theorem \ref{proposisimiltheoreA2} below. The result was originally proved by Ibrahim in \cite{10A2} for the defocusing case, and later extended by Laurent in \cite{14A2} to the damped defocusing wave equation. Here, we provide a detailed proof in our setting for the focusing nonlinear wave equation \eqref{eq1A2}.\\

In this part, for any time interval $I\subset \mathbb{R}$, we define the Energy-Strichartz norm
\begin{align*}
    |||u|||_I:=\underset{t\in I}{\mathrm{sup}}\Big(\|u\|_{\dot{H}^1(\mathcal{M})}+\|\partial_tu\|_{L^2(\mathcal{M})}\Big)+\|u\|_{L^5(I,L^{10}(\mathcal{M}))}.
\end{align*}
We also denote by $\mathcal{E}_\infty$ the energy space $\mathcal{E}_\infty:=\dot{H}^1(T_{x_\infty}\mathcal{M})\times L^2(T_{x_\infty}\mathcal{M})$ associated with the norm
\begin{align}
    \|(f_1,f_2)\|_{\mathcal{E}_\infty}^2=\int_{\mathbb{R}^3}|\nabla_{g(x_\infty)}f_1(x)|^2_g \sqrt{\mathrm{det}(g(x_\infty))} dx+\int_{\mathbb{R}^3} |f_2(x)|^2 \sqrt{\mathrm{det}(g(x_\infty))}dx,
\end{align}
where $|\nabla_{g(x_\infty)} f_1(x)|^2_g:= \sum\limits_{1\leq i,k\leq3}g^{i,k}(x_\infty)\frac{\partial f_1}{\partial x_i}(x) \frac{\partial f_1 }{\partial x_k}(x)$.
\begin{theorem} \label{proposisimiltheoreA2}
 Let $\underline{v} = [(\varphi, \psi), \underline{h},\underline{x}, \underline{t}]$ be a linear concentrating wave and $\underline{u}$ its associated nonlinear concentrating wave. Assume that for every interval $I_n$ containing $t_\infty$ such that $\bar{I}_n \subset \subset I_{\mathrm{max}}(u_n)$, and for all $n \in \mathbb{N}$, the interval $I_n$ is contained in some interval $[-T,T]$ satisfying the following non-focusing property (see Definition \ref{defoffocusA2})
\begin{align}\label{nonfocusingpropertyA2}
    |\mathcal{F}_{x_\infty}(y,t_\infty-t)|=0 \quad \text{for all} \quad t\in [-T,T]\setminus\{t_\infty\}, y\in \mathcal{M},
\end{align}
and that
\begin{align}\label{3.111155555A2}
          \underset{n \rightarrow +\infty}{\mathrm{lim\; sup}}\; \big\| u_n\|_{L^5(\bar{I}_n,L^{10}(\mathcal{M}))}<+\infty.
\end{align}
There exist three linear concentrating waves denoted by
$[(\varphi_i, \psi_i),\underline{h}, \underline{x}, \underline{t}]$, $i = 1, 2$ and $3$ such that the following holds. 
For $\Lambda>0$, set
\begin{align*}
\quad \quad K^{n,\Lambda}_1=[-T, t_n -& \Lambda h_n],\quad \quad \quad K^{n,\Lambda}_2=[t_n-\Lambda h_n , t_n + \Lambda h_n] \quad \quad\\&\text{and}\quad \quad K^{n,\Lambda}_3=[ t_n + \Lambda h_n,T].
\end{align*}
Then, one has
\begin{align}
    &\underset{n\rightarrow +\infty}{\mathrm{lim \;sup} }\;|||u_n-[(\varphi_1, \psi_1),\underline{h}, \underline{x}, \underline{t}] |||_{K^{n,\Lambda}_1} \underset{\Lambda \rightarrow +\infty}{\longrightarrow} 0, \label{firstassertion1A2}
    \\&\underset{n\rightarrow +\infty}{\mathrm{lim \;sup} }\;|||u_n-[(\varphi_3, \psi_3),\underline{h}, \underline{x}, \underline{t}] |||_{K^{n,\Lambda}_3}\underset{\Lambda \rightarrow +\infty}{\longrightarrow} 0,
\label{thirdassertion2A2}
 \\&
    \underset{n\rightarrow +\infty}{\mathrm{lim \;sup} }\;|||u_n-w_n |||_{K^{n,\Lambda}_2} = 0, \quad \Lambda>0, \label{secondassertion3A2} 
\end{align}
where 
$$w_n(t,x):=\frac{1}{\sqrt{h_n}} w(\frac{t-t_n}{h_n},\frac{x-x_n}{h_n}),$$ and $w$ is a solution to
\begin{equation} 
\left\{
\begin{array}{l}
\partial^2_sw-\Delta_{g(x_\infty)} w=w^5 \quad \text{on} \quad\mathbb{R}\times \mathbb{R}^3, \\
(w,\partial_s w)_{|s=0}=(\varphi_2,\psi_2),
\end{array}
\right.
\end{equation}
\end{theorem}
\begin{remark} \label{choicesoffiA2}
\begin{enumerate}
\item The limits expressed with the Energy–Strichartz norms \eqref{firstassertion1A2}, \eqref{thirdassertion2A2}  and \eqref{secondassertion3A2} follow from the energy estimate in \eqref{strichA2} once the result is proven in the norm $L^5L^{10}$.
\item Translating in time and
rescaling $v_n(t,x)$, and extracting subsequences, we will always assume that
one of the following three cases occurs
\begin{align*}
 \forall  n, \quad t_n=0 \quad \text{or}\quad  \underset{n\rightarrow+\infty}{\mathrm{lim}}\frac{t_n}{h_n}\in \{- \infty,+\infty\}.
 \end{align*}
 \item If condition \eqref{3.111155555A2} is replaced by the requirement that
 $$\underline{v} = [(\varphi, \psi), \underline{h},\underline{x}, \underline{t}]$$
be a linear concentrating wave satisfying
\begin{align} \label{hyp2A22}
\|(\varphi,\psi)\|_{\mathcal{E}\infty}^2
< \frac{2}{3} \|\nabla W\|_{L^2(\mathbb{R}^3)}^2,
\end{align}
then the conclusions of Theorem \ref{proposisimiltheoreA2} remain valid. Although condition \eqref{3.111155555A2} appears explicitly in Theorem \ref{proposisimiltheoreA2}, it is in fact automatically verified under assumption \eqref{hyp2A22}, as follows from the proof.
Moreover, the behavior of the associated nonlinear concentrating wave for times near and after the concentration time is governed by the following proposition, which yields scattering for any constant metric on the tangent plane $T_{x_\infty}\mathcal{M}=\mathbb{R}^3$.
\end{enumerate}
\end{remark}

\begin{proposition}[Scattering on $\mathbb{R}^3$] \label{PropoScatteringA2}
Let $v$ be a solution of 
    \begin{equation*} 
\left\{
\begin{array}{l}
\partial^2_tv-\Delta_{g(x_\infty)} v=0 \quad \text{on} \quad\mathbb{R}\times \mathbb{R}^3, \\
(v,\partial_tv)_{|t=0}=(\varphi,\psi)\in \mathcal{E}_\infty,
\end{array}
\right.
\end{equation*}
with
\begin{align}\label{hyposcatterA2}
\|(\varphi,\psi)\|_{\mathcal{E}_\infty}^2<\frac{2}{3}\|\nabla W\|^2_{L^2(\mathbb{R}^3)},
\end{align}
and let $u_\pm$ be a solution to  
  \begin{equation}  \label{upmA2}
\left\{
\begin{array}{l} 
\partial^2_tu_\pm-\Delta_{g(x_\infty)} u_\pm=u_\pm^5 \quad \text{on} \quad\mathbb{R}\times \mathbb{R}^3, \\
\underset{t\rightarrow \pm \infty}{\mathrm{lim}} \|(v-u_\pm,\partial_t(v-u_\pm))\|_{\mathcal{E}_\infty} =0.
\end{array}
\right.
\end{equation}
Then, $u_\pm$ is globally defined and scatters for $t\rightarrow \mp \infty$. Moreover, the wave operators
\begin{align*}
    \Omega_\pm: (\varphi, \psi)\rightarrow (u_\pm, \partial_tu_\pm)_{|t=0},
\end{align*}
are bijective from 
\begin{align*}
\Big\{(\varphi,\psi) \in \mathcal{E}_\infty,\; \|(\varphi,\psi)\|_{\mathcal{E}_\infty}^2< \frac{2}{3}\|\nabla W\|^2_{L^2(\mathbb{R}^3)} \Big\},
\end{align*} onto 
\begin{align*}
\Big\{(u_\pm, \partial_tu_\pm)_{t=0} \in \mathcal{E}_\infty&,\;\; \mathcal{E}_\infty(u_{\pm_{|t=0}},\partial_tu_{\pm_{|t=0}})< E(W,0)\; \\& \quad\quad\text{and} \quad\int_{\mathbb{R}^3}|\nabla_{g(x_\infty)} u_{\pm_{|t=0}}|^2_g \sqrt{\mathrm{det}(g(x_\infty)}dx<\int_{\mathbb{R}^3}|\nabla W|^2 dx \Big\},
\end{align*}
where \begin{align*}
    \mathcal{E}_\infty(u_\pm,\partial_tu_\pm):=\frac{1}{2}\|( u_\pm,\partial_t u_\pm)\|_{\mathcal{E}_\infty}^2-\frac{1}{6}\int_{\mathbb{R}^3}|u_\pm|^6\sqrt{\mathrm{det}(g(x_\infty)}dx.
    \end{align*}
\end{proposition}
\begin{remark}\label{existwellA2}
    Note that, using a fixed point argument as in the proof of Theorem \ref{th2A2}, one can show that there exist $T,T'\in \mathbb{R}$ such that $u_+\in L^5((T,+\infty), L^{10}(\mathbb{R}^3))$ and $u_-\in L^5((-\infty,T'), L^{10}(\mathbb{R}^3))$ and
    \begin{align*}
        &u_+(t,x)=v(t,x)-\int_t^{+\infty}\frac{\mathrm{sin}\big((t-s)\sqrt{-\Delta_{g(x_\infty)}}\big)}{\sqrt{-\Delta_{g(x_\infty)}}}u_+^5(s) ds,
        \\&u_-(t,x)=v(t,x)+\int_{-\infty}^t\frac{\mathrm{sin}\big((t-s)\sqrt{-\Delta_{g(x_\infty)}}\big)}{\sqrt{-\Delta_{g(x_\infty)}}}u_-^5(s) ds.
    \end{align*}
\end{remark}
\begin{proof}[Proof of Proposition \ref{PropoScatteringA2}]
Let $B=(b_{i,k})_{1\leq i,k\leq 3}$ be a symmetric, positive and definite matrix such that 
$B^2=(g^{i,k}(x_\infty))_{1\leq i,k\leq 3}$. Define $\tilde{v}(t,x)=v(t,Bx)$ and $\tilde{u}_\pm(t,x)=u_\pm(t,Bx)$. Then, $\tilde{v}$ is a solution to
 \begin{equation*} 
\left\{
\begin{array}{l}
\partial^2_t\tilde{v}-\Delta \tilde{v}=0 \quad \text{on} \quad\mathbb{R}\times \mathbb{R}^3, \\
(\tilde{v},\partial_t\tilde{v})_{|t=0}(x)=\big(\varphi(Bx),\psi(Bx)\big)\in \dot{H}^1(\mathbb{R}^3)\times L^2(\mathbb{R}^3),
\end{array}
\right.
\end{equation*}
and $\tilde{u}_\pm$ satisfies
\begin{align} \label{eqoftildeuA2}
\partial^2_t\tilde{u}_\pm(t,x)-\Delta\tilde{u}_\pm(t,x)=\tilde{u}_\pm^5(t,x)\quad \text{on} \quad\mathbb{R}\times \mathbb{R}^3.
\end{align}
One has 
\begin{align} \label{calcul1A2}
 \|\tilde{v}(t,.)\|_{\dot{H}^1(\mathbb{R}^3)}^2&=\|\nabla \tilde{v}(t,.)\|^2_{L^2(\mathbb{R}^3)}=\int_{\mathbb{R}^3} \sum\limits_{1\leq i,k\leq3}\frac{\partial \tilde{v}}{\partial x_i}(t,x) \frac{\partial\tilde{v} }{\partial x_k}(t,x)dx\nonumber
\\&
=\int_{\mathbb{R}^3}\sum\limits_{1\leq i,k\leq3}(\sum\limits_{1\leq l\leq3}b_{i,l}b_{l,k})\frac{\partial v}{\partial x_i}(t,Bx) \frac{\partial v }{\partial x_k}(t,Bx)dx\nonumber
\\&
=\int_{\mathbb{R}^3}\sum\limits_{1\leq i,k\leq3}g^{i,k}(x_\infty)\frac{\partial v}{\partial x_i}(t,Bx) \frac{\partial v }{\partial x_k}(t,Bx)dx\nonumber
\\&
=\frac{1}{|\mathrm{det}(B)|}\int_{\mathbb{R}^3}\sum\limits_{1\leq i,k\leq3}g^{i,k}(x_\infty)\frac{\partial v}{\partial y_i}(t,y) \frac{\partial v }{\partial y_k}(t,y)dy \nonumber
\\&
=\sqrt{\mathrm{det}(g(x_\infty))}\int_{\mathbb{R}^3}\sum\limits_{1\leq i,k\leq3}g^{i,k}(x_\infty)\frac{\partial v}{\partial y_i}(t,y) \frac{\partial v }{\partial y_k}(t,y)dy.
\end{align}
Moreover,
\begin{align} \label{calcul2A2}
\|\tilde{v}(t,.)\|_{L^2(\mathbb{R}^3)}^2=\int _{\mathbb{R}^3}|v(t,Bx)|^2 dx=\sqrt{\mathrm{det}(g(x_\infty))}\|v(t,.)\|^2_{L^2(\mathbb{R}^3)}.
\end{align}
From \eqref{calcul1A2} and \eqref{calcul2A2}, together with the second equality in \eqref{upmA2}, it follows that
\begin{align}\label{newvers3544A2}
  \|(\tilde{v}-\tilde{u}_\pm,\partial_t(\tilde{v}-\tilde{u}_\pm)\|_{\dot{H}^1(\mathbb{R}^3)\times L^2(\mathbb{R}^3)}= \|(v-u_\pm,\partial_t(v-u_\pm)\|_{\mathcal{E}_{\infty}}  \underset{t\rightarrow \pm \infty}{\longrightarrow}0.
\end{align}
 By the Sobolev inequality, one gets
\begin{align}
    \|\tilde{v}-\tilde{u}_\pm\|_{L^6(\mathbb{R}^3)}\leq C_{I_3}\|\tilde{v}-\tilde{u}_\pm\|_{\dot{H}^1(\mathbb{R}^3)}.
\end{align}
Then, letting $t$ tend to $\pm\infty$ and using respectively \eqref{newvers3544A2} and $\underset{t\rightarrow \pm \infty}{\mathrm{lim}} \|\tilde{v}\|_{L^6(\mathbb{R}^3)}=0$ along with \eqref{hyposcatterA2}, we obtain
\begin{align*}
    \underset{t\rightarrow \pm \infty}{\mathrm{lim}}E(\tilde{u}_\pm,\partial_t\tilde{u}_\pm)&:= \underset{t\rightarrow \pm \infty}{\mathrm{lim}}\frac{1}{2}\Big(\|\nabla \tilde{u}_\pm\|_{L^2(\mathbb{R}^3)}^2+\|\partial_t \tilde{u}_\pm\|_{L^2(\mathbb{R}^3)}^2\Big)-\frac{1}{6}\|\tilde{u}_\pm\|_{L^6(\mathbb{R}^3)}^6
    \\&=\frac{1}{2}\|(\tilde{v},\partial_t\tilde{v})_{|t=0}\|_{\dot{H}^1(\mathbb{R}^3)\times L^2(\mathbb{R}^3)}^2= \frac{1}{2}\|\big(\varphi(B.),\psi(B.)\big)\|_{\dot{H}^1(\mathbb{R}^3)\times L^2(\mathbb{R}^3)}^2
    \\&=\frac{1}{2}\|(\varphi,\psi)\|_{\mathcal{E}_\infty}^2<\frac{1}{3}\|\nabla W\|_{L^2(\mathbb{R}^3)}^2.
\end{align*}
Thus,
\begin{align} \label{kenig1A2}
    E(\tilde{u}_\pm,\partial_t\tilde{u}_\pm)<\frac{1}{3}\|\nabla W\|_{L^2(\mathbb{R}^3)}^2=E(W,0).
\end{align}
 Moreover, since 
 \begin{align*}
      \underset{t\rightarrow \pm\infty}{\mathrm{lim}} \|(\tilde{u}_\pm,\partial_t\tilde{u}_\pm)\|_{\dot{H}^1(\mathbb{R}^3)\times L^2(\mathbb{R}^3)}^2  =\|\varphi,\psi\|_{\mathcal{E}_\infty}^2 <  \frac{2}{3}\|\nabla W\|_{L^2(\mathbb{R}^3)}^2,
 \end{align*}
 then, for large $T>0$
    \begin{align} \label{kenig2A2}
       \|\nabla \tilde{u}_\pm(T)\|_{L^2(\mathbb{R}^3)}^2<\|\nabla W\|_{L^2(\mathbb{R}^3)}^2.
    \end{align}
  Therefore, in view of \eqref{kenig1A2} and \eqref{kenig2A2}, using the result of Kenig-Merle \cite{9A2} (see Theorem 1.1), one concludes that $\tilde{u}_\pm$ with initial data at time $T$ is globally defined and scatters as $t$ tends to $\mp\infty$. Furthermore, by definition, $u_\pm$ is globally defined and scatters as $t$ tends to $\mp\infty$. In particular, $u_\pm$ is defined for $t=0$ and as a 
  consequence the wave operators $\Omega_\pm$ are well defined. Conversely, let $u_\pm$ satisfy the first equation of \eqref{upmA2}, such that 
\begin{align*}
\mathcal{E}_\infty(u_{\pm_{|t=0}},\partial_t u_{\pm_{|t=0}})< E(W,0)\; \text{and}\; \int_{\mathbb{R}^3}|\nabla_{g(x_\infty)} u_{\pm_{|t=0}}|^2_g \sqrt{\mathrm{det}(g(x_\infty)}dx<\int_{\mathbb{R}^3}|\nabla W|^2 dx.
\end{align*}
Then, $\tilde{u}_\pm$ verifies \eqref{eqoftildeuA2}, with
\begin{align}\label{3.6000000A2}
E(\tilde{u}_{\pm_{|t=0}},\partial_t\tilde{u}_{\pm_{|t=0}})< E(W,0)\quad \text{and}\quad\int_{\mathbb{R}^3}|\nabla \tilde{u}_{\pm_{|t=0}}|^2 dx<\int_{\mathbb{R}^3}|\nabla W|^2 dx.
\end{align}
Hence, based on the result of Kenig-Merle \cite{9A2}, $\tilde{u}_\pm$ is globally defined and scatters, i.e, there exists $(\tilde{\varphi},\tilde{\psi})\in \dot{H}^1(\mathbb{R}^3)\times L^2(\mathbb{R}^3)$ such that
\begin{align*}
    \underset{t\rightarrow \pm \infty}{\mathrm{lim}} \|(S_{LE}(t)(\tilde{\varphi},\tilde{\psi})-\tilde{u}_\pm),\partial_t(S_{LE}(t)(\tilde{\varphi},\tilde{\psi})-\tilde{u}_\pm)\|_{\dot{H}^1(\mathbb{R}^3)\times L^2(\mathbb{R}^3)} =0,
\end{align*}
with $$S_{LE}(t)(\tilde{\varphi},\tilde{\psi}):=\mathrm{cos}(t\sqrt{-\Delta})\tilde{\varphi} +\frac{\mathrm{sin}(t\sqrt{-\Delta})}{\sqrt{-\Delta}}\tilde{\psi}.$$
Thus, there exists $(\varphi,\psi)\in \mathcal{E}_\infty$, such that 
\begin{align*}
    \varphi(Bx)=\tilde{\varphi}(x) \quad \text{and}\quad \psi(Bx)=\tilde{\psi}(x).
\end{align*}
Moreover, from \eqref{3.6000000A2}, one has
\begin{align*}
  \frac{1}{2}\|(\varphi,\psi)\|_{\mathcal{E}_\infty}^2=\frac{1}{2}  \|(\tilde{\varphi},\tilde{\psi})\|_{\dot{H}^1(\mathbb{R}^3)\times L^2(\mathbb{R}^3)}^2=E(\tilde{u}_{\pm_{|t=0}},\partial_t\tilde{u}_{\pm_{|t=0}})<\frac{1}{3}\|\nabla W\|^2_{L^2(\mathbb{R}^3)}.
\end{align*}
Hence, the wave operators $\Omega_\pm$ are bijective from \begin{align*}
\Big\{(\varphi,\psi) \in \mathcal{E}_\infty,\; \|(\varphi,\psi)\|_{\dot{H}^1(\mathbb{R}^3)\times L^2(\mathbb{R}^3)}^2< \frac{2}{3}\|\nabla W\|^2_{L^2(\mathbb{R}^3)} \Big\},
\end{align*} onto 
\begin{align*}
\Big\{(u_\pm, \partial_tu_\pm)_{t=0}& \in \mathcal{E}_\infty,\;\; \mathcal{E}_\infty(u_{\pm_{|t=0}},\partial_tu_{\pm_{|t=0}})< E(W,0)\; \\& \quad\quad\text{and} \quad\int_{\mathbb{R}^3}|\nabla_{g(x_\infty)} u_{\pm_{|t=0}}|^2_g \sqrt{\mathrm{det}(g(x_\infty)}dx<\int_{\mathbb{R}^3}|\nabla W|^2 dx \Big\},
\end{align*} which completes the proof of Proposition \ref{PropoScatteringA2}. 
\end{proof}
\begin{remark}
 Let $u$ be a solution to $\partial_t^2u-\Delta u=u^5$ with initial data $(u, \partial_tu)_{|t=0}=(u_0,u_1) \in \dot{H}^1(\mathbb{R}^3)\times L^2(\mathbb{R}^3)$ satisfying
    \begin{align} \label{5.444444A2}
        \frac{1}{2}\Big(\|\nabla u_0\|_{L^2(\mathbb{R}^3)}^2+\|u_1\|_{L^2(\mathbb{R}^3)}^2\Big)\leq \frac{1}{3} \|\nabla W\|_{L^2(\mathbb{R}^3)}^2.
    \end{align}
    We claim that $u$ is globally defined and scatters. In fact, from \eqref{5.444444A2}, one has
    \begin{align}\label{4.41kenigA2}
       \|\nabla u_0\|_{L^2(\mathbb{R}^3)}^2\leq\frac{2}{3}  \|\nabla W\|_{L^2(\mathbb{R}^3)}^2< \|\nabla W\|_{L^2(\mathbb{R}^3)}^2,
    \end{align}
leading to two distinct cases:
    \begin{itemize}
        \item  \textbf{Case 1: $u_0 \neq 0$.}\\
        In this case, due to \eqref{5.444444A2}, one has
        \begin{align}\label{4.42kenigA2}
            E(u_0,u_1)&<  \frac{1}{2}\Big(\|\nabla u_0\|_{L^2(\mathbb{R}^3)}^2+\|u_1\|_{L^2(\mathbb{R}^3)}^2\Big)\nonumber
            \\&\leq \frac{1}{3} \|\nabla W\|_{L^2(\mathbb{R}^3)}^2=E(W,0).
        \end{align}
Thus, from \eqref{4.41kenigA2}, \eqref{4.42kenigA2} and the result of Kenig-Merle \cite{9A2}, we conclude that $u$ is globally defined and scatters.
        \item  \textbf{Case 2: $u_0 = 0$.}\\
        Here, by \eqref{5.444444A2}, we obtain
        \begin{align}
        E(u_0,u_1)=\frac{1}{2}\|u_1\|_{L^2(\mathbb{R}^3)}^2\leq \frac{1}{3} \|\nabla W\|_{L^2(\mathbb{R}^3)}^2=E(W,0).
        \end{align}
        \begin{itemize}
            \item If $E(u_0,u_1)<E(W,0)$, then \eqref{4.41kenigA2} gives, using the result of Kenig-Merle \cite{9A2}, that $u$ is globally defined and scatters.\\
       \item If $ E(u_0,u_1)=E(W,0)$, then \eqref{4.41kenigA2} implies, by the result of Duyckaerts-Merle \cite{15A2} (see Theorem 2 there), that the solution $u$ is globally defined and scatters, unless $u = W^-$ (up to the symmetry of the equation) where $W^-$ is a globally defined non-scattering radial solution of 
       \begin{align*}
        \partial_t^2W^--\Delta W^-=(W^-)^5,
        \end{align*}
        with initial conditions $(W^-_0,W^-_1)\in \dot{H}^1(\mathbb{R}^3)\times L^2(\mathbb{R}^3) $ satisfying, based on Theorem 1 in \cite{15A2},
        \begin{align*}
           & E(W,0)= E(W^-_0,W^-_1),\quad\quad \quad\quad\quad\quad\quad\|\nabla W^-\|_{L^2(\mathbb{R}^3)}< \|\nabla W\|_{L^2(\mathbb{R}^3)}, \\& \|W^-\|_{L^8((-\infty,0)\times \mathbb{R}^3)}<+\infty\quad\quad\text{and}\quad\quad \underset{t\rightarrow +\infty}{\mathrm{lim}}W^-(t)=W \quad \text{in} \quad \dot{H}^1(\mathbb{R}^3).
        \end{align*}
        \end{itemize}
        However, this case is excluded since $u$ is odd in time, whereas $W^-$ is not.
    \end{itemize}
\end{remark}

The proof of Theorem \ref{proposisimiltheoreA2} relies, in particular, on the following bootstrap lemma.
\begin{lemma}[Bootstrap lemma] \label{bootsrapA2}
    Let $S:[0,T]\rightarrow \mathbb{R}_+$ be a continuous map such that for all $t\in [0,T]$, one has 
    \begin{align} \label{sdetA2}
        S(t)\leq a+bS(t)^\theta,
    \end{align}
    where the constants $a,b>0$ and $\theta>1$ satisfy 
    \begin{align}\label{hypA2}
    a<(1-\frac{1}{\theta})\frac{1}{(\theta b)^\frac{1}{\theta-1}}\quad \quad \text{and } \quad\quad S(0)\leq \frac{1}{(\theta b)^\frac{1}{\theta-1}}.
    \end{align}
    Then, for all $t\in [0,T]$, we have 
    \begin{align*}
        S(t)\leq \frac{\theta}{\theta -1} a .
    \end{align*}
\end{lemma}
\begin{proof}[Proof of Lemma \ref{bootsrapA2}]
Let 
\begin{align*}
  U:=\left\{t\in [0,T] \quad\text{such that}\quad S(t)\leq  \frac{\theta}{\theta -1} a\right\}.
\end{align*}
    We prove that $U$ is non empty, open and closed in $[0,T]$, which implies by the connectedness  of $[0,T]$ that $U=[0,T]$. First, observe that, by \eqref{sdetA2} and the second estimate in \eqref{hypA2}, 
 \begin{align*}
     S(0)&\leq a+b S(0)S(0)^{\theta-1}
     \leq a+bS(0) \frac{1}{\theta b}=a+\frac{1}{\theta }S(0),
 \end{align*}
hence $U$ is non empty.
Moreover, by the continuity of $S$, the set $U$ is closed in $[0,T]$. We now show that $U$ is also open in $[0,T]$. Let $t\in U$. Then using the first estimate in \eqref{hypA2}, one gets
 \begin{align*}
     S(t)\leq a+bS(t)^{\theta-1} S(t) &\leq a+b\big(\frac{\theta}{\theta -1} a\big)^{\theta-1} \frac{\theta}{\theta -1} a
     \\&<a+b\frac{1}{\theta b}  \frac{\theta}{\theta -1} a=a+\frac{1}{\theta -1} a=\frac{\theta}{\theta -1}a.
     \end{align*}
 Thus,  
 \begin{align*} U=\left\{t\in [0,T] \quad\text{such that}\quad S(t)< \frac{\theta}{\theta -1} a\right\},
 \end{align*}
and by continuity of $S$, $U$ is open in $[0,T]$. This concludes the result.
\end{proof}
\begin{remark}\label{remarkdiscA2}
Let $\underline{v} = [(\varphi, \psi), \underline{h},\underline{x}, \underline{t}]$ be a linear concentrating wave satisfying
    \begin{align} \label{hyp2A2}
     \|(\varphi,\psi)\|_{\mathcal{E}_\infty}^2<\frac{2}{3} \|\nabla W\|_{L^2(\mathbb{R}^3)}^2.
    \end{align}
Replacing condition \eqref{3.111155555A2} with \eqref{hyp2A2} in Theorem \ref{proposisimiltheoreA2}, the proofs of \eqref{firstassertion1A2}, \eqref{thirdassertion2A2} and
\eqref{secondassertion3A2} in the cases $\underset{n\rightarrow +\infty}{\mathrm{lim}} \frac{t_n}{h_n}\in \{\pm \infty\}$ and $t_n=0$ carry over in the same way using suitable choices of the functions $(\varphi_i, \psi_i)$. More precisely:
\begin{itemize}
\item In the case $\underset{n\rightarrow +\infty}{\mathrm{lim}} \frac{t_n}{h_n}=+\infty$, we take
\begin{align*}
(\varphi_1, \psi_1) = (\varphi, \psi),\quad
(\varphi_2, \psi_2) = \Omega_-(\varphi, \psi), \quad \text{and} \quad
(\varphi_3, \psi_3) = \Omega_+^{-1} \circ \Omega_-(\varphi, \psi).
\end{align*}
    \item
    When $\underset{n\rightarrow +\infty}{\mathrm{lim}} \frac{t_n}{h_n}=-\infty$, we set
\begin{align*}
(\varphi_1, \psi_1) = \Omega_-^{-1} \circ \Omega_+(\varphi, \psi),\quad
(\varphi_2, \psi_2) = \Omega_+(\varphi, \psi), \quad \text{and} \quad
(\varphi_3, \psi_3) = (\varphi, \psi).
\end{align*}
\item If $t_n=0$ , we take
\begin{align*}
(\varphi_1, \psi_1) = \Omega_-^{-1}(\varphi, \psi), \quad
(\varphi_2, \psi_2) = (\varphi, \psi),\quad \text{and} \quad
(\varphi_3, \psi_3) = \Omega_+^{-1}(\varphi, \psi). 
\end{align*}
\end{itemize}
\end{remark}
\begin{proof}[Proof of Theorem \ref{proposisimiltheoreA2}]
In this proof, we focus on the case where $\underset{n\rightarrow +\infty}{\mathrm{lim}} \frac{t_n}{h_n}=-\infty$; the other cases proceed similarly, as discussed in Remark \ref{remarkdiscA2}. For clarity, we divide the proof into three parts.
\begin{enumerate}
\item \textbf{Part 1: proof of  \eqref{thirdassertion2A2}.}
\\ Let $\tau>0$ such that $\tau\in [0,T]\cap I_{\mathrm{max}}(u_n)$.
 Define $r_n:=u_n-v_n$. Applying Theorem \ref{thm1A2}, we obtain
  \begin{align} \label{striiA2}
      \|r_n\|_{L^5([0,\tau], L^{10}(\mathcal{M}))}&\leq C_{[0,T]}\|(r_n+v_n)^5\|_{L^1([0,\tau], L^2(\mathcal{M}))} \nonumber
      \\&\leq C_{[0,T]}\Big(\|v_n\|_{L^5([0,\tau], L^{10}(\mathcal{M}))}+\|r_n\|_{L^5([0,\tau], L^{10}(\mathcal{M}))}\Big)^5.
  \end{align}
Using Hölder inequality along with Lemma \ref{lemma3.9A2}, yields
  \begin{align*}
      \underset{n \rightarrow +\infty}{\mathrm{lim\;sup}}\; \|v_n\|_{L^5([0,T]), L^{10}(\mathcal{M}))}=0.
  \end{align*}
 Hence, there exists $n_0\in \mathbb{N}^*$ such that for all $n\geq n_0$, one has
  \begin{align} \label{defilimmA2}
     \delta_n:= \|v_n\|_{L^5([0,T], L^{10}(\mathcal{M}))}\leq \varepsilon.
  \end{align}
Set
 \begin{align*}
 S_n(t):= \|r_n\|_{L^5([0,t), L^{10}(\mathcal{M}))},\quad\quad \text{for all} \quad \quad  t\in[0,\tau].
\end{align*}
 From \eqref{striiA2}, it follows that
  \begin{align*}
      S_n(t)\leq C_{[0,T]}(\delta_n+S_n(t))^5\leq 16C_{[0,T]}( \delta_n^5+S_n(t)^5).\end{align*}
Then, by a bootstrap argument  (see Lemma \ref{bootsrapA2}), we conclude
\begin{align*}
   S_n(\tau)\leq 20C_{[0,T]}\delta_n^5, \quad  \quad \; \forall n\geq n_0,\; \forall \tau\in [0,T]\cap I_{\mathrm{max}}(u_n).
\end{align*}
If $T_+(u_n)\leq T$, we deduce
\begin{align*}
   S_n\big(T_+(u_n)\big)\leq 20C_{[0,T]}\delta_n^5, \quad  \quad \; \forall n\geq n_0,
\end{align*}
which gives a contradiction. Thus, $T_+(u_n)>T$ for $n$ sufficiently large and in view of \eqref{defilimmA2}, we get
\begin{align}\label{firstintervallA2}
\|r_n\|_{L^5([0,T], L^{10}(\mathcal{M}))}\leq \varepsilon, \quad  \quad \; \forall n\geq n_0.
\end{align}
By Lemma \ref{lemma3.9A2}, there exist $n_0, \Lambda_0>0$ such that
\begin{align} \label{4.66666newwA2}
 \tilde{\delta}_n:= \|v_n\|_{L^5([t_n+h_n\Lambda_0,0], L^{10}(\mathcal{M}))}\leq \varepsilon,\quad  \quad \; \forall n\geq n_0.
\end{align}
 For all $\tau\in [t_n+h_n\Lambda_0,0]\cap I_{\mathrm{max}}(u_n)$, set
\begin{align*}
 \tilde{S}_n(t):= \|r_n\|_{L^5((t,0], L^{10}(\mathcal{M}))},\quad \forall  t\in[\tau,0].
\end{align*}
Using \eqref{4.66666newwA2} and arguing similarly as above, one shows that, for all $\tau\in [t_n+h_n\Lambda_0,0]\cap I_{\mathrm{max}}(u_n)$
\begin{align*}
 \tilde{S}_n(\tau)\leq 20C_{[-T,0]}\tilde{\delta}_n^5, \quad  \quad \; \forall n\geq n_0.
\end{align*}
If $T_-(u_n)\geq t_n+h_n\Lambda_0$, then
\begin{align*}
\tilde{S}_n(T_-(u_n))\leq 20C_{[-T,0]}\tilde{\delta}_n^5, \quad  \quad \; \forall n\geq n_0,
\end{align*}
which leads to a contradiction. As a consequence, $T_-(u_n)< t_n+h_n\Lambda_0$ for $n$ sufficiently large and using \eqref{4.66666newwA2}, we conclude 
\begin{align}\label{seconfdintervallA2}
\|r_n\|_{L^5([t_n+h_n\Lambda_0,0], L^{10}(\mathcal{M}))}\leq \varepsilon, \quad  \quad \; \forall n\geq n_0.
\end{align}
Collecting \eqref{firstintervallA2} and \eqref{seconfdintervallA2} and using Strichartz estimates again, we deduce the desired estimate \eqref{thirdassertion2A2} with 
\begin{align*}
\underline{v}=[(\varphi,\psi),\underline{h},\underline{x},\underline{t}]=[(\varphi_3,\psi_3),\underline{h},\underline{x},\underline{t}].
\end{align*}
\item \textbf{Part 2: proof of \eqref{secondassertion3A2}.}\\
 Denote by $\tilde{u}_n$ the rescaled function associated to $u_n$, so that 
\begin{align*}
u_n(t,x)=\frac{1}{\sqrt{h_n}}\tilde{u}_n\Big(\frac{t-t_n}{h_n},\frac{x-x_n}{h_n}\Big)=\frac{1}{\sqrt{h_n}}\tilde{u}_n(s,y). 
\end{align*} 
Let $v_0$ be a solution to 
 \begin{equation}  \label{3.11166666666666reA2}
\left\{
\begin{array}{l}
\partial^2_s v_0-\Delta_{g(x_\infty)} v_0=0 \quad \text{on} \quad\mathbb{R}\times \mathbb{R}^3, \\
(v_0,\partial_s v_0)_{|s=0}=(\varphi,\psi),
\end{array}
\right.
\end{equation}
and $w$ satisfies \begin{equation} \label{3.117A2}
\left\{
\begin{array}{l}
\partial^2_sw-\Delta_{g(x_\infty)} w=w^5 \quad \text{on} \quad\mathbb{R}\times \mathbb{R}^3, \\
\underset{s\rightarrow + \infty}{\mathrm{lim}}\; \big\|\big(w-v_0,\partial_s(w-v_0)\big)\big\|_{\mathcal{E}_\infty}=0.
\end{array}
\right.
\end{equation}
Note that $w$ exists by Remark \ref{existwellA2}.
For all $\alpha>0$, denote by $w_{\alpha,\Lambda}$ the smooth solution of 
 \begin{equation} 
\left\{
\begin{array}{l}
\partial^2_s w_{\alpha,\Lambda}-\Delta_{g(x_\infty)} w_{\alpha,\Lambda}={w_{\alpha,\Lambda}}^5 \quad \text{on} \quad \mathbb{R}\times \mathbb{R}^3, \\
(w_{\alpha,\Lambda},\partial_s w_{\alpha,\Lambda})_{|s=\Lambda}=\chi_\alpha(v_0,\partial_s v_0)_{|s=\Lambda}\in (C^\infty_c(\mathcal{M}))^2,
\end{array}
\right.
\end{equation} 
where $\chi_\alpha:\mathcal{E}_\infty\rightarrow (C^\infty_c(\mathcal{M}))^2$ is a family of smoothing operators satisfying 
\begin{align*}
  \|\chi_\alpha(f,g)-(f,g)\|_{\mathcal{E}_\infty} \underset{\alpha \rightarrow+\infty}{\longrightarrow} 0.
\end{align*}
Note that $\chi_\alpha$ is well-defined due to the density of $(C^\infty_c(\mathcal{M}))^2$ in $\mathcal{E}_\infty$. For $\widetilde{T}>T_-(w) $,
one has \begin{align}\label{limiteinfinitoTA2}
 \underset{\alpha \rightarrow +\infty}{\mathrm{lim}}\; \|w_{\alpha,\Lambda}-w\|_{L^5([\widetilde{T},+\infty), L^{10}(\mathbb{R}^3))} \underset{\Lambda \rightarrow +\infty}{\longrightarrow}0.
\end{align}
Indeed, set $r_{\alpha,\Lambda}:=w_{\alpha,\Lambda}-w$, so that $r_{\alpha,\Lambda}$ satisfies
 \begin{equation} 
\left\{
\begin{array}{l}
\big(\partial_s^2-\Delta_{g(x_\infty)} \big) r_{\alpha,\Lambda}=(w_{\alpha,\Lambda})^5-w^5\\
\big(r_{\alpha,\Lambda}(\Lambda),\partial_s r_{\alpha,\Lambda}(\Lambda)\big)=\big(\chi_\alpha(v_0,\partial_s v_0)-(w,\partial_sw)\big)_{|s=\Lambda}.
\end{array}
\right.
\end{equation} 
Fix $\Lambda\geq\Lambda_0$ for some $\Lambda_0>0$ so that $\|w \|_{L^5([\Lambda,+\infty),L^{10}(\mathbb{R}^3))}\leq \varepsilon$  and let $\tau \in[\Lambda,+\infty)\cap [\Lambda, T_+(w_{\alpha,\Lambda}))$, using Strichartz inequality, Young's inequality then Hölder's inequality, we obtain 
\begin{align} \label{4.1111111sA2}
   \|r_{\alpha,\Lambda}&\|_{L^5([\Lambda,\tau], L^{10}(\mathbb{R}^3))} \nonumber
   \\& \leq C \Big( \Big\|\big(\chi_\alpha(v_0,\partial_s v_0)-(w,\partial_sw)\big)_{|s=\Lambda}\Big\|_{\dot{H} ^1(\mathbb{R}^3)\times L^2(\mathbb{R}^3)} \nonumber 
   \\&\quad+\|r_{\alpha,\Lambda}\|_{L^5([\Lambda,\tau],L^{10}(\mathbb{R}^3))}^5+\|w \|_{L^5([\Lambda,\tau],L^{10}(\mathbb{R}^3))}^4 \|r_{\alpha,\Lambda}\|_{L^5([\Lambda, \tau],L^{10}(\mathbb{R}^3))}\Big) \nonumber
   \\&\leq C \Big( \Big\|\big(\chi_\alpha(v_0,\partial_s v_0)-(w,\partial_sw)\big)_{|s=\Lambda}\Big\|_{\dot{H} ^1(\mathbb{R}^3)\times L^2(\mathbb{R}^3)} \nonumber 
   +\|r_{\alpha,\Lambda}\|_{L^5([\Lambda,\tau],L^{10}(\mathbb{R}^3))}^5\\&\quad\quad \quad+\varepsilon^4 \|r_{\alpha,\Lambda}\|_{L^5([\Lambda, \tau],L^{10}(\mathbb{R}^3))}\Big).
\end{align}
From \eqref{4.1111111sA2}, one gets for $\Lambda>\Lambda_0$ and $\tau \in[\Lambda,+\infty)\cap [\Lambda, T_+(w_{\alpha,\Lambda}))$,
\begin{align}
  \|r_{\alpha,\Lambda}\|_{L^5([\Lambda,\tau], L^{10}(\mathbb{R}^3))}&\leq 2C\Big( \Big\|\big(\chi_\alpha(v_0,\partial_s v_0)-(w,\partial_sw)\big)_{|s=\Lambda}\Big\|_{\dot{H} ^1(\mathbb{R}^3)\times L^2(\mathbb{R}^3)} \nonumber 
  \\&\quad\quad \quad\quad+\|r_{\alpha,\Lambda}\|_{L^5([\Lambda,\tau],L^{10}(\mathbb{R}^3))}^5\Big).
\end{align}
Due to \eqref{1.222222eqA2}, the definition of $\chi_\alpha$ and the second item in \eqref{3.117A2}, one deduces that there exists $\alpha_0(\Lambda)>0$ such that
\begin{align}\label{3.1225559888A2}
   \Big\|\big(\chi_\alpha(v_0,\partial_s v_0)-(w,\partial_sw)\big)_{|s=\Lambda}\Big\|_{\dot{H} ^1(\mathbb{R}^3)\times L^2(\mathbb{R}^3)}\leq \varepsilon, \;\;\forall \Lambda\geq\Lambda_0, \;  \forall \alpha\geq \alpha_0(\Lambda).
\end{align}
Suppose that $T_+(w_{\alpha,\Lambda})<+\infty$.
Then, by a bootstrap argument (see Lemma \ref{bootsrapA2}) 
\begin{align*}
  \|r_{\alpha,\Lambda}&\|_{L^5([\Lambda,T_+(w_{\alpha,\Lambda})), L^{10}(\mathbb{R}^3))} \lesssim \varepsilon, \;\;\quad \forall\Lambda\geq\Lambda_0, \;  \forall \alpha\geq \alpha_0(\Lambda),
\end{align*}
which gives a contradiction. Thus, $T_+(w_{\alpha,\Lambda})=+\infty$ for $\alpha$ and $\Lambda$ sufficiently large and 
\begin{align}\label{firstlimitA2}
   \underset{\alpha \rightarrow +\infty}{\mathrm{lim}}\;\|r_{\alpha,\Lambda}\|_{L^5([\Lambda,+\infty), L^{10}(\mathbb{R}^3))}\underset{\Lambda \rightarrow +\infty}{\longrightarrow}0.
\end{align}
Now, let $\tau\in [\widetilde{T},\Lambda]\cap (T_-(w_{\alpha,\Lambda}),\Lambda]$. By Strichartz estimates, one has
    \begin{align*}
   \|r_{\alpha,\Lambda}&\|_{L^5([\tau,\Lambda], L^{10}(\mathbb{R}^3))}
   \\& \leq C \Big( \Big\|\big(\chi_\alpha(v_0,\partial_s v_0)-(w,\partial_sw)\big)_{|s=\Lambda}\Big\|_{\dot{H} ^1(\mathbb{R}^3)\times L^2(\mathbb{R}^3)}  \\&\quad+\frac{5}{2}\|r_{\alpha,\Lambda}\|_{L^5([\tau,\Lambda],L^{10}(\mathbb{R}^3))}^5+\frac{5}{2}\Big\|\|r_{\alpha,\Lambda}\|_{L^{10}(\mathbb{R}^3))}\|w\|_{L^{10}(\mathbb{R}^3))}^4\Big\|_{L^1([\tau,\Lambda])}\Big).
   \end{align*}
Applying a Gronwall-type lemma (see Lemma 8.1 in \cite{17A2})
   \begin{align*}
    \|r_{\alpha,\Lambda}&\|_{L^5([\tau,\Lambda], L^{10}(\mathbb{R}^3))}
   \\& \leq C \Big( \big\|\big(\chi_\alpha(v_0,\partial_s v_0)-(w,\partial_sw)\big)_{|s=\Lambda}\big\|_{\dot{H} ^1(\mathbb{R}^3)\times L^2(\mathbb{R}^3)}  \\&\quad+\frac{5}{2}\|r_{\alpha,\Lambda}\|_{L^5([\tau,\Lambda],L^{10}(\mathbb{R}^3))}^5\Big)\Phi\Big(C\frac{5}{2}\|w\|_{L^5([\tau,\Lambda],L^{10}(\mathbb{R}^3))}^4\Big),
\end{align*}
where $\Phi(s)=2\Gamma(3+2s)$ and $\Gamma$ is the Gamma function. If $T_-(w_{\alpha,\Lambda})\geq \widetilde{T}$, then arguing similarly by using \eqref{3.1225559888A2} and a bootstrap argument, we deduce 
\begin{align*}
    \|r_{\alpha,\Lambda}&\|_{L^5((T_-(w_{\alpha,\Lambda}),\Lambda]), L^{10}(\mathbb{R}^3))} \lesssim \varepsilon, \;\;\quad \forall \Lambda\geq\Lambda_0, \; \forall \alpha\geq\alpha_0(\Lambda),
\end{align*}
leading to a contradiction and as a consequence $T_-(w_{\alpha,\Lambda})<\widetilde{T}$ for $\alpha$ and $\Lambda$ sufficiently large and 
\begin{align}\label{secondlimitA2}
   \underset{\alpha \rightarrow +\infty}{\mathrm{lim}}\;\|r_{\alpha,\Lambda}\|_{L^5([\tilde{T},\Lambda], L^{10}(\mathbb{R}^3))}\underset{\Lambda \rightarrow +\infty}{\longrightarrow}0.
\end{align}
Combining \eqref{firstlimitA2} and \eqref{secondlimitA2}, we conclude \eqref{limiteinfinitoTA2} for all $\widetilde{T}> T_-(w)$. By the energy estimate applied on the time intervals $[\Lambda, +\infty)$ and $[\widetilde{T},\Lambda]$, respectively, and using \eqref{3.1225559888A2} with \eqref{firstlimitA2} for the first one and \eqref{3.1225559888A2} with \eqref{secondlimitA2} for the second one, we obtain
\begin{align*}
&\underset{\alpha\rightarrow +\infty}{\mathrm{lim}}\; \underset{s\in [\Lambda,+\infty)}{\mathrm{sup}}\;\big\|\big(w_{\alpha,\Lambda}(s)-w(s),\partial_s(w_{\alpha,\Lambda}-w)(s)\big)\big\|_{\mathcal{E}_\infty}^2\underset{\Lambda \rightarrow + \infty}{\longrightarrow}0,
\\&
   \underset{\alpha\rightarrow +\infty}{\mathrm{lim}}\; \underset{s\in [\widetilde{T},\Lambda]}{\mathrm{sup}}\;\big\|\big(w_{\alpha,\Lambda}(s)-w(s),\partial_s(w_{\alpha,\Lambda}-w)(s)\big)\big\|_{\mathcal{E}_\infty}^2\underset{\Lambda \rightarrow + \infty}{\longrightarrow}0.
\end{align*}
Thus, due to \eqref{1.3A2} 
\begin{align}\label{3.123333333A2}
   \underset{\alpha\rightarrow +\infty}{\mathrm{lim}}\; |||w_{\alpha,\Lambda}-w|||_{[\widetilde{T},+\infty)}\underset{\Lambda \rightarrow + \infty}{\longrightarrow}0, \quad \text{for all} \quad \widetilde{T}>T_-(w).
\end{align}
We want to prove 
\begin{align}\label{4.1166666666A2}
   \underset{n \rightarrow +\infty}{\mathrm{lim\;sup}}\; |||\tilde{u}_n-w|||_{[\widetilde{T},\Lambda]}=0,\quad \forall \Lambda >0, \;\forall \widetilde{T}>T_-(w).
\end{align}
In view of \eqref{3.123333333A2} and since $T_-(w_{\alpha,\Lambda})<\widetilde{T}$ for $\alpha$ and $\Lambda$ sufficiently large, it is enough to show that
\begin{align} \label{4.11888888A2}
\underset{\alpha \rightarrow +\infty}{\mathrm{lim}}\; \underset{n \rightarrow +\infty}{\mathrm{lim\;sup}}\; |||\tilde{u}_n-w_{\alpha,\Lambda}|||_{[\widetilde{T},\Lambda]}\underset{\Lambda \rightarrow + \infty}{\longrightarrow}0.
\end{align}
For that, we introduce the auxiliary family of functions $\tilde{u}_n^\Lambda$ defined by
\begin{equation} 
\left\{
\begin{array}{l}
\partial^2_s \tilde{u}_n^\Lambda-\Delta_{g(h_n.+x_n)} \tilde{u}_n^\Lambda=({\tilde{u}_n}^\Lambda)^5 \quad \text{on} \quad [-\Lambda, \Lambda]\times \mathbb{R}^3, \\
(\tilde{u}_n^\Lambda,\partial_s \tilde{u}_n^\Lambda)_{|s=\Lambda}=(\tilde{v}_n,\partial_s \tilde{v}_n)_{|s=\Lambda},
\end{array}
\right.
\end{equation} 
where $\tilde{v}_n$ denotes the rescaled function associated to $v_n$.
 To prove \eqref{4.11888888A2}, it suffices to show that
\begin{align}
    &\underset{\alpha \rightarrow +\infty}{\mathrm{lim}}\; \underset{n \rightarrow +\infty}{\mathrm{lim\;sup}}\;  |||\tilde{u}_n^\Lambda-w_{\alpha,\Lambda}|||_{[\widetilde{T},\Lambda]}\underset{\Lambda \rightarrow+\infty}{\longrightarrow} 0, \label{4.51thirdassertionnnnA2}
    \\&\underset{\alpha \rightarrow +\infty}{\mathrm{lim}}\; \underset{n \rightarrow +\infty}{\mathrm{lim\;sup}}\; |||\tilde{u}_n-\tilde{u}_n^\Lambda|||_{[\widetilde{T},\Lambda]}\underset{\Lambda \rightarrow+\infty}{\longrightarrow} 0.\label{4.52thirdassertionnnnnnA2}
\end{align}
We begin with \eqref{4.51thirdassertionnnnA2}. By \eqref{limiteinfinitoTA2}, fix $\Lambda\geq \Lambda_0$ for some $\Lambda_0>0$  and $\alpha \geq \alpha_0(\Lambda)$ so that 
\begin{align}\label{490000sA2}
\|w_{\alpha,\Lambda}\|_{L^5([\widetilde{T},\Lambda],L^{10}(\mathbb{R}^3))}\leq \varepsilon +\|w\|_{L^5([\widetilde{T},\Lambda],L^{10}(\mathbb{R}^3))}.
\end{align}
Let $\eta<\Lambda$ sufficiently close to $\Lambda$ such that 
\begin{align}\label{4911111sA2}
\|w\|_{L^5([\eta,\Lambda],L^{10}(\mathbb{R}^3))}\leq \varepsilon.
\end{align}
 Let $\tau\in [\eta,\Lambda]\cap (T_-(\tilde{u}_n^\Lambda),\Lambda]$ and set $r_{n,\alpha}^\Lambda:=\tilde{u}_n^\Lambda-w_{\alpha,\Lambda}$. Note that $r_{n,\alpha}^\Lambda$ verifies
\begin{equation} 
\left\{
\begin{array}{l}
(\partial^2_s -\Delta_{g(h_n.+x_n)} )r_{n,\alpha}^\Lambda =(r_{n,\alpha}^\Lambda+w_{\alpha,\Lambda})^5-(w_{\alpha,
\Lambda})^5+(\Delta_{g(h_n.+x_n)}-\Delta_{g(x_\infty)})w_{\alpha,\Lambda}\\
(r_{n,\alpha}^\Lambda,\partial_s r_{n,\alpha}^\Lambda)_{|s=\Lambda}=\Big(\big(\tilde{v}_n,\partial_s \tilde{v}_n\big)-\chi_\alpha \big(v_0,\partial_sv_0\big)\Big)_{|s=\Lambda},
\end{array}
\right.
\end{equation} 
where one can easily show that the operator $\partial^2_s-\Delta_{g(h_n.+x_n)}$ satisfies the
same Strichartz and energy estimates as $\partial^2_t-\Delta_g$ with a time of order $\Lambda$. Using Theorem \ref{thm1A2}, Young's inequality and Hölder's inequality, we obtain
\begin{align} \label{rnlambdddddA2}
    &\|r_{n,\alpha}^\Lambda\|_{L^5([\tau,\Lambda],L^{10}(\mathcal{M}))}\nonumber
    \leq C_{[\widetilde{T},\Lambda]}\Big( \big\|\big((\tilde{v}_n,\partial_s \tilde{v}_n)-\chi_\alpha (v_0,\partial_sv_0)\big)_{|s=\Lambda}\big\|_{\mathcal{E}}  \nonumber
    \\& \quad\; +\frac{5}{2}\|w_{\alpha,\Lambda}\|_{L^5([\tau,\Lambda], L^{10}(\mathcal{M}))}^4\|r_{n,\alpha}^\Lambda\|_{L^5([\tau,\Lambda], L^{10}(\mathcal{M}))} \nonumber
    \\&\quad\;
   +\frac{5}{2}\|r_{n,\alpha}^\Lambda\|_{L^5([\tau,\Lambda], L^{10}(\mathcal{M}))}^5+\|(\Delta_{g(h_n.+x_n)}-\Delta_{g(x_\infty)})w_{\alpha,\Lambda}\|_{L^1([\tau,\Lambda], L^2(\mathcal{M}))}\Big).
\end{align}
From \eqref{490000sA2} and \eqref{4911111sA2}, one gets
\begin{align*}
\|w_{\alpha,\Lambda}\|_{L^5([\tau,\Lambda],L^{10}(\mathbb{R}^3))}\leq 2\varepsilon \quad \forall \Lambda\geq \Lambda_0, \forall \alpha \geq \alpha_0(\Lambda).
\end{align*} Then, by \eqref{rnlambdddddA2}, we deduce
\begin{align}\label{rnLambdaA2}
\|&r_{n,\alpha}^\Lambda\|_{L^5([\tau,\Lambda],L^{10}(\mathcal{M}))}\nonumber
    \leq 2C_{[\tilde{T},\Lambda]}\Big( \big\|\big((\tilde{v}_n,\partial_s \tilde{v}_n)-\chi_\alpha (v_0,\partial_sv_0)\big)_{|s=\Lambda}\big\|_{\mathcal{E}}  \nonumber\\&
   +\frac{5}{2}\|r_{n,\alpha}^\Lambda\|_{L^5([\tau,\Lambda], L^{10}(\mathcal{M}))}^5+\|(\Delta_{g(h_n.+x_n)}-\Delta_{g(x_\infty)})w_{\alpha,\Lambda}\|_{L^1([\tau,\Lambda], L^2(\mathcal{M}))}\Big).
\end{align}
By the finite speed of propagation, $w_{\alpha,\Lambda}$ has a compact support contained in a compact set $K$. Therefore, by Lemma \ref{lemmm3.7A2} there exists $n_0(\Lambda)\in \mathbb{N}^*$ such that
\begin{align} \label{4.54 wLambdaA2}
    \|(\Delta_{g(h_n.+x_n)}-\Delta_{g(x_\infty)}) w_{\alpha,\Lambda}\|_{L^1([\tau,\Lambda],L^2(\mathcal{M}))}\leq \varepsilon,\quad \forall n\geq n_0(\Lambda).
    \end{align}
Moreover, by the definition of $\chi_\alpha$ and thanks to Lemma \ref{lemmmmm3.8A2}, one has
\begin{align} \label{4.55v0A2}
  \big\|\big((\tilde{v}_n,\partial_s \tilde{v}_n)-\chi_\alpha (v_0,\partial_sv_0)\big)_{|s=\Lambda}\big\|_{\mathcal{E}} \leq \varepsilon, \quad \forall \Lambda\geq \Lambda_0, \forall \alpha>\alpha_0(\Lambda),\forall n\geq n_0(\Lambda).
\end{align}
Suppose that $T_-( \tilde{u}_n^\Lambda)\geq \eta$. Collecting \eqref{rnLambdaA2}, \eqref{4.54 wLambdaA2} and \eqref{4.55v0A2} and using a bootstrap argument (see Lemma \ref{bootsrapA2}), we obtain
\begin{align}\label{4.57bootstA2}
    \|r_{n,\alpha}^\Lambda\|_{L^5((T_-( \tilde{u}_n^\Lambda),\Lambda],L^{10}(\mathcal{M}))}\lesssim \varepsilon \quad \forall \Lambda\geq \Lambda_0, \forall n\geq n_0(\Lambda),\forall \alpha>\alpha_0(\Lambda),
\end{align}
leading to a contradiction. Hence, $T_-( \tilde{u}_n^\Lambda)<\eta$ for $n$ and $\Lambda$  sufficiently large and 
\begin{align*}
 \underset{\alpha \rightarrow + \infty}{\mathrm{lim}}\;  \underset{n\rightarrow+\infty}{\mathrm{lim\;sup}} \;\|r_{n,\alpha}^\Lambda\|_{L^5([\eta,\Lambda],L^{10}(\mathcal{M}))}\underset{\Lambda \rightarrow+\infty}{\longrightarrow} 0.
 \end{align*}
This concludes \eqref{4.51thirdassertionnnnA2} on $[\eta,\Lambda]$. Then, we iterate the process by dividing the interval  $[\widetilde{T}, \Lambda]$ into a finite number of intervals where we can use the bootstrap
argument leading to $T_-( \tilde{u}_n^\Lambda)<\widetilde{T}$ for $n$ and $\Lambda$  sufficiently large and \eqref{4.51thirdassertionnnnA2} holds.\\
Now, to prove \eqref{4.52thirdassertionnnnnnA2}, fix $\Lambda\geq\Lambda_0$, $n\geq n_0(\Lambda)$ and $\alpha>\alpha_0(\Lambda)$ so that 
\begin{align}\label{499988888sA2}
\|\tilde{u}_n^\Lambda\|_{L^5([\tilde{T},\Lambda], L^{10}(\mathcal{M}))}\leq \varepsilon+ \|w_{\alpha,\Lambda}\|_{L^5([\tilde{T},\Lambda], L^{10}(\mathcal{M}))}.
\end{align}
Let $\eta<\Lambda$ sufficiently close to $\Lambda$ such that 
\begin{align}\label{4999999sA2}
\|w_{\alpha,\Lambda}\|_{L^5([\eta,\Lambda],L^{10}(\mathbb{R}^3))}\leq \varepsilon.
\end{align}
 Let $\tau \in [\eta,\Lambda]\cap I_{\mathrm{max}}(\tilde{u}_n)$. Denoting $\tilde{r}_{n,\Lambda}:=\tilde{u}_n-\tilde{u}_n^\Lambda$ and applying Theorem \ref{thm1A2} along with Young's inequality and Hölder's inequality, we get
\begin{align} \label{5.68888888A2}
      \|&\tilde{r}_{n,\Lambda}\|_{L^5([\tau,\Lambda],L^{10}(\mathcal{M}))}\nonumber
    \\&
    \leq C_{[\tilde{T},\Lambda]}\Big( \|\big(\tilde{u}_n-\tilde{v}_n,\partial_s (\tilde{u}_n-\tilde{v}_n)\big)_{|s=\Lambda}\|_{\mathcal{E}} \nonumber
   +\frac{5}{2}\|\tilde{r}_{n,\Lambda}\|_{L^5([\tau,\Lambda], L^{10}(\mathcal{M}))}^5
   \\&
  \quad\quad\quad\hspace{2,5cm}+ \frac{5}{2}\|\tilde{u}_n^\Lambda\|_{L^5([\tau,\Lambda], L^{10}(\mathcal{M}))}^4\|\tilde{r}_{n,\Lambda}\|_{L^5([\tau,\Lambda], L^{10}(\mathcal{M}))}\Big) \nonumber
  \\& \leq C_{[\tilde{T},\Lambda]}\Big( \|\big(\tilde{u}_n-\tilde{v}_n,\partial_s (\tilde{u}_n-\tilde{v}_n)\big)_{|s=\Lambda}\|_{\mathcal{E}} 
   +\frac{5}{2}\|\tilde{r}_{n,\Lambda}\|_{L^5([\tau,\Lambda], L^{10}(\mathcal{M}))}^5\Big).
\end{align}
Thanks to \eqref{thirdassertion2A2}, one has
\begin{align} \label{5.6999999A2}
    \|\big(\tilde{u}_n-\tilde{v}_n,\partial_s (\tilde{u}_n-\tilde{v}_n)\big)_{|s=\Lambda}\|_{\mathcal{E}}\leq \varepsilon,\quad \forall \Lambda\geq \Lambda_0, \forall n\geq n_0(\Lambda).
\end{align}
 If $T_-(\tilde{u}_n)\geq \eta$, collecting \eqref{499988888sA2}, \eqref{4999999sA2}, \eqref{5.68888888A2} and \eqref{5.6999999A2}, a bootstrap argument (see Lemma \ref{bootsrapA2}) leads to a contradiction. Thus, $T_-(\tilde{u}_n)< \eta$ for $n$ sufficiently large and \eqref{4.52thirdassertionnnnnnA2} holds on $[\eta,\Lambda]$. We then iterate the process by dividing the interval $[\widetilde{T},\Lambda]$ into a finite number of intervals where we can apply the bootstrap argument, thereby concluding  \eqref{4.52thirdassertionnnnnnA2} and $T_-(\tilde{u}_n)< \widetilde{T}$ for $n$ sufficiently large.\\ From \eqref{3.111155555A2} and \eqref{4.1166666666A2} with large $n$, we obtain that for any $\Lambda>0$, there exists $C>0$ such that 
\begin{align*}
 \|w\|_{L^5([\widetilde{T},\Lambda],L^{10}(\mathcal{M}))}< C, \quad \text{for all} \quad \widetilde{T} > T_-(w).
\end{align*}
Due to \eqref{1.3A2} and the blow-up criterion,
we deduce that $T_-(w)=-\infty$ and  there exists $v^\infty_-$ satisfying
\begin{equation} \label{scatterssolutionexA2}
\left\{
\begin{array}{l}
\big(\partial_s^2-\Delta_{g(x_\infty)} \big) v^\infty_-=0 \quad \text{on} \quad\mathbb{R}\times \mathbb{R}^3,\\
\underset{s\rightarrow - \infty}{\mathrm{lim}}\; \big\|\big(w-v^\infty_-,\partial_s(w-v^\infty_-)\big)\big\|_{\mathcal{E}_\infty}=0.
\end{array}
\right.
\end{equation}
Then, define
\begin{align*}
 (w,\partial_sw)_{|s=0}=(\varphi_2,\psi_2).
 \end{align*}
Hence, from \eqref{4.1166666666A2} we deduce \eqref{secondassertion3A2}. 
\item \textbf{Part 3: proof of  \eqref{firstassertion1A2}:}\\
It remains to show \eqref{firstassertion1A2}.
For that, define the function $v_{n,-}^\infty$ satisfying 
\begin{equation} 
\left\{
\begin{array}{l}
\partial^2_sv_{n,-}^\infty -\Delta_{g(h_n.+x_n)} v_{n,-}^\infty=0 \quad \text{on} \quad\mathbb{R}\times \mathcal{M},\\
(v_{n,-}^\infty,\partial_s v_{n,-}^\infty)_{|s=0}=(v^\infty_-,\partial_s v^\infty_-)_{|s=0}=(\varphi_1,\psi_1).
\end{array}
\right.
\end{equation}
By the definition of $\chi_\alpha$ and the second element in \eqref{scatterssolutionexA2}, one has
\begin{align} \label{4.88888888sssA2}
  \underset{\alpha\rightarrow +\infty}{\mathrm{lim}}  \;\big\|\big((w,\partial_s w)-&\chi_\alpha (v^\infty_-,\partial_sv^\infty_-)\big)_{s=-\Lambda}\big\|_{\mathcal{E}_\infty} \nonumber\\&= \big\|\big(v^\infty_--w,\partial_s(v^\infty_--w)\big)_{s=-\Lambda}\big\|_{\mathcal{E}_\infty} \underset{\Lambda\rightarrow +\infty}{\longrightarrow} 0.
\end{align}
Thus, using \eqref{3.123333333A2}, the fact that $T_-(w)=-\infty$, and \eqref{4.88888888sssA2}, we deduce
\begin{align}\label{4.633A2}
 \underset{\alpha\rightarrow +\infty}{\mathrm{lim}}\;\big\|\big((w_{\alpha,\Lambda},\partial_s w_{\alpha,\Lambda})-\chi_\alpha (v^\infty_-,\partial_sv^\infty_-)\big)_{s=-\Lambda}\big\|_{\mathcal{E}}\underset{\Lambda\rightarrow +\infty}{\longrightarrow} 0.
  \end{align}
Observe that
\begin{align} \label{4.6222A2}
    \big\|\big(\tilde{u}_n-v_{n,-}^\infty &,\partial_s (\tilde{u}_n-v_{n,-}^\infty )\big)_{|s=-\Lambda}\big\|_{\mathcal{E}}\nonumber
    \\&\leq  \big\|\big(\tilde{u}_n-w_{\alpha,\Lambda} ,\partial_s (\tilde{u}_n-w_{\alpha,\Lambda} )\big)_{|s=-\Lambda}\big\|_{\mathcal{E}}\nonumber
    \\&\quad+ \big\|\big((w_{\alpha,\Lambda},\partial_s w_{\alpha,\Lambda})-\chi_\alpha (v^\infty_-,\partial_sv^\infty_-)\big)_{s=-\Lambda}\big\|_{\mathcal{E}} \nonumber
    \\&\quad
    +\big\|\big(\chi_\alpha (v^\infty_-,\partial_sv^\infty_-)-(v_{n,-}^\infty,\partial_s v_{n,-}^\infty)\big)_{s=-\Lambda}\big\|_{\mathcal{E}}.
\end{align}
Furthermore, \eqref{4.11888888A2} and $T_-(w)=-\infty$ yield
\begin{align}\label{4.631A2}
   \underset{\alpha\rightarrow +\infty}{\mathrm{lim}}\;  \underset{n \rightarrow +\infty}{\mathrm{lim\;sup}}\;  \big\|\big(\tilde{u}_n-w_{\alpha,\Lambda} ,\partial_s (\tilde{u}_n-w_{\alpha,\Lambda} )\big)_{|s=-\Lambda}\big\|_{\mathcal{E}}\underset{\Lambda \rightarrow+\infty}{\longrightarrow} 0.
\end{align}
Thanks to Lemma \ref{lemmmmm3.8A2}, one has
\begin{align}\label{4.632A2}
   \underset{\alpha\rightarrow +\infty}{\mathrm{lim}}\;  \underset{n \rightarrow +\infty}{\mathrm{lim\;sup}}\;  \big\|\big(\chi_\alpha (v^\infty_-,\partial_sv^\infty_-)-(v_{n,-}^\infty,\partial_s v_{n,-}^\infty)\big)_{s=-\Lambda}\big\|_{\mathcal{E}}  \underset{\Lambda \rightarrow+\infty}{\longrightarrow} 0.
\end{align}
Collecting \eqref{4.633A2}, \eqref{4.6222A2}, \eqref{4.631A2} and \eqref{4.632A2}, we deduce
\begin{align}\label{4.66ffinA2}
   \underset{\alpha\rightarrow +\infty}{\mathrm{lim}}\; \underset{n \rightarrow +\infty}{\mathrm{lim\;sup}}\;  \big\|\big(\tilde{u}_n-v_{n,-}^\infty &,\partial_s (\tilde{u}_n-v_{n,-}^\infty )\big)_{|s=-\Lambda}\big\|_{\mathcal{E}} \underset{\Lambda \rightarrow +\infty}{\longrightarrow} 0.
\end{align}
Let $\tau<0$ be such that $\tau\in [\frac{-T-t_n}{h_n},-\Lambda]\cap I_{\mathrm{max}}(\tilde{u}_n)$. Applying Theorem \ref{thm1A2} for $r_n^\infty:=\tilde{u}_n-v_{n,-}^\infty$, we obtain
\begin{align*}
  \|r_n^\infty\|_{L^5([\tau,-\Lambda],L^{10}(\mathcal{M}))}&\leq C_{[-T,-\Lambda]}\Big(\big\|\big(\tilde{u}_n-v_{n,-}^\infty ,\partial_s (\tilde{u}_n-v_{n,-}^\infty)\big)_{|s=-\Lambda}\big\|_{\mathcal{E}} \nonumber
      \\&\quad+ \big(\|v_{n,-}^\infty\|_{L^5([\tau,-\Lambda], L^{10}(\mathcal{M}))}+\|r_n^\infty\|_{L^5([\tau,-\Lambda], L^{10}(\mathcal{M}))}\big)^5\Big).
\end{align*}
Hence, arguing similarly as in the first part above by using Lemma \ref{lemma3.9A2} and a bootstrap argument along with \eqref{4.66ffinA2}, we conclude the desired estimate \eqref{firstassertion1A2} and we finish the proof of Theorem \ref{proposisimiltheoreA2}.
\end{enumerate}
\end{proof}
\section{Bound from below of the energy norm for type II blow-up solutions} \label{section55A2}
Here, we prove the last two estimates, (\ref{1.25***A2}) and (\ref{1.26***A2}).
We recall that our analysis is carried out on the manifold $\mathcal{M}=\mathbb{R}^3$ equipped with a metric $g(x)=(g_{i,j}(x))_{1\leq i,j\leq 3}$, where $g_{i,j}(x)=\delta_{i,j}$ for all $|x|>R$. We recall that $\mathcal{E}$ denotes the energy space $\dot{H}^1(\mathcal{M})\times L^2(\mathcal{M})$. We will assume throughout this section without loss of generality that the blow-up time is $T_+(u) = 1$.
\subsection{Preliminaries} \label{preelemmmA2}
In this subsection, we present some preliminary results needed for the proof of assertion (4) in our main result. We begin by stating a long-time perturbation theory for \eqref{eq1A2}.
\begin{proposition}[Long time perturbation theory] \label{longtimeA2}
    Let $M>0$. Let $I$ be a time interval and $t_0\in I$. There exist constants $C(I,M)>0$ and $\varepsilon_0(M)>0$ with the following property. Let $\tilde{u}$ be an approximate solution to \eqref{eq1A2} on $I\times \mathcal{M}$ in the sense that
    \begin{equation} 
\left\{
\begin{array}{l}
 (\partial_t^2-\Delta_g)\tilde{u} =\tilde{u}^5+e,\quad (t,x)\in I\times \mathcal{M}, \ \\[10pt]
\tilde{u}_{|t=t_0}= \tilde{u}_0\in \dot{H}^1(\mathcal{M}) , \quad \partial_t\tilde{u}_{|t=t_0}= \tilde{u}_1 \in L^2(\mathcal{M}),
\end{array}
\right.
\end{equation}
    for some function $e$. Assume that \begin{align}
       & \|\tilde{u}\|_{L^5(I,L^{10}(\mathcal{M}))}\leq M, \label{5.8111111A2}
       \\&\|e\|_{L^1(I,L^2(\mathcal{M}))}\leq \varepsilon,\label{5.888222222A2}
    \end{align}
    for some constant $\varepsilon$ with $0<\varepsilon< \varepsilon_0(M)$. Let $u$ be a solution to \eqref{eq1A2} with initial data $(u_0,u_1)\in\dot{H}^1(\mathcal{M})\times L^2(\mathcal{M})$ at $t=t_0$ such that
    \begin{align} \label{5.833333A2}
        \|(u_0-\tilde{u}_0,u_1-\tilde{u}_1)\|_{\mathcal{E}}\leq \varepsilon.
    \end{align}
 Then, $ I \subset I_{\mathrm{max}}(u)$ and 
    \begin{align}
        \|u-\tilde{u}\|_{L^5(I,L^{10}(\mathcal{M}))}\leq C(I,M) \varepsilon.
    \end{align}
\end{proposition}
\begin{proof}[Proof]
 Let $K$ be a compact interval of $I\cap I_{\mathrm{max}}(u)$ containing $t_0$. Set $w:=u-\tilde{u}$. Note that $w$ satisfies 
     \begin{equation} 
\left\{
\begin{array}{l}
 (\partial_t^2-\Delta_g)w =(w+\tilde{u})^5-\tilde{u}^5-e,\quad (t,x)\in K\times \mathcal{M}, \ \\[10pt]
w_{|t=t_0}= u_0-\tilde{u}_0\in \dot{H}^1(\mathcal{M}) , \quad \partial_tw_{|t=t_0}= u_1-\tilde{u}_1 \in L^2(\mathcal{M}).
\end{array}
\right.
\end{equation}
Applying respectively Theorem \ref{thm1A2}, \eqref{5.888222222A2} and \eqref{5.833333A2}, one gets
\begin{align} \label{inestrichA2}
    \|w\|_{L^5(K,L^{10}(\mathcal{M}))}&\leq   C_K\Big(\|u_0-\tilde{u}_0\|_{\dot{H}^1(\mathcal{M})}+\|u_1-\tilde{u}_1\|_{L^2(\mathcal{M})}\nonumber\\&\quad\;+\|(w+\tilde{u})^5-\tilde{u}^5\|_{L^1(K,L^2(\mathcal{M}))}+\|e\|_{L^1(K,L^2(\mathcal{M}))}\Big) \nonumber
    \\& \leq C_K\Big(2\varepsilon+\|(w+\tilde{u})^5-\tilde{u}^5\|_{L^1(K,L^2(\mathcal{M}))}\Big).
\end{align}
By Young's inequality, one writes
\begin{align*}
|(w+\tilde{u})^5-\tilde{u}^5|\leq \frac{5}{2} |w||w^4+\tilde{u}^4|.
\end{align*}
Thus, using Hölder's inequality, we obtain
\begin{align} \label{5.87777A2}
\|(w+\tilde{u})^5-\tilde{u}^5\|_{L^1(K,L^2(\mathcal{M}))}&\leq \frac{5}{2}\Big\|\|w\|_{L^{10}(\mathcal{M})}\|w^4+\tilde{u}^4\|_{L^{\frac{10}{4}}(\mathcal{M})}\Big\|_{L^1(K)}\nonumber
\\&\leq \frac{5}{2} \Big(\|w\|_{L^5(K,L^{10}(\mathcal{M}))}^5+\Big\|\|w\|_{L^{10}(\mathcal{M})}\|\tilde{u}\|_{L^{10}(\mathcal{M})}^4\Big\|_{L^1(K)}\Big).
\end{align}
Inserting \eqref{5.87777A2} in \eqref{inestrichA2} then applying Gronwall-type lemma (see Lemma 8.1 in \cite{17A2}) and \eqref{5.8111111A2}, yields
\begin{align*}
 \|w\|_{L^5(K,L^{10}(\mathcal{M}))} &\leq C_K\Big(2\varepsilon+\frac{5}{2} \|w\|_{L^5(K,L^{10}(\mathcal{M}))}^5+\frac{5}{2}\Big\|\|w\|_{L^{10}(\mathcal{M})}\|\tilde{u}\|_{L^{10}(\mathcal{M})}^4\Big\|_{L^1(K)}\Big)
 \\&\leq C_K\Big(2\varepsilon+\frac{5}{2}\|w\|_{L^5(K,L^{10}(\mathcal{M}))}^5\Big) \Phi\Big(C_K\frac{5}{2}M^4\Big),
\end{align*}
where $\Phi(s)=2\Gamma(3+2s)$ and $\Gamma$ is the Gamma function. Hence, assuming that $\|w\|_{L^5(K,L^{10}(\mathcal{M}))}^5\leq \varepsilon$, we deduce that for every compact interval $K$ of $I\cap I_{\mathrm{max}}(u)$,
\begin{align*}
    \|w\|_{L^5(K,L^{10}(\mathcal{M}))} \leq C_K \frac{9}{2}\Phi\Big(C_K\frac{5}{2}M^4\Big)\varepsilon\leq \frac{1}{2}\varepsilon^\frac{1}{5},
\end{align*}
if $\varepsilon$ is chosen small enough. By the blow-up criterion, $I \subset I_{\mathrm{max}}(u)$, which concludes the proof.
\end{proof}
\begin{theorem}[Profile decomposition]\label{camillelaurentA2}
 Let ${(u_{0,n} , u_{1,n} )}_n$ be a bounded sequence in $\mathcal{E}$. Then, up to extraction of a subsequence, one has the following profile decomposition: for any $\ell \in \mathbb{N}^*$
\begin{equation}\label{prodeA2}
\left\{
\begin{array}{l}
u_{0,n}=\sum\limits_{j=1}^\ell\frac{1}{\sqrt{h_n^{(j)}}} \tilde{p}_n^{(j)}\Big(\frac{-t_n^{(j)}}{h_n^{(j)}}, \frac{x-x_n^{(j)}}{h_n^{(j)}}\Big)+w^{(\ell)}_{0,n}(x), \ \\[10pt]
u_{1,n}=\sum\limits_{j=1}^\ell (\frac{1}{h_n^{(j)}})^\frac{3}{2} \partial_t \tilde{p}_n^{(j)}\Big(\frac{-t_n^{(j)}}{h_n^{(j)}}, \frac{x-x_n^{(j)}}{h_n^{(j)}}\Big)+w^{(\ell)}_{1,n}(x),
\end{array}
\right.
\end{equation}
where 
\begin{align*}
\underline{p}^{(j)}:=(p_n^{(j)})_n=\Big(\frac{1}{\sqrt{h_n^{(j)}}} \tilde{p}_n^{(j)}\Big(\frac{t-t_n^{(j)}}{h_n^{(j)}},\frac{x-x_n^{(j)}}{h_n^{(j)}}\Big)\Big)_n=[(\varphi^{(j)},\psi^{(j)}),\underline{h}^{(j)},\underline{x}^{(j)},\underline{t}^{(j)}]
\end{align*}
is the linear concentrating wave, as defined in Definition \ref{def3.5A2}, and $w_n^{(\ell)}$ denotes the solution to $\partial_t^2w_n^{(\ell)}-\Delta_g w_n^{(\ell)}=0$ with initial data $(w^{(\ell)}_{0,n},w^{(\ell)}_{1,n})$ and satisfies
\begin{align}\label{3.9666A2}
     \underset{\ell\rightarrow +\infty}{\mathrm{lim}}\;\underset{n\rightarrow+\infty}{\mathrm{lim}\;\mathrm{sup}}\;||w^{(\ell)}_n ||_{L^\infty([-T,T ],L^6(\mathcal{M}))\cap L^5([-T,T ],L^{10}(\mathcal{M}))}=0, \quad \forall \;T>0.
\end{align}
Moreover, we have the following Pythagorean expansion
\begin{align}\label{pythagorean expansionA2}
    \|(u_{0,n} , u_{1,n} )\|_{\mathcal{E}}^2=\sum\limits_{j=1}^\ell\|(p_n^{(j)},\partial_tp_n^{(j)})_{|t=0}\|_{\mathcal{E}}^2+\|(w_{0,n}^{(\ell)},w_{1,n}^{(\ell)})\|_{\mathcal{E}}^2+ o(1), \quad \text{as}\quad n\rightarrow+\infty.
\end{align}
\end{theorem}
For the proof of Theorem \ref{camillelaurentA2}, we refer the reader to Theorem 0.3 in \cite{14A2}. 

In what follows, we recall that $T_{focus}$ denotes the infimum of $t > 0$ such that there exists a couple of focus at
distance~$t$ in the sense of Definition \ref{defoffocusA2}.
\begin{remark}\label{orthoparamA2}
Note that, if
$T<\frac{T_{focus}}{2}$, then the parameters $(t_n^{(j)},h_n^{(j)},x_n^{(j)})$ and $(t_n^{(k)},h_n^{(k)},x_n^{(k)})$ are orthogonal, for $j \neq k$ in the following sense 
\begin{align*}
    \underset{n\rightarrow +\infty} {\mathrm{lim}}\Big|\mathrm{log}&(\frac{h_n^{(j)}}{h_n^{(k)}})\Big| =+\infty \quad\quad\text{or}\quad \quad \underset{n\rightarrow +\infty} {\mathrm{lim}}\; \Big(\frac{|t_n^{(j)}-t_n^{(k)}|}{h_n^{(j)}}+\frac{|x_n^{(j)}-x_n^{(k)}|}{h_n^{(j)}}\Big)=+\infty.
\end{align*}
\end{remark}

In view of the profile decomposition \eqref{prodeA2}, Proposition \ref{longtimeA2} and Theorem \ref{proposisimiltheoreA2} yield the subsequent approximation result. 
\begin{theorem}\label{scalargenA2}
Consider a bounded sequence ${(u_{0,n} , u_{1,n} )}_n$ in $\mathcal{E}$, which admits the profile decomposition \eqref{prodeA2}. Let $(U_n^{(j)})_n$ denote the nonlinear concentrating wave associated with $\underline{p}^{(j)}=[(\varphi^{(j)},\psi^{(j)}),\underline{h}^{(j)},\underline{x}^{(j)},\underline{t}^{(j)}]$. Let $T_n\in (0,+\infty)$ with $\underset{n \rightarrow+\infty}{\mathrm{lim}}T_n=T<\frac{T_{focus}}{2}$. Assume that for all $j\geq 1$, 
\begin{align}\label{5.12222A2}
    T_n<T_+(U_n^{(j)})\quad\quad\quad \text{and}\quad\quad
         \underset{n\rightarrow+\infty}{\mathrm{lim}\;\mathrm{sup} \;} \|U_n^{(j)}\|_{L^5([0,T_n],L^{10}(\mathcal{M}))} <+\infty.
\end{align}
Let $u_n$ be the solution of \eqref{eq1A2} with initial data $(u_{0,n} , u_{1,n} )$. Then, for large $n$, $u_n$ is defined on $[0,T_n]$,
\begin{align}
   \underset{n\rightarrow+\infty}{\mathrm{lim}\;\mathrm{sup}\;} \|u_n\|_{L^5([0,T_n],L^{10}(\mathcal{M}))} <+\infty,
\end{align}
and 
\begin{align*}
    u_n(t,x)=\sum\limits_{j=1}^\ell U_n^{(j)}(t,x)+w^{(\ell)}_n(t,x)+r^{(\ell)}_n(t,x),\quad \forall \;t\in [0,T_n]
\end{align*}
where
\begin{align}\label{resteA2}
    \underset{\ell\rightarrow +\infty}{\mathrm{lim}}\;\underset{n\rightarrow+\infty}{\mathrm{lim}\;\mathrm{sup}}\;\|r^{(\ell)}_n \|_{L^5([0,T_n],L^{10}(\mathcal{M}))}=0.
\end{align}
\end{theorem}
\begin{proof}[Proof of Theorem \ref{scalargenA2}]
This proof is inspired by the strategy used in the proof of Proposition 2.8 in \cite{18A2}, with careful consideration of potential geometric effects.  Define
\begin{align}\label{untildeA2}
    \tilde{u}_n^{(\ell)}(t,x):= \sum\limits_{j=1}^\ell U_n^{(j)}(t,x)+w^{(\ell)}_n(t,x).
\end{align}
By the Pythagorean expansion \eqref{pythagorean expansionA2} and the uniform boundedness of the sequence ${(u_{0,n} , u_{1,n} )}_n$ in $\mathcal{E}$, we deduce that for all $\varepsilon>0$,
\begin{align*}
 \underset{n\rightarrow+\infty}{\mathrm{lim}\;\mathrm{sup}}\;\Big( \|p^{(j)}_{n_{|t=0}}\|_{\dot{H}^1(\mathcal{M})}^2+\|\partial_tp^{(j)}_{n_{|t=0}}\|_{L^2(\mathcal{M})}^2\Big)\leq \varepsilon^2,
\end{align*}
for large $j$.
Since $(U_n^{(j)},\partial_tU_n^{(j)})_{|t=0}=(p^{(j)},\partial_tp^{(j)})_{|t=0}$, part (2) of Remark \ref{remarlemA2} implies that for $n$ and $j$ sufficiently large, $U_n^{(j)}$ is defined on $[0,T_n]$ and satisfies for large $j$
\begin{align}\label{3.10000A2}
   \underset{n\rightarrow+\infty}{\mathrm{lim}\;\mathrm{sup}}\; \|U_n^{(j)}\|_{L^5([0,T_n],L^{10}(\mathcal{M}))}\lesssim \varepsilon.
\end{align}
 From \eqref{5.12222A2} and \eqref{3.10000A2}, one gets
\begin{align} \label{3.1033333333A2}
\sum\limits_{j=1}^{+\infty}\underset{n\rightarrow+\infty}{\mathrm{lim}\;\mathrm{sup}}\;\|U_n^{(j)}\|_{L^5([0,T_n],L^{10}(\mathcal{M}))}^5<+\infty.
\end{align}
One has the inequality 
\begin{align} \label{5.18888A2}
\Big||\sum_{j=1}^\ell \alpha_j|^5-\sum_{j=1}^\ell| \alpha_j|^5\Big|\leq C_\ell\sum_{j\neq k}|\alpha_j||\alpha_k|^4.
\end{align}
Given that $\underset{n \rightarrow+\infty}{\mathrm{lim}}T_n=T<\frac{T_{focus}}{2}$, by Remark \ref{orthoparamA2}, the parameters $(t_n^{(j)},h_n^{(j)},x_n^{(j)})$ and $(t_n^{(k)},h_n^{(k)},x_n^{(k)})$ are orthogonal, for large $n$ and $j \neq k$. Thus, by \eqref{5.12222A2}, \eqref{5.18888A2} and Theorem \ref{proposisimiltheoreA2} (see Proposition 2.4 in \cite{14A2} for a proof), we deduce, for large $\ell$
\begin{align}\label{3.1011111111A2}
    \|\sum\limits_{j=1}^\ell U_n^{(j)}\|_{L^5([0,T_n],L^{10}(\mathcal{M}))}^5=\sum\limits_{j=1}^\ell\|U_n^{(j)}\|_{L^5([0,T_n],L^{10}(\mathcal{M}))}^5+ o(1), \quad \text{as}\quad n\rightarrow+\infty.
\end{align}
Combining \eqref{3.9666A2}, \eqref{untildeA2}, \eqref{3.1033333333A2} and \eqref{3.1011111111A2}, we get
\begin{align}\label{tildeunestimationA2}
  \underset{\ell\rightarrow +\infty}{\mathrm{lim}}\;\underset{n\rightarrow+\infty}{\mathrm{lim}\;\mathrm{sup}}\; \|\tilde{u}_n^{(\ell)}\|_{L^5([0,T_n],L^{10}(\mathcal{M}))}<+\infty.
\end{align}
Set
\begin{align}
 r^{(\ell)}_n:=
 u_n - \tilde{u}_n^{(\ell)}\quad\text{and} \quad e^{(\ell)}:=(\partial_t^2-\Delta_g)\tilde{u}_n^{(\ell)}-(\tilde{u}_n^{(\ell)}-w^{(\ell)}_n)^5.
\end{align}
Due to \eqref{untildeA2},
\begin{align}   e^{(\ell)}=\sum\limits_{j=1}^\ell f\Big(U_n^{(j)}\Big)-f\Big(\sum\limits_{j=1}^\ell U_n^{(j)}\Big),
\end{align}
where $f(s)=s^5$. By Proposition 2.4 in \cite{14A2}, one obtains
\begin{align}\label{3.10666eeeA2}
  \underset{\ell\rightarrow +\infty}{\mathrm{lim}}\;\underset{n\rightarrow+\infty}{\mathrm{lim}\;\mathrm{sup}}\; \|e^{(\ell)}\|_{L^1([0,T_n],L^2(\mathcal{M}))}=0.
\end{align}
Collecting \eqref{tildeunestimationA2} and \eqref{3.10666eeeA2}, and noting that $\tilde{u}_n^{(\ell)}(0,x)=u_n(0,x)$ and $\partial_t\tilde{u}_n^{(\ell)}(0,x)=\partial_tu_n(0,x)$, Proposition \ref{longtimeA2} yields
\begin{align}
    \underset{\ell\rightarrow +\infty}{\mathrm{lim}}\;\underset{n\rightarrow+\infty}{\mathrm{lim}\;\mathrm{sup}}\;\|r^{(\ell)}_n \|_{L^5([0,T_n],L^{10}(\mathcal{M}))}=0.
\end{align}
This concludes the proof of Theorem \ref{scalargenA2}.
\end{proof}
\subsection{Bound from below of the limit inferior} \label{sectionliminfA2}
Here, we use the notation in Remark \ref{remarkofmaintheA2}. We assume that $x_\infty\in S$. Consider an increasing sequence $(\tilde{t}_n)_n $ such that $0< \tilde{t}_n< 1$ and $\tilde{t}_n \underset{n\rightarrow +\infty}{\longrightarrow} 1$. Let $\tilde{\varphi} \in C^\infty_c(\mathcal{M})$ such that $\mathrm{supp}(\tilde{\varphi})\cap S=\{x_\infty\}$ and $\tilde{\varphi}=1$ on  $\mathscr{B}_{r_0}(x_\infty)$ for some $r_0>0$. 
From \eqref{3.1**A2}, $(u(\tilde{t}_n),\partial_tu(\tilde{t}_n))$ converges weakly to $(v_0,v_1)$ in $\mathcal{E}$ as $n$ goes to infinity and by continuity of $v$ in time $(v(\tilde{t}_n),\partial_tv(\tilde{t}_n))$ converges to $(v(1),\partial_tv(1))=(v_0,v_1)$.\\
Define 
\begin{equation} \label{firstsequnceA2}
\left\{
\begin{array}{l}
u_{0,n}=\tilde{\varphi} \big(u(\tilde{t}_n)-v(\tilde{t}_n)\big),\\[10pt]
u_{1,n}=\tilde{\varphi} \big(\partial_tu(\tilde{t}_n)-\partial_tv(\tilde{t}_n)\big).
\end{array}
\right.
\end{equation}
Then, applying Theorem \ref{camillelaurentA2} to the bounded sequence ${(u_{0,n} , u_{1,n} )}_n$ in $\mathcal{E}$, we obtain, up to extraction of a subsequence, the profile decomposition \eqref{prodeA2}.\\
Due to \eqref{3.2**A2} and the finite speed of propagation, one has 
\begin{align}\label{5.26A2}
\underset{t\rightarrow 1}{\mathrm{lim}} \int_{\varepsilon >|x-x_\infty|>|t-1|} |\nabla(u(t)-v_0)|^2+|\partial_tu(t)-v_1|^2=0,
\end{align}
which implies by Lemma 2.5 in \cite{18A2},
\begin{align}\label{5.277A2}
 \forall j\geq 1, \quad \underset{n\rightarrow+\infty}{\mathrm{lim}}x_n^{(j)}=x_\infty.
 \end{align}
Furthermore,
Theorems \ref{proposisimiltheoreA2} and \ref{scalargenA2} yield the following lemma, which subsequently allows us to establish \eqref{1.26***A2}.
\begin{lemma}\label{lemmmmA2}
We reorder \eqref{prodeA2} so that
\begin{align} \label{orderA2}
\|(\varphi^{(1)},\psi^{(1)})\|_{\mathcal{E}_\infty}^2=\underset{j\geq 1}{\mathrm{sup}}\|(\varphi^{(j)},\psi^{(j)})\|_{\mathcal{E}_\infty}^2.
\end{align}
Then, one has
\begin{align}\label{3.18**A2}
\|(\varphi^{(1)},\psi^{(1)})\|_{\mathcal{E}_\infty}^2\geq \frac{2}{3} \|\nabla W\|_{L^2(\mathbb{R}^3)}^2.
\end{align}
\end{lemma}
\begin{proof}[Proof]
We proceed by contradiction. Assume that
\begin{align*}
\|(\varphi^{(1)},\psi^{(1)})\|_{\mathcal{E}_\infty}^2 < \frac{2}{3} \|\nabla W\|_{L^2(\mathbb{R}^3)}^2.
\end{align*}
From \eqref{orderA2}, one gets
\begin{align*}
   \|(\varphi^{(j)},\psi^{(j)})\|_{\mathcal{E}_\infty}^2 < \frac{2}{3} \|\nabla W\|_{L^2(\mathbb{R}^3)}^2, \quad \text{for all} \quad j\geq 1.
\end{align*}
Then, by (c) of Remark \ref{choicesoffiA2} and Theorem \ref{proposisimiltheoreA2}, for any $j\geq 1$, the nonlinear concentrating wave $U_n^{(j)}$  associated with $\underline{p}^{(j)}=[(\varphi^{(j)},\psi^{(j)}),\underline{h}^{(j)},\underline{x}^{(j)},\underline{t}^{(j)}]$ is defined on $[0,T_n]$ with $T_n=\frac{T_{focus}}{4}$ and 
\begin{align*}
   \underset{n\rightarrow +\infty}{\mathrm{lim}\;\mathrm{sup}}\; \|U_n^{(j)}\|_{L^5([0,T_n], L^{10}(\mathcal{M}))}<+\infty.
\end{align*}
Furthermore, due to Theorem \ref{scalargenA2}, the solution $\tilde{u}_n$ of \eqref{eq1A2} with initial condition $\big(\tilde{\varphi} u(\tilde{t}_n), \tilde{\varphi} \partial_tu(\tilde{t}_n)\big)$ satisfies 
\begin{align} \label{3.1100000000aA2}
     \underset{n\rightarrow+\infty}{\mathrm{lim}\;\mathrm{sup}}\;\|\tilde{u}_n\|_{L^5([0,T_n],L^{10}(\mathcal{M}))} <+\infty.
\end{align}
Moreover, by finite speed of propagation, one has
\begin{align}\label{3.111111111aA2}
    \tilde{u}_n(t,x)=u(\tilde{t}_n+t,x)  \quad \text{for almost all} \quad x \in \mathscr{B}_{r_0-|t|}(x_\infty).
\end{align}
 Combining \eqref{3.1100000000aA2} and \eqref{3.111111111aA2} gives a contradiction with the fact that $1$ is the maximal time of existence of $u$ and that $x_\infty$ is a singular point.
\end{proof}
From \eqref{pythagorean expansionA2}, using the change of variables $y=\frac{x-x_n^{(1)}}{h_n^{(1)}}$ along with  $\underset{n \rightarrow +\infty}{\mathrm{lim}} h_n^{(1)}= 0$ and $\underset{n \rightarrow +\infty}{\mathrm{lim}} x_n^{(1)}=x_\infty$, we conclude that
\begin{align*}
&\underset{n\rightarrow+\infty}{\mathrm{lim\;\mathrm{inf}}}\; \|(u_{0,n} , u_{1,n} )\|_{\mathcal{E}}^2
\\&\geq \underset{n\rightarrow+\infty}{\mathrm{lim\;\mathrm{inf}}}\;\|(p_n^{(1)},\partial_tp_n^{(1)})_{t=0}\|_{\mathcal{E}}^2
=\underset{n\rightarrow+\infty}{\mathrm{lim\;\mathrm{inf}}}\;\Big\|\frac{1}{\sqrt{h_n^{(1)}}}\big(\varphi^{(1)},\frac{1}{h_n^{(1)}}\psi^{(1)}\big)(\frac{.-x_n^{(1)}}{h_n^{(1)}})\Big\|_{\mathcal{E}}^2
\\& \quad=
\underset{n\rightarrow+\infty}{\mathrm{lim\;\mathrm{inf}}}\;\Big(\int_{\mathbb{R}^3}\sum\limits_{1\leq i,k\leq3} g^{i,k}(h_n^{(1)}y+x_n^{(1)}) \frac{\partial \varphi^{(1)}}{\partial y_i}(y) \frac{\partial \varphi^{(1)}}{\partial y_k}(y)\sqrt{\mathrm{det}(g(h_n^{(1)}y+x_n^{(1)}))}dy
\\&\quad
\hspace{1.8cm}+ \int_{\mathbb{R}^3} |\psi^{(1)}(y)|^2 \sqrt{\mathrm{det}(g(h_n^{(1)}y+x_n^{(1)}))}dy \Big)
\\&\quad
=\|(\varphi^{(1)},\psi^{(1)})\|_{\mathcal{E}_\infty}^2\geq \frac{2}{3} \|\nabla W\|_{L^2(\mathbb{R}^3)}^2,
\end{align*}
where in the last line we use Lemma \ref{lemmmmA2}.
Thus, according to \eqref{supportA2}, we deduce \eqref{1.26***A2}.
\subsection{Bound from below of the limit superior}
It remains to prove \eqref{1.25***A2} of our main theorem. We begin by recalling the following result.
\begin{theorem} \label{proposition3.177777A2}
Let $\varepsilon_0>0$. There exists $C_{\varepsilon_0}>0$ with the following property.
Let $u$ be a solution to 
\begin{equation}
\left\{
\begin{array}{l}
 (\partial_t^2-\Delta_{g(x_\infty)})u=u^5, \quad \text{on} \quad \mathbb{R}\times \mathbb{R}^3\\[10pt]
(u,\partial_tu)_{|t=0}=(u_0,u_1)\in \mathcal{E}_\infty.
\end{array}
\right.
\end{equation}
Let $I\subset{I_{\mathrm{max}}}$. Assume
    \begin{align*}
        \|( u(t),\partial_tu(t))\|_{\mathcal{E}_\infty}^2\leq \|\nabla W\|_{L^2(\mathbb{R}^3)}^2-\varepsilon_0,
    \end{align*}
    for all $t\in I$.  
    Then, 
    \begin{align*}
    \|u\|_{L^5(I,L^{10}(\mathbb{R}^3))}\leq C_{\varepsilon_0}.
    \end{align*}
    \end{theorem}
Theorem \ref{proposition3.177777A2} follows by applying the change of variables $\tilde{u}(t,x) = u(t, Bx)$, where $B$ is the matrix introduced in the proof of Proposition \ref{PropoScatteringA2}, and then using Corollary 4.14 in \cite{1888A2} for $\tilde{u}$.
  
In what follows, we adopt the notation introduced in Remark~\ref{remarkofmaintheA2} and in Subsection~\ref{sectionliminfA2}. The following proposition together with \eqref{supportA2} implies \eqref{1.25***A2}.
\begin{proposition} \label{proposition3.18A2}
One has  
    \begin{align} \label{5.30000A2}
       \underset{t\rightarrow 1}{\mathrm{lim}\;\mathrm{sup}}\;\Big\| \tilde{\varphi}\Big(u(t)-v(t)\Big), \tilde{\varphi}\Big(\partial_tu(t)-\partial_tv(t)\Big)\Big\|^2_{\mathcal{E}}\geq \|\nabla W\|_{L^2(\mathbb{R}^3)}^2.
    \end{align} 
\end{proposition}
\begin{proof}[Proof]  The proof follows the general strategy of Proposition 3.9 in \cite{18A2} and Corollary 7.5 in \cite{9A2}, with suitable modifications. We argue by contradiction and assume that \eqref{5.30000A2} does not hold. Then, 
\begin{align}\label{5.322222A2}
 \underset{t \rightarrow 1}{\mathrm{lim}\;\mathrm{sup}}\;\Big\| \tilde{\varphi}\Big(u(t)-v(t)\Big), \tilde{\varphi}\Big(\partial_tu(t)-\partial_tv(t)\Big)\Big\|^2_{\mathcal{E}}< \|\nabla W\|_{L^2(\mathbb{R}^3)}^2.
\end{align}
Consider an increasing sequence $(\tau_n)_n $ such that $0< \tau_n< 1$ and $\tau_n \underset{n\rightarrow +\infty}{\longrightarrow} 1$. Define 
\begin{equation}
\left\{
\begin{array}{l}
\tilde{u}_{0,n}=\tilde{\varphi} \big(u(\tau_n)-v(\tau_n)\big),\\[10pt]
\tilde{u}_{1,n}=\tilde{\varphi} \big(\partial_tu(\tau_n)-\partial_tv(\tau_n)\big).
\end{array}
\right.
\end{equation}
By Theorem \ref{camillelaurentA2}, the sequence ${(\tilde{u}_{0,n} , \tilde{u}_{1,n} )}_n$ admits a profile decomposition as in \eqref{prodeA2}. 
Recall that $U_n^{(j)}$ is the nonlinear concentrating wave associated with the profile $p_n^{(j)}$ in \eqref{prodeA2}. We denote by $V_n^{(j)}$ its associated rescaled function, so that
\begin{align}\label{thenonlinearconA2}
      U_n^{(j)}(t,x):=\frac{1}{\sqrt{h_n^{(j)}}} V_n^{(j)}\Big(\frac{t-t_n^{(j)}}{h_n^{(j)}}, \frac{x-x_n^{(j)}}{h_n^{(j)}}\Big).
\end{align}
Applying the change of variables $s=\frac{t-t_n^{(j)}}{h_n^{(j)}}$ and $y=\frac{x-x_n^{(j)}}{h_n^{(j)}}$ to \eqref{thenonlinearconA2}, we deduce that the profile $V_n^{(j)}$ satisfies
 \begin{align*}  
 (\partial_s^2-\Delta_{g(h_n^{(j)}.+x_n^{(j)})})V_n^{(j)}=(V_n^{(j)})^5.
 \end{align*}
Using the Pythagorean expansion \eqref{pythagorean expansionA2}, \eqref{5.322222A2}, and the fact that $(U_n^{(j)},\partial_tU_n^{(j)})_{|t=0}=(p^{(j)},\partial_tp^{(j)})_{|t=0}$, for all $j\geq 1$, we deduce
\begin{align} \label{5.355A2}
 \underset{n \rightarrow +\infty}{\mathrm{lim \;sup}}\; \|(U_n^{(j)},\partial_tU_n^{(j)})_{|t=0}\|_{\mathcal{E}}^2\leq \underset{n \rightarrow +\infty}{\mathrm{lim \;sup}}\; \|(\tilde{u}_{0,n} , \tilde{u}_{1,n} )\|_{\mathcal{E}}^2 < \|\nabla W\|_{L^2(\mathbb{R}^3)}^2-\frac{\varepsilon_0}{2},
\end{align}
for some $\varepsilon_0>0$.
Furthermore, by Theorem \ref{scalargenA2}, there exists at least one profile $U_n^{(j)}$ that does not scatter forward in time; otherwise, $T_+(u)>1$, contradicting the assumption that $T_+(u)=1$. Hence, there exists $J_0\geq1$ such that
\begin{align}\label{3.111555555A2}
   \forall 1\leq j\leq J_0,  \quad \|U_n^{(j)}&\|_{L^5([0,T_+(U_n^{(j)})),L^{10}(\mathcal{M}))}=+\infty 
\end{align}
and
\begin{align} \label{5.47777A2}
  \forall j\geq J_0+1, \quad \|U_n^{(j)}\|_{L^5([0,T_+(U_n^{(j)})),L^{10}(\mathcal{M}))}< +\infty.
\end{align}
Note that, by the blow-up criterion \eqref{blowupcriterionA2}, \eqref{5.47777A2} implies
\begin{align}\label{tempsmaximalforjgrandA2}
    \forall j\geq J_0+1, \quad T_+(U_n^{(j)})=+\infty.
\end{align}
Moreover, since $U_n^{(j)}$ does not scatter forward in time for all $1\leq j\leq J_0$, the case
\begin{align*}
  \underset{n\rightarrow+\infty}{\mathrm{lim}}  \frac{-t_n^{(j)}}{h_n^{(j)}}=+\infty,
\end{align*}
is excluded. Therefore, 
\begin{align}
 \forall   1\leq j\leq J_0, \quad t_n^{(j)}=0 \quad \text{or}\quad  \underset{n\rightarrow+\infty}{\mathrm{lim}}\frac{-t_n^{(j)}}{h_n^{(j)}}=-\infty.
 \end{align}
From \eqref{5.355A2}, \eqref{3.111555555A2}, and Theorem \ref{proposition3.177777A2}, it follows that for each $ 1\leq j\leq J_0$ there exists $T_{n,j}$ such that, for large $n$,
 \begin{align} \label{3.119999999A2}
     &-\frac{t_n^{(j)}}{h_n^{(j)}}<T_{n,j}<T_+(V_n^{(j)}), \quad \big\|\big( V_n^{(j)}(T_{n,j}),\partial_sV_n^{(j)}(T_{n,j})\big)\big\|_{\mathcal{E}_\infty}^2= \|\nabla W\|_{L^2(\mathbb{R}^3)}^2-\frac{\varepsilon_0}{2}\nonumber
   \\& \text{and}\quad \forall  s\in [-\frac{t_n^{(j)}}{h_n^{(j)}},T_{n,j}[,\quad  \big\|\big( V_n^{(j)}(s),\partial_sV_n^{(j)}(s)\big)\big\|_{\mathcal{E}_\infty}^2< \|\nabla W\|_{L^2(\mathbb{R}^3)}^2-\frac{\varepsilon_0}{2}.
 \end{align}
Reordering the profiles $(V_n^{(j)})_j$ so that
 \begin{align}\label{thexpressionofTnA2}
   S_n:=  t_n^{(1)}+h_n^{(1)}T_{n,1}=\underset{1\leq j\leq J_0}{\mathrm{min}}(t_n^{(j)}+h_n^{(j)}T_{n,j}),
 \end{align}
we obtain from the definition of $T_{n,j}$ that $S_n> 0$, for large $n$.
From \eqref{tempsmaximalforjgrandA2} and \eqref{thexpressionofTnA2}, we deduce 
\begin{align}\label{3.1222222222A2}
\forall j\geq 1,\quad \frac{S_n-t_n^{(j)}}{h_n^{(j)}}<T_+(V_n^{(j)}),\\
 \forall   1\leq j\leq J_0, \quad \frac{S_n-t_n^{(j)}}{h_n^{(j)}}\leq T_{n,j}.
\end{align}
Thus, by \eqref{3.119999999A2} and Theorem \ref{proposition3.177777A2}, we obtain
\begin{align} \label{5.45555A2}
    \underset{n\rightarrow+\infty}{\mathrm{lim}\;\mathrm{sup} \;} \|U_n^{(j)}\|_{L^5([0,S_n],L^{10}(\mathcal{M}))} <+\infty \quad \text{for all }\quad j\geq 1. 
\end{align}
Therefore, we split the interval $[0,S_n]$ into finitely many subintervals on which the non-focusing property of Theorem \ref{proposisimiltheoreA2} holds. Using \eqref{5.45555A2}, we then apply Theorem \ref{scalargenA2} successively on each subinterval. As a consequence, we deduce that the solution $\tilde{\varphi}\Big(u(t+\tau_n)-v(t+\tau_n)\Big)$ to \eqref{eq1A2}, with initial data $(\tilde{u}_{0,n} , \tilde{u}_{1,n})$, is defined, for large $n$, on $[0,S_n]$,
\begin{align}\label{pourSnA2}
   \underset{n\rightarrow+\infty}{\mathrm{lim}\;\mathrm{sup}\;} \Big\|\tilde{\varphi}\Big(u(t+\tau_n)-v(t+\tau_n)\Big)\Big\|_{L^5([0,S_n],L^{10}(\mathcal{M}))} <+\infty,
\end{align}
 and 
\begin{align}\label{profilenonlinedecompA2}
\tilde{\varphi}\Big(u(t+\tau_n)-v(t+\tau_n)\Big)=  \sum\limits_{j=1}^\ell&\frac{1}{\sqrt{h_n^{(j)}}} V_n^{(j)}\Big(\frac{t-t_n^{(j)}}{h_n^{(j)}}, \frac{x-x_n^{(j)}}{h_n^{(j)}}\Big)\nonumber\\&+w^{(\ell)}_n(t,x)+r^{(\ell)}_n(t,x), \quad \text{for all}\quad t\in [0,S_n],
\end{align}
where $r^{(\ell)}_n$ verifies \eqref{resteA2}. 
 Setting
\begin{align*}
    \tilde{t}_n:=\tau_n+ S_n,
\end{align*}
one has from \eqref{pourSnA2} and the fact that $T_+(u)=1$,
\begin{align*}
\tilde{t}_n<1.
\end{align*}
Moreover, since $S_n>0$ for large $n$ and $\tau_n$ tends to $1$ as $n$ goes to infinity, we deduce
\begin{align*}
\tilde{t}_n \underset{n\rightarrow +\infty}{\longrightarrow} 1.
\end{align*}
 Taking $t=S_n$ in \eqref{profilenonlinedecompA2}, we get
\begin{equation}
\left\{
\begin{array}{l}
u_{0,n}=\sum\limits_{j=1}^\ell\frac{1}{\sqrt{h_n^{(j)}}} V_n^{(j)}\Big(\frac{S_n-t_n^{(j)}}{h_n^{(j)}}, \frac{x-x_n^{(j)}}{h_n^{(j)}}\Big)+w^{(\ell)}_n(S_n,x)+r^{(\ell)}_n(S_n,x),\\[10pt]
u_{1,n}=\sum\limits_{j=1}^\ell(\frac{1}{h_n^{(j)}} )^\frac{3}{2}\partial_tV_n^{(j)}\Big(\frac{S_n-t_n^{(j)}}{h_n^{(j)}}, \frac{x-x_n^{(j)}}{h_n^{(j)}}\Big)+\partial_tw^{(\ell)}_n(S_n,x)+\partial_tr^{(\ell)}_n(S_n,x),
\end{array}
\right.
\end{equation}
where $(u_{0,n},u_{1,n})$ are defined as in \eqref{firstsequnceA2}.
 Recalling that $T_{n,1}:=\frac{S_n-t_n^{(1)}}{h_n^{(1)}}$ , using $\underset{n\rightarrow+\infty}{\mathrm{lim}}x_n^{(1)}=x_\infty$(see \eqref{5.277A2}), $\underset{n \rightarrow +\infty}{\mathrm{lim}} h_n^{(1)}= 0$ and  \eqref{3.119999999A2}, we deduce
 \begin{align*}
     \underset{n\rightarrow+\infty}{\mathrm{lim}\;\mathrm{sup}}\;\| (u_{0,n}, u_{1,n})\|^2_{\mathcal{E}}&\geq \underset{n\rightarrow+\infty}{\mathrm{lim}\;\mathrm{sup}} \;\Big(\Big\|\frac{1}{\sqrt{h_n^{(1)}}} V_n^{(1)}(T_{n,1}, \frac{.-x_n^{(1)}}{h_n^{(1)}})\Big\|_{\dot{H}^1(\mathcal{M})}^2
     \\&\quad\quad\hspace{1.5cm}+\Big\|(\frac{1}{h_n^{(1)}})^{\frac{3}{2}} \partial_tV_n^{(1)}(T_{n,1}, \frac{.-x_n^{(1)}}{h_n^{(1)}})\Big\|_{L^2(\mathcal{M})}^2\Big)
    \\&=\underset{n\rightarrow+\infty}{\mathrm{lim}\;\mathrm{sup}} \;\|\big( V_n^{(1)}(T_{n,1}),\partial_tV_n^{(1)}(T_{n,1})\big)\|_{\mathcal{E}_\infty}^2 = \|\nabla W\|_{L^2(\mathbb{R}^3)}^2-\frac{\varepsilon_0}{2},
 \end{align*}
which contradicts \eqref{5.355A2}. This completes the proof of Proposition \ref{proposition3.18A2}.
\end{proof}
\section*{\large \textbf{Acknowledgements.}}  
The author sincerely thanks Thomas Duyckaerts for his guidance and valuable comments.
\newpage
\newpage
\begin{appendix}
\section{Pseudo-differential calculus}\label{appenAA2}
In this appendix, we recall some basic facts from pseudo-differential calculus on $ \mathbb{R}^3$. For the proofs, we refer the reader to Taylor \cite{19A2}.
\begin{definition}
For  $m\in \mathbb{R}$, the space $S^m(\mathbb{R}^3\times \mathbb{R}^3)$ is the set of $a\in C^\infty(\mathbb{R}^3\times \mathbb{R}^3)$ such that, for all $\alpha, \beta \in \mathbb{N}^3$,
\begin{align*}
|\partial_x^\alpha\partial_\xi^\beta a(x,\xi)|\leq C_{\alpha,\beta}(1+|\xi|)^{m-|\beta|}.
\end{align*}
An element $a\in S^m(\mathbb{R}^3\times \mathbb{R}^3)$ is called a symbol of order $m$.
\end{definition}
Next, we recall the definition of pseudo-differential operators acting on functions in the Schwartz space $\mathcal{S}(\mathbb{R}^3)$.
\begin{definition}
Given $u\in \mathcal{S}(\mathbb{R}^3)$ and $a \in S^m(\mathbb{R}^3\times \mathbb{R}^3)$, the pseudo-differential operator associated with $a$ is defined by 
\begin{align*}
    a(x,D)u(x)=\frac{1}{(2\pi)^3}\int_{\mathbb{R}^3}\int_{\mathbb{R}^3} e^{i(x-y)\xi} a(x,\xi) u(y) dyd\xi.
\end{align*}
We will also denote this operator by
\begin{align*}
\Op(a):=a(x,D).
\end{align*}
\end{definition}
The following theorem describes the composition of pseudo-differential operators.
\begin{theorem}\label{ThoremA3A2}
Let $m_1, m_2\in \mathbb{R}$ and $a_1 \in S^{m_1}(\mathbb{R}^3\times \mathbb{R}^3)$, $a_2\in S^{m_2}(\mathbb{R}^3\times \mathbb{R}^3)$. Then, $$\Op(a_1)\Op(a_2)=\Op(a_1\# a_2),$$ where $a_1\# a_2 \in S^{m_1+m_2}(\mathbb{R}^3\times \mathbb{R}^3)$ and, for any $N\in \mathbb{N}$, admits the asymptotic expansion 
\begin{align*}
a_1\# a_2= \sum_{|\alpha|<N} \frac{1}{\alpha !} D_\xi^\alpha a_1 \partial_x ^\alpha a_2 +r_N(a_1,a_2),
\end{align*}
with $r_N(a_1,a_2) \in S^{m_1+m_2-N}(\mathbb{R}^3\times \mathbb{R}^3)$.
\end{theorem}
We also have the following result concerning the adjoint of a pseudo-differential operator.
\begin{theorem}\label{ThoremA4A2}
Let $ m\in \mathbb{R}$ and  $a\in S^{m}(\mathbb{R}^3\times \mathbb{R}^3)$. Then, $$\Op(a)^*=\Op(a^*),$$ where $a^* \in S^{m}(\mathbb{R}^3\times \mathbb{R}^3)$ and
for all $N\in \mathbb{N}$
\begin{align*}
a^*= \sum_{|\alpha|<N} \frac{1}{\alpha !} D_\xi^\alpha  \partial_x ^\alpha \bar{a} +r_N(a),
\end{align*} 
with $r_N(a) \in S^{m-N}(\mathbb{R}^3\times \mathbb{R}^3)$.
\end{theorem}
\end{appendix}
\renewcommand{\theequation}{A.\arabic{equation}}
\setcounter{equation}{0}

\newpage
\let\oldaddcontentsline\addcontentsline
\renewcommand{\addcontentsline}[3]{}


\begin{thebibliography}{99}


\bibitem[1]{3A2} T. Aubin, Équations différentielles non linéaires et problème de Yamabe concernant la courbure scalaire. J. Math. Pures Appl., 55, 269–296, (1976).

\bibitem[2]{13A2} N. Burq and P. Gérard, Contrôle optimal des équations aux dérivées partielles. École Polytechnique, Département de mathématiques, (2002). Available at: \url{https://www.imo.universite-paris-saclay.fr/~nicolas.burq/articles/coursX.pdf}.

\bibitem[3]{ben} S. Ben Said, Global well-posedness for the focusing energy-critical wave
equation in non-euclidean geometries. Submitted for publication (2026).
\bibitem[4]{ben1} S. Ben Said. Observabilité et étude des propriétés fines des solutions des équations des ondes. PhD thesis, Sorbonne Paris Nord University, (2026). \url{https://theses.fr/s385246}. 
\bibitem[5]{18A2} T. Duyckaerts, C. Kenig and F. Merle, Universality of blow-up profile for small radial type II blow-up solutions of the energy-
critical wave equation, J Eur Math Soc (JEMS), 13, 533-599, (2011)·
\bibitem[6]{1888A2} T. Duyckaerts, C. Kenig and F. Merle, Universality of the blow-up profile for small type II blow-up solutions of the energy-
critical wave equation: the nonradial case, J. Eur. Math. Soc., 14, 5, 1389-1454, (2012).
\bibitem[7]{15A2}T. Duyckaerts and F. Merle, Dynamic of threshold solutions for energy-critical NLS. Geom. Funct. Anal. 18, no. 6, 1787–1840, (2009).





\bibitem[8]{17A2}D. Fang, J. Xie and T. Cazenave, Scattering for the focusing energy-subcritical nonlinear Schrödinger equation, Sci. China Math., 54, 10, 2037-2062, (2011)· 





\bibitem[9]{6A2}P. Gérard, Description du défaut de compacité de l’injection de sobolev. ESAIM Control Optim. Calc. Var., 3:213–233, (1998).
\bibitem[10]{188A2} P. Gérard, Microlocal defect mesures, Comm. Partial diff. Eq. , 16, pp. 1761-
1794, (1991).


\bibitem[11]{10A2}S. Ibrahim, Geometric-optics for nonlinear concentrating waves in focusing and non-focusing two geometries. Commun. Contemp. Math. 6 (1), 1–23, (2004).

\bibitem[12]{12A2} S. Ibrahim and M. Majdoub, Comparaison des ondes linéaires et non linéaires à coefficients variables, Bull. Soc. Math. Belg.,
10, 299-312, (2003)·

\bibitem[13]{2A2} E. Y. Jaffe, Wave equation on manifolds and finite speed of propagation. Available at: \url{https://r-grande.github.io/Expository/Wave%20Equation%20on%20manifolds%20and%20finite%20speed%20of%20propagation.pdf}.





\bibitem[14]{7A2} L. V. Kapitanskii, Some generalizations of the Strichartz-Brenner inequality. Leningrad Math. J., 1(10):693–726, (1990).
\bibitem[15]{88A2} L. V. Kapitanskii. The Cauchy problem for the semilinear wave equation. III. Zap. Nauchn. Sem. Leningrad.
Otdel. Mat. Inst. Steklov. (LOMI), 181(Differentsial’naya Geom. Gruppy Li i Mekh. 11):24–64, 186, (1990).
\bibitem[16]{9A2} C. E. Kenig and F. Merle, Global well-posedness, scattering and blow-up for the
energy-critical focusing non-linear wave equation. Acta Math. 201, 147–212, (2008).
\bibitem[17]{8A2} J. Krieger, W. Schlag, and D. Tataru, Slow blow-up solutions for the $H^1
(\mathbb{R}^3)$ critical focusing semilinear wave equation. Duke Math. J. 147, 1, 1–53, (2009).






 
\bibitem[18]{1666A2} D. Lafontaine and C. Laurent, On scattering and profile decomposition for critical nonlinear waves outside weakly trapping obstacles. arXiv:2604.15947, (2026).

\bibitem[19]{14A2}C. Laurent, On stabilization and control for the critical Klein-Gordon equation on a 3-D compact manifold, J. Funct. Anal.,
260, 5, 1304-1368, (2011).



\bibitem[20]{4A2} G. Talenti, Best constant in Sobolev inequality. Ann. Mat. Pura Appl., 110,
353–372, (1976).
\bibitem[21]{19A2}M. Taylor, Pseudodifferential Operators, Princeton Univ. Press, Princeton
NJ, (1981).










\end{thebibliography}
\end{document}